\documentclass[11pt]{amsart}
\usepackage{amstext}
\usepackage{amsthm,euscript}
\usepackage{amscd}
\usepackage{amssymb}
\usepackage{amsmath}
\usepackage{mathrsfs}
\usepackage{mathtools}
\usepackage[all]{xy}
\usepackage{paralist}
\usepackage{enumitem}
\usepackage{moreenum}
\usepackage{scalerel,stackengine}
\usepackage[colorlinks=true, pdfstartview=FitV, linkcolor=blue,
citecolor=blue,urlcolor=blue,breaklinks=true]{hyperref}
\usepackage{mathdots}

\swapnumbers 

\newcommand{\lv}[1]{} 
\newcommand{\comments}[1]{}
\newcommand{\pcomments}[1]{}
\newcommand{\inparcom}[1]{}
\newcounter{question} 
\newcommand{\quest}[1]{} 
\newcommand{\new}{}
\newcommand{\enew}{}

\newcommand{\arxiv}[1]{\href{http://arxiv.org/abs/#1}{\tt arXiv:\nolinkurl{#1}}}
\newtheorem{thm}[subsection]{Theorem}

\newtheorem{prop}[subsection]{Proposition}

\newtheorem{lem}[subsection]{Lemma}

\newtheorem{cor}[subsection]{Corollary}

\newtheorem{principle}[subsection]{Cancellation Principle}

\newtheorem{appli}[subsection]{Application}

\theoremstyle{definition}

\newtheorem{remark}[subsection]{Remark}

\newtheorem{remarks}[subsection]{Remarks}

\newtheorem{example}[subsection]{Example}

\newtheorem{examples}[subsection]{Examples}

\numberwithin{equation}{subsection}

\newcommand\sm{\smallskip}
\newcommand\ms{\medskip}

\def\limind{\mathop{\oalign{lim\cr\hidewidth$\longrightarrow$\hidewidth \cr}}}

\newcommand{\ol}{\overline} 

\newcommand{\simlgr}{\buildrel \sim \over \longrightarrow}

\newcommand{\simla}{\buildrel \sim \over \longleftarrow}

\newcommand{\wdh}{\widehat}

\newcommand{\et}{_\mathrm{\acute{e}t}}

\newcommand{\longto}{\longrightarrow}

\newcommand{\ad}{^{\, {\rm ad}}}

\newcommand{\ti}{^\times}
\newcommand{\lan}{\langle}\newcommand{\ran}{\rangle}
\def\co{\colon}
\def\ot{\otimes} 
\def\op{^{\rm op}}
\def\ideal{\triangleleft}
\newcommand{\me}{^{-1}}
\def\dar[#1]{\ar@<2pt>[#1]\ar@<-2pt>[#1]}
\def\wtl{\widetilde}

\newcommand{\an}{_\mathrm{an}} 

\stackMath
\newcommand\reallywidehat[1]{%
\savestack{\tmpbox}{\stretchto{%
  \scaleto{%
    \scalerel*[\widthof{\ensuremath{#1}}]{\kern.1pt\mathchar"0362\kern.1pt}%
    {\rule{0ex}{\textheight}}
  }{\textheight}%
}{2.4ex}}%
\stackon[-6.9pt]{#1}{\tmpbox}%
}

\newcommand{\rmA}{\mathrm{A}}

\newcommand{\Aut}{\operatorname{Aut}}

\newcommand{\Br}{\operatorname{Br}} 

\newcommand{\rmC}{\mathrm{C}}

\newcommand{\Char}{\operatorname{char}} 

\newcommand{\Cent}{\mathrm{Cent}} 
\newcommand{\rmD}{\mathrm{D}}
\newcommand{\diag}{\operatorname{diag}}

\newcommand{\Di}{{\operatorname{Dick}}}
\newcommand{\Dyn}{\mathrm{Dyn}} 

\newcommand{\End}{\operatorname{End}}

\newcommand{\rmG}{\mathrm{G}}

\newcommand{\GL}{{\operatorname{GL}}}

\newcommand{\Gen}{\operatorname{Gen}} 

\newcommand{\rmH}{\mathrm{H}}
\newcommand{\Hom}{\operatorname{Hom}}

\newcommand{\hyp}{\mathrm{hyp}}

\newcommand{\Id}{\operatorname{Id}}

\newcommand{\ind}{\operatorname{ind}}

\newcommand{\inc}{{\operatorname{inc}}}
\newcommand{\Isom}{\operatorname{Isom}}

\newcommand{\idemp}{\operatorname{idemp}}

\newcommand{\Jac}{{\operatorname{Jac}}}

\newcommand{\Ker}{\operatorname{Ker}}

\newcommand{\Lie}{\operatorname{Lie}}

\newcommand{\Mat}{{\operatorname{M}}}

\newcommand{\Norm}{\operatorname{Norm}} 

\newcommand{\Orth}{{\operatorname{O}}}\newcommand{\orth}{\Orth}

\newcommand{\rmP}{\operatorname{P}}

\newcommand{\pr}{\operatorname{pr}}

\newcommand{\Pic}{{\operatorname{Pic}}}

\newcommand{\Par}{{\operatorname{Par}}} 
\newcommand{\per}{{\operatorname{per}}} 

\newcommand{\rank}{\operatorname{rank}}

\newcommand{\rad}{\operatorname{rad}}
\newcommand{\reg}{{{\operatorname{reg}}}} 

\newcommand{\Spec}{\operatorname{Spec}}

\newcommand{\SL}{{\operatorname{SL}}}

\newcommand{\type}{\operatorname{type}}

\newcommand{\W}{\operatorname{W}}

\newcommand{\rmX}{\operatorname{X}}

\DeclareMathOperator{\rmZ}{Z} 

\newcommand{\euO}{\EuScript{O}}

\newcommand{\euZ}{\EuScript{Z}}

\newcommand{\calA}{\operatorname{{\mathcal A}}}
\newcommand{\calB}{\operatorname{\mathcal B}}
\newcommand{\calC}{\operatorname{\mathcal C}}
    \newcommand{\Cliff}{{\mathcal C}\ell} 
    \newcommand{\Cli}{\Cliff}
\newcommand{\calD}{\operatorname{\mathcal D}}
    \newcommand{\Dis}{{\mathcal Dis}} 

\newcommand{\calI}{\operatorname{\mathcal I}}

\newcommand{\calO}{\operatorname{\mathcal O}}

\newcommand{\calS}{{\operatorname{\mathcal S}}}

\newcommand{\scB}{\mathscr{B}}

\newcommand{\scD}{\mathscr{D}}

\newcommand{\scO}{\mathscr{O}}

\newcommand{\scT}{\mathscr{T}}

\newcommand{\scX}{\mathscr{X}}

\newcommand{\uIsom}{\underline{\rm Isom}}

\newcommand{\uLev}{\underline{\rm Lev}} 

\newcommand{\Of}{\mathrm{Of}} 

\newcommand{\uOpp}{\underline{\rm Opp}} 

\newcommand{\uAut}{\mathbf {Aut}}

\newcommand{\uGL}{\mathbf{GL}}

\newcommand{\uO}{\mathbf {O}}

\newcommand{\uPGL}{\mathbf {PGL}}

\newcommand{\uSL}{\mathbf {SL}}
\newcommand{\uSO}{\mathbf {SO}}%

\newcommand{\bft}{\mathbf{t}}

\newcommand{\uU}{\mathbf U}

\newcommand{\uW}{\mathbf W}  

\newcommand{\frg}{\mathfrak g}

\newcommand{\gh}{\mathfrak h}

\newcommand{\m}{\mathfrak m} 
    \newcommand{\gm}{\m}

\newcommand{\p}{\mathfrak p}

\newcommand{\frR}{\mathfrak R}

\newcommand{\Ralg}{R\mathchar45\mathbf{alg}}

\newcommand{\bbA}{{\mathbb A}}

\newcommand{\GG}{{\mathbb G}}
\newcommand{\HH}{{\mathbb H}}

\newcommand{\NN}{{\mathbb N}}

\newcommand{\QQ}{{\mathbb Q}}

\newcommand{\ZZ}{{\mathbb Z}}

\newcommand\al{\alpha}
\newcommand\be{\beta}
\newcommand\ga{\gamma} 
\newcommand\Ga{\Gamma}

\newcommand\de{\delta}
\newcommand\De{\Delta}

\newcommand\veps{\varepsilon} 
\newcommand\io{\iota}
\newcommand\ka{\kappa}
\newcommand\la{\lambda} 
 \newcommand\vphi{\varphi}

\newcommand\si{\sigma}

\newcommand\om{\omega} \newcommand\Om{\Omega}

\newcommand{\bmu}{\boldsymbol{\mu}}

\title[Group schemes over LG-rings]
{Group schemes over LG-rings and applications to cancellation theorems and Azumaya algebras}

\author[P. Gille]{Philippe Gille}
\address{UMR 5208 du CNRS -
Institut Camille Jordan - Universit\'e Claude Bernard Lyon 1, 43 boulevard du
11 novembre 1918, 69622 Villeurbanne cedex, France and Institute of Mathematics "Simion Stoilow" of the Romanian Academy,
21 Calea Grivitei Street, 010702 Bucharest, Romania}
\thanks{The first author was supported by the project "Group schemes, root systems, and related representations" founded by the European Union - NextGenerationEU through Romania's National Recovery and Resilience Plan (PNRR) call no. PNRR-III-C9-2023-I8, Project CF159/31.07.2023, and coordinated by the Ministry of Research, Innovation and Digitalization (MCID)
of Romania. }
\email{gille@math.univ-lyon1.fr}

\author[E.~Neher]{Erhard Neher}
\address{ Department of Mathematics and Statistics, University of Ottawa,
Ottawa, Ontario, Canada, K1N 6N5}
\email{neher@uottawa.ca}

\date{\today}

\begin{document}

\begin{abstract} We prove several results on reductive group schemes over LG-rings, e.g., existence of maximal tori and conjugacy of parabolic subgroups. These were  proven in \cite{SGA3} for the special case of semilocal rings. We apply these results to establish cancellation theorems for hermitian and quadratic forms over LG-rings and show that the Brauer classes of Azumaya algebras over connected LG-rings have a unique representative and allow Brauer decomposition.
\end{abstract}


\maketitle

\[ \new \text{\em In memory of Nikolai Vavilov} \]

\setcounter{tocdepth}{1}
\tableofcontents

\section*{Introduction}

In volume III of the fundamental treatise \cite{SGA3} on group schemes over a scheme $S$, several results are only proven if $S=\Spec(R)$ for $R$ a semilocal ring. In this paper we extend most of these results to the case $S=\Spec(R)$ for $R$ an LG-ring.

LG-rings, an abbreviation of ``local-global" rings, axiomatize a basic property of a semilocal ring $R$. By definition, a ring $R$ is an LG-ring if a polynomial $f$ in several variables represents a unit of $R$ as soon as for every maximal ideal $\gm \ideal R$ the canonical image $f_{R/\gm}$ represents a unit over $R/\gm$. Besides semilocal rings, the class of LG-rings include rings that are von Neumann regular modulo their radical and the ring of algebraic integers, see \ref{LG-ex} for more examples and an historical background.

Among the results we prove for a reductive group scheme $G$ over an LG-ring $R$, we highlight the following:
\begin{enumerate}
  \item \label{intro-i} Every parabolic subgroup of $G$ admits a maximal $R$--torus, in particular this holds for $G$ itself (\ref{thm_maxtorus}, \ref{cor_maxtorus}).

  \item \label{intro-ii}  Two parabolic subgroups of the same type are conjugate under an element of $G(R)$ (\ref{thm_BTD}).

   \item \label{intro-3} If $R$ is a connected LG-ring, then $G(R)$ acts transitively on the set of minimal parabolic subgroups of $G$, and if $G$ is semisimple, then $G(R)$ also acts transitively on the set of maximal split subtori of $G$ (\ref{thm_conj}).
\end{enumerate}
An important ingredient in the proofs of \eqref{intro-i}--\eqref{intro-3} is a geometric characterization of LG-rings:
\begin{enumerate}\setcounter{enumi}{3}
\item \label{intro-4} If $M$ is a finite locally free module over an LG-ring $R$ and $U$ is an open quasi-compact subscheme of the scheme $\uW(M)$ associated with $M$, then $U(R) \ne \emptyset \iff U(R/\gm) \ne \emptyset$ for every maximal ideal $\gm \ideal R$, see \ref{prop_baire}\ref{prop_baire-a}.
\end{enumerate}
With \eqref{intro-4} in place, our proof strategy often is to recast a claim in terms of an open quasi-compact subscheme of an affine space and then follow the approach of \cite{SGA3}.

Our applications are twofold. First, in \S\ref{sec:canc} we prove several cancellation theorems, based on the cohomological cancellation principle \ref{prin} which says that certain embeddings of reductive group schemes over LG-rings induce injective maps in cohomology. Specializing the groups involved, we easily derive cancellation of modules and Azumaya algebras in tensor products \ref{canfgp}, cancellation of hermitian forms in \ref{hercan} and Witt cancellation of quadratic forms in \ref{canqf}.

Second, we consider Azumaya $R$--algebras $A$ in \S\ref{sec:azu} and extend several results from $R$ connected semilocal to $R$ being connected LG. In particular, we show:
\begin{enumerate}\setcounter{enumi}{4}
  \item  \label{intro3} Each two indecomposable finite projective $A$--modules are isomorphic. Equivalently, the Brauer class of $A$ contains, up to isomorphism, a unique  algebra $B$ with $\idemp(B)= \{0,1\}$, where $\idemp(B)$ is the set of idempotents of $B$ (\ref{LGdeM}).

  \item Every $A$ with $\idemp(A) = \{0,1\}$ has a Brauer decomposition (\ref{raco}).
\end{enumerate}
We note that \eqref{intro3} says that connected LG-rings have the Wedderburn property in the sense of \cite{AW14} (or see \cite[\S 7.6]{Ford}). 
\ms

{\em Organization of the paper.} We begin with a short introduction to LG-rings in \S\ref{sec:LG_rings}, list examples, recall some immediate consequences and prove the important characterization \eqref{intro-4} of LG-rings, stated above. Our investigation of group schemes over LG-rings starts in \S\ref{sec:exi}, where we prove the crucial result \eqref{intro-i}. We finish this section by discussing quasi-split and split reductive groups over LG-rings.

We study parabolic subgroups of reductive group schemes over LG-rings in \S\ref{sec:btd}. Besides \eqref{intro-ii} above and some immediate corollaries on three or two parabolic subgroups (\ref{thm_BTD22}, \ref{cor_BTD+}), we prove in \ref{cor57}\eqref{cor57-b} that there exists a unique smallest element $\bft_{\min}$ in the set of types of parabolic subgroups of a reductive $R$--group scheme $G$ for $R$ an LG-ring. This allows us to introduce the Tits index in case $G$ is semisimple and $R$ is connected (no new indices occur, \ref{tti}).

We focus on minimal parabolic subgroups and maximal split tori in \S\ref{sec:minipara}, in particular prove \eqref{intro-3} above. As a consequence, it makes sense to define the anisotropic kernel of a reductive $R$--group scheme $G$, $R$ connected LG, as the derived group of a minimal Levi subgroup, \eqref{tak1}. 
Finally, \S\ref{sec:canc} and \S\ref{sec:azu} are devoted to the applications mentioned above. \sm

{\em Notation.} We use standard notation and terminology, but note here that $\Ralg$ denotes the category of unital associative commutative $R$--algebras. An $R$--scheme is a scheme over $\Spec(R)$. A finite locally free $R$--module is the same as a finitely generated projective $R$--module, often abbreviated as finite projective $R$--module. If $M$ is such a module, $\uW(M)$ denotes the $R$--scheme representing the $R$--functor $T \mapsto M \ot_R T$, $T\in \Ralg$. For any unital associative algebra $S$ we write $S\ti$ for the set of invertible elements of $S$. Given an Azumaya algebra $A$, over a scheme or over a ring $R$, we follow \cite[2.4.2.2 and 2.4.2.2]{CF} and denote by $\uGL_1(A)$ and $\uPGL(A)$ the group schemes of invertible elements and of automorphisms of $A$. Regarding cohomology, $\rmH^1(R, G) = \rmH^1(\Spec(R), G)$ is fppf-cohomology.

\sm

{\em References to \cite{SGA3}.} Our paper has many references to \cite{SGA3}. For better readability, a reference to \cite{SGA3} will simply be written by specifying the expos\'e in Roman numbers and the result  in arabic numbers, but leaving out \cite{SGA3}. For example,\ \cite[XIV, 3.20]{SGA3} = [XIV, 3.20]. \sm

\new
{\em Acknowledgements.} We thank Skip Garibaldi and Holger Petersson for having aroused our interest in LG-rings and for comments on an earlier version of this paper. We also thank Laurent Moret-Bailly for correspondence regarding Example~\ref{lg-cor-ex}. Last, not least we thank 
the referee for a very thorough report. 
\enew

\section{LG-rings}\label{sec:LG_rings}

This section gives an introduction to LG-rings (definition, examples, immediate consequences of the definition, preliminary results). In particular, we prove the important geometric characterization \ref{prop_baire}  of LG-rings and of rings satisfying the primitive condition.

\subsection{LG-rings (definition and some known facts)}\label{LG-def} For $S\in \Ralg$ we say that a polynomial $g\in S[X_1, \ldots, X_n]$ {\em represents a unit over $S$\/} if there exist $s_1, \ldots, s_n \in S$ such that $g(s_1, \ldots, s_n) \in S\ti$. We apply this notion for a polynomial $f\in R[X_1, \ldots, X_n]$ by viewing $f$ as a polynomial over $S$ using the structure homomorphism $R \to S$. We call $R$ an {\em $LG$-ring\/} if for every $n\in \NN$ and every $f\in R[X_1, \ldots, X_n]$ the polynomial $f$ represents a unit over $R$ if and only if one of the following obviously equivalent conditions hold:
\begin{enumerate}[label={\rm (\roman*)}]
 \item\label{LG-defi} $f$ represents a unit over every localization $R_\m$, $\m$ a maximal ideal of $R$;

 \item\label{LG-defii} $f$ represents a unit over every field $R/\m$, $\m$ a maximal ideal of $R$;

  \item\label{LG-defiii} $f$ represents a unit over every field $F\in \Ralg$.
 \end{enumerate}
An LG-ring is sometimes called a ``local-global ring'', hence the short form LG, or a {\em ring with many units\/}.
The following facts will be helpful:
\sm

\begin{inparaenum}[(a)]
\item\label{LG-defa} ({\em Direct products}) Let $R_1, \ldots, R_n$ be rings. Then the direct product $R_1 \times \cdots \times R_n$ is an LG-ring if and only if every $R_i$ is an LG-ring.
\sm

\item \label{LG-defb} ({\em Characterization}) Recall that $\Jac(R)$ denotes the Jacobson radical of a ring $R$, i.e., the intersection of all maximal ideals of $R$. The following are equivalent for $R$:

      \begin{inparaenum}[(i)]
      \quad \item \label{LG-defbi} $R$ is an LG-ring;


      \quad \item \label{LG-defbiii} $R/\Jac(R)$ is an LG-ring.

\end{inparaenum}

\noindent The equivalence \eqref{LG-defbi} $\iff$ \eqref{LG-defbiii}
   holds because the maximal ideals of $R$ and of $R/\Jac(R)$ are in obvious bijection.
\sm

\item\label{LG-defe} ({\em Finite modules}) Let $M$ be a finitely presented and let $N$ be a finitely generated module over an LG-ring $R$. Then $M\cong N$ if and only if $M_\m \cong N_\m$ for all maximal ideals $\m\ideal R$ \cite[Thm.~2.6]{EG}. In particular, {\em any finite projective $R$--module of constant rank is free\/} \cite[Thm.~2.10]{EG}, \cite[II, Thm.]{MW}, and therefore
    \begin{equation}
      \label{LG-defe1}  \Pic(R) = \{0\} .
    \end{equation}
We will give a quick proof, different from the published ones, of the last statement in Corollary~\ref{prop_baireco}\eqref{prop_baireco-b}. \end{inparaenum}

\subsection{Examples of LG-rings and some history}\label{LG-ex}
\begin{inparaenum}[(a)] \item Clearly, $\{0\}$ is an LG-ring and so is every field.
It then follows from \ref{LG-def}\eqref{LG-defa} and \ref{LG-def}\eqref{LG-defb} that {\em every semilocal ring  is an LG-ring}. \sm

  \item \label{LG-ex-c} {\em If $R$ is an LG-ring and $R'\in \Ralg$ is an integral extension, then $R'$ is also an LG-ring\/} by \cite[Cor.~2.3]{EG}. In particular, every finite $R$--algebra of an LG-ring $R$ is itself an LG-ring. (Recall that an $R$--algebra $S$ is finite if and only if $S$ is integral and of finite type as $R$--algebra \cite[Tag 00GN]{Stacks}.)
\cite[\S2]{EG}.

\inparcom{(2025-02-21) I deleted the previous 1.2(c). It had two parts which we never referred to. In addition, the first part, i.e., {\tt every homomorphic image of an LG-ring is an LG-ring}, is not true. I moved the second part of the previous 1.2(c) to the new 1.2(c). }
\sm

\item\label{LG-ex-af} {\em If $R/\Jac(R)$ is von Neumann regular, then $R$ is an LG-ring} \cite[I, Prop.]{MW}.  Recall \cite[I, \S2, Exc.~17]{BAC} that a commutative ring $A$ is von Neumann regular if and only if $A$ is absolutely flat, i.e., all $A$--modules are flat. Also, {\em any direct limit of LG-rings is an LG-ring}. 
\sm

 \item {\em A zero-dimensional ring is an LG-ring\/} \cite[I, Cor.]{MW}. Indeed, a ring $R$ is zero-dimensional if all its prime ideals are maximal. For such a ring the Jacobson radical $\Jac(R)$ equals the nil radical, and so $R/\Jac(R)$ is absolutely flat = von Neumann regular by \cite[II,\S4, Exc.~16(d)]{BAC} or by \cite[Tag 092F]{Stacks}. Therefore $R$ is LG by \eqref{LG-ex-af}.
\sm

\item \label{LG-ex-pri} A polynomial in $R[X_1, \ldots, X_n]$ is called  {\em primitive\/} if its coefficients generate $R$ as ideal.
One says that a ring $R$ {\em satisfies the primitive criterion\/} \cite{EG, MW}
if the following equivalent conditions hold: \end{inparaenum}

\begin{enumerate}[label={\rm (\Roman*)}]
\item for every primitive polynomial $P\in R[X]$ there exists $r\in R$ such that $P(r) \in R\ti$;

\item for every primitive $Q\in R[X_1, \ldots, X_n]$ there exists $(r_1, \ldots, r_n)\in R^n$ such that $Q(r_1, \ldots, r_n) \in R\ti$;

\item \label{LG-ex-prii} $R$ is LG and all residue fields are infinite.
\end{enumerate}

\sm
An example of a ring satisfying the primitive criterion, is the ring $S\me R[X]$ where $R$ is arbitrary and $S$ is the multiplicative subset of all primitive polynomials in the polynomial ring $R[X]$, \cite[1.13]{vdK}.

Another interesting example is the ring of all algebraic integers or of all real algebraic integers is an LG-ring. That these rings are LG-rings, is shown in \cite{Dade, EG}. 
\new  Since they do not have finite non-zero homomorphic images, \ref{LG-ex-prii} proves our claim. \enew
\sm

\begin{inparaenum}[(a)] \setcounter{enumi}{6}  \item {\em Non-examples}: 
\new
If $R$ is an LG-ring, the polynomial $X^2 + r$, $r\in R$ arbitrary, represents a unit over $R$. It follows that 
\enew 
the rings $\ZZ$ and $R[X]$, $R$ an integral domain, are not LG-rings \cite[Exc.~11.42]{PRbook}. \sm

\item {\em Some history of LG-rings.} It seems that the concept of an LG-ring goes back to the paper \cite{MW} by McDonald and Waterhouse, where the authors refer to such a ring as a ``ring in which every polynomial with local unit values has unit values''.  The motivation of \cite{MW} comes from $K$--theory, in particular the study of $\GL_2(R)$ and $\Aut\big(\GL_2(R)\big)$. The name ``LG-ring'' was introduced in the paper \cite{EG} by Estes and Guralnick, which is the first in-depth investigation of LG-rings. Later on, LG-rings were used by Dias \cite{Dias} in her Ph.D. thesis with the goal of extending some classical theorems on quadratic forms over fields to quadratic  forms over LG-rings (her result on Witt cancellation of quadratic forms is generalized in \ref{canqf}). The subject seem to have fallen in oblivion until it was recently resurrected by Garibaldi-Petersson-Racine in their paper \cite{GPR} and book \cite{PRbook}.%
\end{inparaenum}
\sm

Our first goal is to give a more geometric approach to LG-rings, see \ref{lem_vdk} and \ref{prop_baire}.

\begin{lem} \label{lem_vdk}
Let $R$ be an LG-ring  and let $n \geq 1$ be an integer. We further assume that
$U$ is an open quasi-compact subscheme of $\bbA^n_R=X$ and let $I \ideal R[X_1, \ldots, X_n]$ be the radical ideal such that $U= X\setminus \Spec(R[X_1, \dots, X_n]/I)$. Hence $U = \bigcup_{f \in I} X_f$. \sm

\begin{inparaenum}[\rm (a)] \item\label{lem_vdk-a} Denoting by $I_{\rm prim}$ the set of primitive polynomials in $I$, we have
\[
\bigcup\limits_{f \in I_{\rm prim}} X_f(R)  =\bigcup\limits_{f \in I} X_f(R)
= U(R).
\]

\item \label{lem_vdk-b} Moreover, if $U(R/\gm) \ne \emptyset$ for every maximal ideal $\gm \ideal R$, then also $U(R) \ne \emptyset$.
\end{inparaenum}\end{lem}

\begin{proof} \eqref{lem_vdk-a} The first equality is obvious since $X_f(R)=\emptyset$ for $f \in I \setminus I_{\rm prim}$. The  inclusion $\bigcup_{f \in I} X_f(R) \subset U(R)$ is obvious. For the converse we are given  $u \in U(R)$ and need to find $f\in I$ such that $u\in X_f(R)$. Since $U$ is quasi-compact, the ideal $I$ is finitely generated (\cite[I, (1.1.4)]{EGA-neu}),
say by $f_1,\dots , f_d$ . 
Let us introduce the auxiliary  linear polynomial in the variables $Y_1, \ldots, Y_d$:
\[ P(Y_1, \dots, Y_d)=  f_1(u) Y_1+  \cdots + f_d(u) Y_d.\]
We claim:
\begin{equation}
  \label{lem_vdK1}
 \text{for every maximal ideal $\gm\ideal R$, $P$ represents a unit in $(R/\gm)^\times$. }
 \end{equation}
Indeed, fix a maximal ideal $\m \ideal R$ and put $\ka = R/\gm$.  We have
$U(\ka) = \bigcup_{g \in I \otimes_R \ka} X_g(\ka)$. Hence $u\ot 1_\ka \in X_g(\ka)$ for some $g\in I$, i.e., $g(u\ot 1_\ka) \in \ka\ti$. We can write $g$ in the form
$g= f_1 b_1+ \dots +f_d b_d$ with suitable $b_i \in R$ and then have
$ P(b_1, \dots, b_d) = f_1(u) b_1+  \cdots + f_d(u) b_d
 = g(u) \ot 1_\ka \in \ka\ti$,
establishing the claim \eqref{lem_vdK1}.

The LG-property now applies and says that
$P$ represents a unit in $R$, say $P(a_1, \ldots ,a_d) \in R\ti$ for suitable $a_i \in R$. Putting
\begin{equation}\label{lem_vdK2}
f= a_1 f_1 + \dots + a_d f_d \in I \subset R[X_1, \ldots , X_n],
\end{equation}
we have
$ f(u) = a_1 f_1(u) + \cdots + a_d f_d(u) = P(a_1, \ldots, a_d)\in R\ti$,
i.e., $u \in X_f(R)$.
\sm


\eqref{lem_vdk-b} We consider the polynomial
\begin{align*} Q &= f_1(X_1, \ldots, X_n)Y_1 + \cdots + f_d(X_1, \ldots, X_n) Y_d
\\ & \in R[X_1, \ldots, X_n, Y_1, \ldots , Y_d].
\end{align*}
We have seen in \eqref{lem_vdk-a} that the assumption of \eqref{lem_vdk-b} implies that $Q$ represents a unit in $R/\gm$ for every maximal ideal $\gm \ideal R$. Hence, since $R$ is LG, the polynomial $Q$ represents a unit in $R$, i.e., there exist $(u_1, \ldots, u_n)\in R^n$ and $(a_1, \ldots, a_d) \in R^d$ such that
\[ Q(u_1, \ldots, u_n, a_1, \ldots , a_d) \in R\ti.\]
It then follows that $u\in X_f(R)$ for the polynomial $f\in I$ defined in \eqref{lem_vdK2}, in particular $U(R) \ne \emptyset$. \end{proof}


\begin{prop} \label{prop_baire} Let $R$ be a LG-ring, let $M$ be a finite  locally free $R$-module, and let $U \subset \uW(M)$ be an open quasi-compact subscheme.
\sm

\begin{enumerate}[label={\rm (\alph*)}] \item \label{prop_baire-a}
$U(R) \not = \emptyset \iff U(R/\gm) \not = \emptyset$ for every maximal ideal $\gm \ideal R$.
\sm

\item  \label{prop_baire-b} If  $R$ satisfies the primitive condition as in {\rm \ref{LG-ex}\eqref{LG-ex-pri}} and $U$ is $R$-dense, then $U(R) \not = \emptyset$.
\end{enumerate}\end{prop}

\begin{proof}
  \ref{prop_baire-a} Of course, only the implication from right to left is to show. To do so, we take a presentation $M \oplus N= R^n$, hence $\uW(M ) \times_R \uW(N)= \bbA_R^n$. It then follows from \cite[IV$_1$, 1.1.2]{EGA} (or
  \cite[Tag 01K5]{Stacks}) 
  that $V= U \times_R \uW(N)$ is an open quasi-compact $R$--subscheme of $\bbA^n_R$. Moreover, $V(R/\gm) \not = \emptyset$ for each
maximal ideal $\gm$ of $R$ since $\uW(N)(R/\gm) = N \ot_R (R/\gm )\ne \emptyset$. 
Now \ref{lem_vdk}\eqref{lem_vdk-b} shows that $U(R) \times N= V(R) \ne \emptyset$, so that also $U(R) \not= \emptyset$. 
 \sm

\ref{prop_baire-b}
Let $\gm$ be a maximal ideal of $R$, and put $k = R/\gm$. By $R$--denseness, the $k$--variety $U_k$ is an open dense subvariety of the affine $k$--space $\uW(M_k)$, in particular $U_k$ is a rational non-empty subvariety of $\bbA_k^n$. Since $k$ is an infinite field by \ref{LG-ex}\eqref{LG-ex-pri}, it follows 
that $U_k(k) = U(k)\ne \emptyset$. Now \ref{prop_baire-a} implies that
$U(R) \ne \emptyset$.
 \end{proof}
\sm

\textbf{Remark.} Proposition~\ref{prop_baire}\ref{prop_baire-a} characterizes LG-rings. More precisely, if $R$ is any base ring such that for every finite locally free $R$--module $M$ and every open quasi-compact $U\subset \uW(M)$ we have the equivalence \ref{prop_baire-a}, then $R$ is an LG-ring.

Indeed, it suffices to evaluate the condition for a free $R$-module $M$ of finite rank and a principal open affine $U \subset U(M)$.

\subsection{Faithfully projective modules}\label{fapmod} Recall that an $R$--module $M$ is {\em faithful\/} if the structure map $R\to \End_R(M)$, $r\mapsto r\Id_M$, is injective. It is a standard fact in commutative algebra, \new see for example \cite[IX, Prop.~4.6, page 476]{Bas2},  
\enew
that the following conditions are equivalent for a finite projective $R$--module $P$:
\begin{enumerate}[label={\rm (\roman*)}]
\item $P$ is faithful;

\item every localization $P_\p$, $\p \in \Spec(R)$, is non-zero;


\item $P_{R/\m} \ne \{0\}$ for every maximal ideal $\m \ideal R$;


\item $P$ is faithfully flat;

\item there exists an $R$--module $Q$ such that $P\ot_R Q \cong R^n$ for some $n\in \NN_+$.
\end{enumerate}
In this case, $P$ is called {\em faithfully projective}.
\sm

The following corollary is a first application of Proposition~\ref{prop_baire}.
An element $m$ of a finite projective $R$--module $M$ is called {\em unimodular\/} if $m\ot_R 1_{\ka(\p)} \ne 0$ for all $\p \in \Spec(R)$, equivalently, $R\cdot m$ is a free $R$--module of rank $1$ and a direct summand of $M$ (\cite[9.17]{PRbook}, \cite[0.3]{Lo-genalg}).

\begin{cor}\label{prop_baireco} Let $R$ be an $LG$-ring and let $M$ be a finite locally free $R$--module.
\sm

\begin{inparaenum}[\rm (a)] \item \label{prop_baireco-a} If $M$ is faithfully projective, $M$ contains a unimodular element. \sm

\item \label{prop_baireco-b} If $M$ has constant rank, then $M$  is free.
\end{inparaenum}\end{cor}

\begin{proof} \eqref{prop_baireco-a} We consider the open subset $U = \uW(M)  \setminus \Spec(R)$, that is, the punctured affine space associated with the $R$-module $M$. Thus $U(R)$ consists of the  unimodular elements of $\uW(M)(R) = M$. We have $\uW(M) = \Spec( {\sf Sym} (M^*))$ where ${\sf Sym}(M^*)$ is the symmetric algebra of $M$, and $R\cong {\sf Sym}(M^*)/ I$ where $I$ is the ideal generated by finitely many linear forms $\la_1, \ldots , \la_n$ spanning $M^*$. Hence, by
\cite[I, (1.1.4)]{EGA-neu},  
$U$ is quasi-compact. Clearly, $U(R/\gm) \ne \emptyset$. Therefore  Proposition~\ref{prop_baire}\ref{prop_baire-a} shows that $U(R)$ is not empty. \sm

\eqref{prop_baireco-b} We prove the statement by induction on the rank $d$ of $M$, starting with the obvious case $d=0$. In the following we will assume that $d \geq 1$. Then $M$ is faithfully projective and therefore contains a unimodular element $m\in M$. Thus $M = Rm \oplus M'$ with $M'$ being a finite projective $R$--module of rank $d-1$. By the induction hypothesis, $M'$ is free of rank $d-1$, so $M$ is free of rank $d$. \end{proof}

\section{Existence of maximal tori}\label{sec:exi}

In this section our goal is to generalize Grothendieck's Theorem on the existence
of maximal tori in reductive group schemes \cite[XIV, 3.20(*)]{SGA3} from the semilocal case to the LG case, see Theorem~\ref{thm_maxtorus}. Our proof is a mild modification of the original proof, and is based on the expos\'es XIII and XIV of \cite{SGA3}.

Throughout this section $G$ is a reductive group scheme defined over an arbitrary ring $R$, unless specified otherwise. We abbreviate $\frg = \Lie(G)(R)$, which is a Lie $R$--algebra whose underlying module is finite projective.

\subsection{The conditions (C$_0$), (C$_1$), (C$'_1$) and (C$_2$) of [XIV, 2.9] hold for $\Lie(G)$}
The expos\'es XIII and XIV of \cite{SGA3} are written in a more general setting than the reductive case. For example, several results of {\em loc.~cit.} assume the conditions (C$_0$), (C$_1$), (C$'_1$) and/or (C$_2$) of [XIV, 2.9]. We claim: {\em These conditions are fulfilled in the reductive case.} Indeed, the nilpotent rank equals the reductive rank in this case,
which is locally constant by [XIX, 2.6], i.e., (C$_0$) holds. Moreover, the proof of [XIV, 3.7] shows that quasi-regular sections are regular for Lie algebras of reductive groups, which proves (C$_2$), and thereby also (C$_1$) and (C$_1'$) in view of [XIV, 2.9].

\begin{lem}\label{lem_reg} 
For $A\in \Ralg$ and $\frg_A = \frg \ot_R A$  we denote by $\frg_A^{\reg}$ the set of regular elements of the Lie $A$--algebra $\frg_A$, defined in  {\rm [XIV, 2.5]}. Then the $R$--functor $A \mapsto (\frg_A)^{\reg}$ is representable by an open quasi-compact $R$--subscheme $U$ of $\uW(\frg)$. Furthermore, $U$ is $R$-dense in  $\uW(\frg)$. \end{lem}

\begin{proof}
Since the condition (C$_0$) holds for $\frg$, the representability of the given $R$--functor by an open subscheme of $\uW(\frg)$ follows from [XIV, 2.10].


Because $\uW(\frg)$ is an affine scheme, the structure morphism $\uW(\frg) \to \Spec(R)$ is quasi-compact, 
\cite[Tag 01K3]{Stacks}. Hence, if the immersion $\io \co U \to \uW(\frg)$ is also a quasi-compact morphism, then so is the structure morphism $U \to \Spec(R)$ since quasi-compact morphisms respect composition \cite[Tag 01K6]{Stacks}. It is therefore sufficient to show that $\io$ is a quasi-compact morphism.

To do so, we will apply noetherian reduction, writing $R= \limind_{\, i \in I} R_i$ as the direct limit of its finitely generated, hence noetherian
$\ZZ$--subalgebras. It follows from [VI$_B$, 10.1, 10.3] and \cite[3.1.11]{Co1} that there exists $i\in I$ and a reductive $R_i$--group scheme $G_i$ such that $G = G_i\times_{R_i} R$. Let $\frg_i = \Lie(G_i)(R_i)$ and let $U_i$ be the open subscheme of $\uW( \frg_i)$ representing regular elements. Since the formation of $U$ commutes with base change, we can assume $U = U_i \times_{R_i} R$. Furthermore, the open immersion $\io$ is obtained by base change from the open immersion $\io_i \co U_i \to \uW(\frg_i)$. Because quasi-compact morphisms respect base change \cite[01K5]{Stacks}, it is now sufficient to show that $\io_i$ is a quasi-compact morphism. This is indeed true: Since $R_i$ is noetherian, $\uW(\frg_i)$ is a noetherian scheme 
and the open immersion $\io_i$ is q quasi-compact morphism by \cite[Tag 01OX]{Stacks}. Thus, $U$ is indeed quasi-compact.

Finally, since Lie algebras of reductive groups over infinite fields admit regular elements [XIV, 2.11(b)], 
$U$ is $R$--dense in $\uW(\frg)$. \end{proof}

\subsection{Cartan subalgebras and regular elements} \label{csre}
Let $S=\Spec R$. A {\em Cartan subalgebra\/} of the Lie algebra $\Lie(G)$ over $S$ is a Lie subalgebra $D$ of $\Lie(G)$ which is locally a direct summand and whose geometric fibres $D_{\bar s}$ are Cartan subalgebras of $\Lie(G_{\bar s})$ for all $s\in S$ [XIV, 2.4]. The concept of a Cartan subalgebra of a finite-dimensional Lie algebra over an infinite field, is defined in [XIII, after 4.5]. In this special case, they are the nil spaces of regular elements, and therefore always exist since regular elements exist.

One can use the equivalence of quasi-coherent $\scO_S$--modules and $R$--modules to translate the definition of a Cartan subalgebra of $\Lie(G)$ to that of a Cartan subalgebra of the $R$--Lie algebra $\frg = \Lie(G)(R)$. In particular, a Cartan subalgebra of $\frg$ is always a direct summand of the finite projective Lie algebra $\frg$.

The $R$--functor, associating with $A\in \Ralg$ the set of Cartan algebras of $\frg_A$, is representable by a quasi-projective finitely presented $R$-scheme $\scD$ [XIV, 2.16].
Also, the $R$--subfunctor $X$ of $\scD \times_R \W(\mathfrak{g})$, defined by
\[
X(A)= \{   ( \gh,  y ) \in \scD(A) \times \frg_A : y  \in \gh \}
\]
for each $R$--algebra $A$,
is representable by a quasi-projective $R$--scheme $\scX$. Furthermore, the first projection $p_1: \scX \to \scD$ is a vector bundle.

We denote by $\psi: \scX \to \W(\frg)$ the second projection,
and recall that the restriction $\psi^{-1}(U) \to U$ is an isomorphism, where $U\subset \uW(\frg)$ is the open scheme of \ref{lem_reg}. In particular, this implies that a regular element of $\frg_A$ is contained in a unique Cartan subalgebra of $\frg_A$.

\begin{thm} \label{thm_LG_reg} We assume that $R$ is an LG-ring and that one of the following conditions hold:
\begin{enumerate}[label={\rm (\roman*)}]
\item \label{thm_LG_reg1}
$R$ satisfies the primitive criterion;
\item \label{thm_LG_reg2} $G$ is adjoint.
\end{enumerate}
Then $\frg$ admits a regular element and hence a Cartan $R$--subalgebra.
\end{thm}

\begin{proof} By \ref{csre}, a regular element is contained in a (unique) Cartan subalgebra of $\frg$. It is therefore enough to prove that $\frg$ contains regular elements.\sm %

{\em Case \ref{thm_LG_reg1}}: Regular elements always exist in finite-dimensional Lie algebras over infinite fields [XIII, 4.2], i.e., $U(R/\m) \ne \emptyset$ for all infinite residue fields of $R$. Since $U \subset \uW(\frg)$ is open and quasi-compact by \ref{lem_reg}, Proposition~\ref{prop_baire}\ref{prop_baire-b} shows that $U(R) \ne \emptyset$. \sm

{\em Case \ref{thm_LG_reg2}}: Applying Proposition~\ref{prop_baire}\ref{prop_baire-a} and Lemma~\ref{lem_reg}, it suffices to show that $U(R/\m) \ne \emptyset$ for every maximal ideal $\m\ideal R$. As mentioned in \ref{thm_LG_reg1}, this is always true if $R/\m$ is infinite. But it is also true if $R/\m$ is finite and \ref{thm_LG_reg2} holds, according to a  result of Chevalley-Serre [XIV, App.]. \end{proof}

\subsection{Groups of type (C) and Cartan subalgebras}\label{gtc}
Let $S=\Spec(R)$. We recall that a smooth $R$--subgroup $H\subset G$ with connected geometric fibers is a {\em subgroup of type\/} (C) if $\Lie(H)$ is a Cartan $\calO_S$--Lie subalgebra of $\Lie(G)$. The map 
$H \mapsto \Lie(H)(R)$ is a bijection between subgroups of type (C) of $G$ and the $R$--Cartan subalgebras of $\Lie(G)(R)$, [XIV, 3.9]. The inverse map is given by
$\gh \mapsto \Norm_G(\gh)^0$. We shall use that $R$--subgroups of type (C)
are precisely the maximal tori of $G$ when $G$ is adjoint, [XIV, 3.18].

\begin{thm}\label{thm_maxtorus} Let $R$ be an LG-ring and let $G$ be a reductive $R$--group scheme. Then $G$ admits a maximal $R$--torus.
\end{thm}

\begin{proof}
Let $\euZ(G)$ be the (schematic) centre of $G$. The flat quotient $G/\euZ(G)$ is represented by a semisimple adjoint $R$--group scheme $G\ad$, [XXII, 4.3.5(ii)] or \cite[3.3.5]{Co1}. Moreover, $T \mapsto T/\euZ(G)$ defines a bijective correspondence between the set of maximal tori of $G$ and the set of maximal tori of $G\ad$. Thus, without loss of generality we can assume that $G$ is adjoint. Then
Theorem~\ref{thm_LG_reg}\ref{thm_LG_reg2} provides a Cartan $R$--algebra
of $\frg$, which by \ref{gtc} ``integrates'' to a maximal torus of $G$.
\end{proof}

\begin{lem}[Tori in subgroups of type (RC)]\label{lemRC} Let $G$ be a reductive group scheme over a ring $R$, let $H\subset G$ be a subgroup scheme of $G$ of type {\rm (RC)} in the sense of {\rm [XXII, 5.11.1]} and let $\rad^u(H)$ be the unipotent radical of $H$, {\rm [XXII, 5.11.4]}. Suppose:
\begin{enumerate}[label={\rm (\roman*)}]
\item \label{lemRCi} every reductive $R$--group scheme admits a maximal torus, e.g., assume that $R$ is an LG-ring, and

\item\label{lemRCii}   $H^1(R, \rad^u(H)) = 0$, e.g., assume that there exists a parabolic subgroup scheme $P \subset G$ such that $\rad^u(H) = \rad^u(P) \cap H$.
\end{enumerate}
Then $G$ admits a maximal torus contained in $H$.
\end{lem}

\begin{proof} (Modelled after the proof of [XXVI, 4.2.7(ii)])
The unipotent radical $U= \rad^u(H)$ is a smooth, finitely presented, and normal subgroup scheme of $H$, whose geometric fibres are connected and unipotent. Furthermore, by [XXII, 5.11.4], we have an exact sequence of $R$--group schemes
\[
 1 \longto U \longto H \xrightarrow{\;f\;} M \longto 1
\]
where $M$ is reductive. The subgroup $H$ is in particular of type (R), so that $H$ and $G$ have the same rank by [XXII, 5.2.2(b)]. By the properties of $U$, this is then also the rank of $M$.

The assumption \ref{lemRCi} 
provides us with a maximal $R$--torus $T$ of $M$. Let
$N=f^{-1}(T)$ be its pre-image. We thus have an induced exact sequence
\[
 1 \longto U \longto N \xrightarrow{\;f\;} T \longto 1
\]
of $R$--group schemes. The properties of $U$ and $T$, together with [VI$_B$, 9.2(viii)], imply that $N$ is a smooth and finitely presented subgroup scheme of $G$. Moreover, by [XVII, 5.1.1(i)(a)], its geometric fibres $N_{\bar s}$, $s\in \Spec(R)$, are connected, solvable and contain a maximal torus of $G_{\bar s}$. Thus, $N$ is a subgroup scheme of $G$ of type (R). We have $\rank N = \rank M = \rank G$, and by [XXII, 5.6.9(ii)] also $\rad^u(N) = \rad^u(H) = U$.

By [XXII, 5.6.13], the functor of maximal tori of $N$ is representable by an $R$--scheme $\scT$, which is a $U$--torsor. It then follows from assumption \ref{lemRCii} that $\scT(R) \ne \emptyset$, i.e., $N$ admits a maximal torus $T$, which is a maximal torus of $G$ because $\rank N = \rank G$. Hence $T \subset N \subset H $ fulfills our claim.

That $H^1(R, \rad^u(H)) = 0$  in case there exists a parabolic subgroup scheme $P \subset G$ such that $\rad^u(H) = \rad^u(P) \cap H$, is shown in [XXVI, 2.11] (use [XXVI, 2.5]). \end{proof}

\begin{cor}[Tori in parabolic subgroups]\label{cor_maxtorus} Let $R$ be an arbitrary ring satisfying the condition {\rm \ref{lemRC}\ref{lemRCi}},
e.g., suppose that $R$ is an LG-ring, and let $G$ be a reductive $R$--group scheme. Then every parabolic subgroup of $G$ contains a maximal torus of $G$.
\end{cor}

\begin{proof} A parabolic subgroup $P$ of $G$ is a group of type (RC) by [XXVI, 1.5]. The condition~\ref{lemRC}\ref{lemRCii} holds by [XXVI, 2.2]. Hence Lemma~\ref{lemRC} with $H=P$ proves the claim. \end{proof}
\sm

Corollary~\ref{cor_maxtorus} is the special case $P=Q$ of Proposition~\ref{prop_stan}, whose proof requires however much more work. We will present an immediate application of Corollary~\ref{cor_maxtorus} to quasi-split groups.

\subsection{Quasi-split reductive groups}\label{qsgr} We refer the reader to [XXIV, 3.9] for the definition of a quasi-split reductive group over an arbitrary scheme $S$. It simplifies greatly if $S$ is an LG-scheme, i.e., $S=\Spec(R)$ for $R$ an LG-ring:
\begin{enumerate}[label={\rm (\alph*)}]
  \item\label{qsgr-a} {\em If $S$ is an LG-scheme, a reductive $S$--group $G$ is quasi-split if and only if $G$ admits a Borel subgroup.}
\end{enumerate}
Indeed, it is explained in [XXIV, 3.9] that a reductive group $G$ over a scheme $S$ with $\Pic(S) = 0$ is quasi-split if and only if it contains a Killing couple, i.e., a pair $(B,T)$ consisting of a Borel subgroup $B$ of $G$ and a maximal torus $T$ of $G$ contained in $T$. But, by Corollary~\ref{cor_maxtorus}, every Borel subgroup contains a maximal torus of $G$. We note that \ref{qsgr-a} is the definition of a quasi-split reductive $S$--group in \cite[5.2.10]{Co1}. \sm

Recall (\cite[\S7.2]{Co1} or \cite[2.2.4.9]{CF}) that an {\em inner form\/} of a reductive $S$--group scheme $G$ is a twisted form of $G$ under a torsor in the image of $H^1(S, G/\euZ(G))$ in $H^1(S, \Aut(G))$.
\begin{enumerate}[label={\rm (\alph*)}]\setcounter{enumi}{1}
  \item\label{qsgr-b} {\em For an LG-scheme $S$, up to isomorphism every reductive $S$--group admits a {\em unique\/} quasi-split inner form. }
\end{enumerate}
This is proven in \cite[7.2.12]{Co1} for semilocal $S$, but as \cite[7.2.13]{Co1} states, \ref{qsgr-b} holds whenever $\Pic(S') = 0$ for any finite \'etale cover $S'$ of $S$. By \ref{LG-ex}\eqref{LG-ex-c}, such a cover is again an LG-scheme, so that \eqref{LG-defe1} establishes the condition $\Pic(S') = 0$.
\ms

In \ref{tars}--\ref{lem_iso} we review some concepts for group schemes over arbitrary schemes. 

\subsection{Type of a reductive group scheme} \label{tars}
Let $S$ be an arbitrary scheme and let $G$ be a reductive $S$-group scheme. We recall the notion of the type of $G$ [XXII, 2.7].  

To each point $s \in S$ we associate the isomorphism class of the root data of the reductive algebraic group $G_{\ol{\ka(s)}}$ over the algebraically closed field $\ol{\ka(s)}$, called the {\em type of $G$ at $s$} and denoted $\type_s(G)$. The function $s\mapsto \type_s(G)$ is locally constant [XXII, 2.8].

Since the type of a reductive  algebraic group is invariant under an arbitrary field extension, it follows that for each morphism $f\co S' \to S$ of schemes and for each point $s' \in S'$ we have $\type_{s'}(G_{S'})= \type_{f(s)}(G)$.   

Let $S' \to S$ be a surjective morphism (for example a flat cover) such that $G \times_S S'$ is a split reductive $S'$-group with root data $\Psi_0$ [XXII, 1.13]. The type function of the $S'$-reductive group scheme  $G \times_S S'$ is constant with value $\Psi_0$. 
The previous compatibility shows that the type function of $G$
is  constant and has value $\Psi_0$.

\subsection{Split reductive $S$--groups with $\Pic(S) = 0$ ([XXIV, 2.14])}\label{srg}  {\em Let $S$ be a scheme for which $\Pic(S) = 0$, and let $G$ be a reductive $S$--group scheme of constant type. Then $G$ is split if and only if $G$ contains a split maximal torus.}

The result applies to $S=\Spec(R)$ for $R$ an LG-ring since by \eqref{LG-defe1} we know that $\Pic(R) = 0$. Moreover, by Theorem~\ref{thm_maxtorus}, a reductive group scheme  over an LG-ring $R$ contains a maximal torus.

\subsection{Isotrivial reductive groups}\label{irg} A reductive group scheme $G$ over $S$ is {\em isotrivial\/} if there exists a finite \'etale cover (= finite \'etale surjective) $S' \to S$ such that $G_{S'}$ is split, see for example \cite[4.5.1(2)]{G2}. By \ref{tars},  an isotrivial reductive $S$--group scheme $G$ is necessarily of constant type. 

\begin{lem} \label{lem_iso}
Let $G$ be reductive group scheme $G$ over $S$ of constant type $\Psi_0$, and let $G_0$ be the Chevalley $\ZZ$--group scheme of type $\Psi_0$. Then the following are equivalent:

\begin{enumerate}[label={\rm (\alph*)}]
\item \label{lem_iso1}  $G$ is isotrivial; \sm

\item \label{lem_iso2} the $\Aut(G_0)$–-torsor $\uIsom(G_{0,S},G)$ (defined in [XXIV, 1.8]) is isotrivial, i.e., splits after passing to a finite \'etale cover of $S$.
\end{enumerate}
\end{lem}

\begin{proof} \ref{lem_iso2} $\implies$ \ref{lem_iso1}.
Our assumption is that there exists a finite \'etale cover $S'$ of $S$ such that $\uIsom(G_{0,S},G)(S') \not = \emptyset$, so that $G_{0,S'}$ is isomorphic to $G_{S'}$.
Since $G_0$ is split, it follows that $G_{S'}$ is split.
We conclude that $G$ is isotrivial. 
\sm

\noindent  \ref{lem_iso1} $\implies$ \ref{lem_iso2}.
We assume that there exists a finite \'etale cover $S'$ of $S$ such that $G_{S'}$ is split of type $\Psi_0$ and so is $G_{0,S'}$. In view of Demazure's unicity theorem [XXIII, 5.3], $G_{S'}$ is isomorphic to $G_{0,S'}$. Thus the $\Aut(G_0)$–-torsor $\uIsom(G_{0,S},G)$ has an $S'$-point and is split by $S'/S$.
Thus the  $\Aut(G_0)$–-torsor $\uIsom(G_{0,S},G)$ is isotrivial.
\end{proof}
\sm 

After this intermezzo on group schemes over arbitrary schemes we come back to the case $S=\Spec(R)$, $R$ an LG-ring. The following is a corollary of Theorem~\ref{thm_maxtorus}.

\begin{cor} \label{lg-cor} Let $R$ be an LG-ring and let $G$ be a reductive  $R$--group scheme of constant type which is linear, e.g., $G$ is semisimple \cite[Cor.\ 4.3]{G3}. Then $G$ is isotrivial. In particular, if $R$ is connected and simply connected (= does not admit non-trivial finite \'etale covers), then $G$ is split.
 \end{cor}

\begin{proof}
Theorem \ref{thm_maxtorus} provides a maximal $R$--torus $T$ of $G$. Since $T$ is linear (and of constant rank), $T$ is isotrivial (\ref{irg}) according to a result of
Grothendieck, see \cite[Thm.\ 3.3]{G3}. In other words, there exists a finite \'etale cover $R'$ of $R$ such that $T_{R'}$ is split. Since $R'$ is an LG-ring as well, \ref{LG-ex}\eqref{LG-ex-c}, 
we conclude that $G_{R'}$ is split in view of \ref{srg}.
Assume now that $R$ is connected and simply connected. Then $R'=R \times  \dots \times R$, so that $T$ is already split.
\end{proof}

\begin{example}\label{lg-cor-ex}  The ring $\overline{\ZZ}$ of algebraic integers is an  LG-ring by \ref{LG-ex}\eqref{LG-ex-pri}. It is also well-known that $\overline{\ZZ}$ is simply-connected. Hence, Corollary \ref{lg-cor}
shows that  any semisimple group scheme over $\overline{\ZZ}$ is split.
In particular, this holds for the semisimple $\overline{\ZZ}$--group $G$ of type $\rmG_2$, which is known to be the automorphism group scheme of an octonion algebra $\euO$ over $\overline{\ZZ}$ (\cite[B.15]{CG} or \cite[55.4]{PRbook}).
As a consequence, $\euO$ is split.

Moret-Bailly told us that this is a direct consequence of a theorem of Rumely
on the existence of points for ``reasonable'' schemes over $\overline{\ZZ}$  \cite[Th.\ 1]{R} (see also \cite[Th.\ 1.3]{MB}). For $H$ a flat affine group scheme of finite type over $\overline{\ZZ}$ whose generic fiber is (smooth) connected,
Rumely's  Theorem implies that  any $H$-torsor  has a $\overline{\ZZ}$-point, so it is trivial. Since  $\overline{\ZZ}$ is simply connected  a semisimple $\overline{\ZZ}$-group $G$ is an inner form of  its Chevalley form $G_0$, that is, is the twist of $G_0$ by a $G_{0,ad}$-torsor $E$ where $G_{0,ad}$ is the adjoint group of $G_0$. Since $E$ is trivial,  it follows that $G$ is split.

Coming back to the  example of octonions, Rumely's theorem  directly does the job,
 since the automorphism group $G_2$ of the split octonions is connected. Of course,
one could also appeal to the strong approximation theorem to show that $H^1( {\overline{\ZZ}},G_2)$ vanishes. Indeed, a  class in   $H^1( {\overline{\ZZ}},G_2)$
arises from   $\gamma \in H^1( R,G_2)$  where  $R$ is  the ring of integers
of a number field $F$. Furthermore  we can assume that $\gamma_F=1$ since
$H^1( \overline{\QQ},G_2)=1$. But then \cite[Satz 3.3]{H} proves $\ga = 1$.
\sm

Yet another way to see that any octonion algebra over the LG-ring $\overline{\ZZ}$ is split, has been communicated to us by Skip Garibaldi and works for any simply connected LG-ring $R$: By \cite[19.16(b)]{PRbook}, any octonion algebra over $R$  contains a quadratic \'etale subalgebra. It is split, since $R$ is simply connected. By \cite[22.9]{PRbook}, the algebra $\euO$ is therefore reduced. But then it is split by \cite[22.16]{PRbook} and \ref{prop_baireco}\eqref{prop_baireco-b}.
\end{example}

\section{Parabolic subgroups (conjugacy, relative positions)}
\label{sec:btd}

A fundamental result of Borel-Tits proves conjugacy of parabolic subgroups of the same type in reductive groups over a field \cite[Th.~4.13(a)]{BT}. This was extended by Demazure to reductive group schemes over a semilocal ring \cite[XXVI, 5.2, 5.10]{SGA3}. In Theorem~\ref{thm_BTD} we prove this result for reductive groups over an LG-ring. In \ref{thm_BTD22} and \ref{cor_BTD+} we apply this  to conjugate parabolic subgroups to a transversal or osculating position.


We use the notion of an opposite parabolic subgroup of a reductive group scheme as defined in [XXVI, 4.3.1]. Over an affine base, every parabolic subgroup admits an opposite parabolic subgroup [XXVI, 4.3.5(i)]. We will also employ the concept of the $\type(P) \in \Of\bigl( \Dyn(G))(R)$ of a parabolic subgroup $P$ of $G$, defined in [XXVI, 3.2].


\begin{thm} \label{thm_BTD} Let $R$ be an LG-ring, and let  $G$ be a reductive $R$-group scheme with a pair $(P,P')$ of opposite parabolic subgroup schemes of $G$. We abbreviate the unipotent radicals of $P$ and $P'$ by $U=\rad^u(P)$ and $U' = \rad^u(P')$ respectively.

\begin{enumerate}[label={\rm (\alph*)}]

\item \label{thm_BTD1} Any parabolic subgroup $Q$ of $G$ of the same type as $P$ has the form $Q= {^{uu'}P}$ for suitable $u\in U(R)$ and $u'\in U'(R)$. In particular, the group $G(R)$ acts transitively on $(G/P)(R)$. \sm

\item \label{thm_BTD3} The map $H^1(R,P) \to H^1(R,G)$, induced by the inclusion $P \subset G$, is injective. \sm

\item \label{thm_BTD2} We have a decomposition $G(R)=U(R) \, U'(R) \, P(R)$.
\end{enumerate}
\end{thm}

\begin{proof} \ref{thm_BTD1}
By [XXVI, 3.6], we have an isomorphism $G/P \simlgr \Par_{\type(P)}(G)$
where   $\Par_{\type(P)}(G)$ stands for  the smooth projective $R$-scheme of parabolic subgroups of type $\type(P)$. Hence it suffices to prove the first part of \ref{thm_BTD1}.

We consider the $R$-subfunctor $\uOpp(/P)$ of $\Par_{\type(P)}(G)$
of parabolic subgroups which are opposite to $P$. By [XXVI, 2.5] (or \cite[5.4.3]{Co1}), we have $U \cong \uW(E)$ for a locally free $R$--module of finite type. (That $E$ has finite type, is not explicitly stated in [XXVI, 2.5], but follows for example from the representability of $\uW(E)$.)
Hence $U$ is an affine smooth $R$-scheme. Furthermore, according to [XXVI, 4.3.6], we have a commutative diagram
\[
\xymatrix@C=1.5cm@M=5pt{
U \ar@{^{(}->}[r] \ar[d]^{\displaystyle \wr}
& G/P' \ar[d]^{\displaystyle \wr}
\\
\uOpp(/P) \ar@{^{(}->}[r] & \Par_{\type(P')} (G) ,
}
\]
This implies that $\uOpp(/P)$ is representable by an affine smooth $R$-scheme, since this is so for $U$. 

We consider the $R$-subfunctor $\uOpp(/P) \cap \uOpp(/Q)$ of  $\Par_{\type(P')}$;  it is representable by an open  $R$-subscheme $V$ of $U$. We claim that it is also quasi-compact. Indeed, we have a commutative diagram
\[ \xymatrix@C=1.5cm@M=5pt{
V \ar[rr] \ar[d]_a & &   \uOpp(/Q) \ar[d]^b \\
\uOpp(/P) \ar[rr] \ar[dr]_c && \Par_{\type(P')} (G) \ar[dl]^d\\
& \Spec(R)
}\]
in which the top rectangle is cartesian and $c$ and  $d \circ b$ are quasi-compact morphisms since $\uOpp(/P)$ and $\uOpp(/Q)$ are affine schemes. As $d$ is a
projective, 
hence separated morphism, cancellation (\cite[IV$_1$, 1.1.2]{EGA} or \cite[Tag 03GI]{Stacks}) says that $b$ is quasi-compact morphism. By base change, $a$ is quasi-compact and so $c\circ a$ is quasi-compact too, i.e., $V$ is a quasi-compact scheme.

For every $R$-ring $A$, the set $V(A)$ is the set of $A$--parabolic subgroups of $G_A$ which are opposite to $P$ and to $Q$. For a residue field $R/\gm$ of $R$, two parabolic subgroups of $G_{R/\gm}$ of the same type are conjugate by \cite[4.3]{BT}, 
so that \cite[6.27]{BT} says that $V(R/\gm)$ is non-empty for each  maximal ideal $\gm$ of $R$.  Since $R$ is a LG-ring, Proposition \ref{prop_baire} implies that  $V(R) \not = \emptyset$. In other words, $P$ and $Q$ admit a common opposite $R$--parabolic subgroup $Q'$. We now use the isomorphism of $R$-functors [XXVI, 4.3.5.(i)]
\[
\uOpp(/P) \simlgr \uLev(/P), \quad H \mapsto H \cap P
\]
between parabolic subgroups of $G$ opposite to $P$ and Levi subgroups of $P$,   and note that it is $P$--equivariant. We consider the two Levi subgroups $L=P' \cap P$ and $M=Q' \cap P$ of $P$. Since $P$ is a subgroup of type (RC), [XXVI, 1.8] says that $M=\, {^u\!L}$ for an unique $u \in U(R)$, so that $Q'= \, {^u \! P'}$
 by using the above bijection. It follows that  $P'=\, {^{u^{-1}}\!Q'}$ is opposite to $P$ and to $^{u^{-1}}\!Q$. Similarly, there exists a unique $u' \in U'(R)$
such that ${^{u^{-1}} \! Q} \cap P'= \, {^{u'}\!(P \cap P')}$
and ${^{u^{-1}}\!Q}= \, {^{u'}\!P}$. Thus $Q= \, {^{uu'}\!P}$.
\sm

\noindent \ref{thm_BTD3} According to \cite[III, 3.3.1]{Gir}, transitivity in \ref{thm_BTD1} implies that the map $H^1(R,P) \to H^1(R,G)$ has trivial kernel. The classical twisting argument applies and yields the injectivity of the map  $H^1(R,P) \to H^1(R,G)$, see for example the proof of [XXVI, 5.10]. \sm

\noindent \ref{thm_BTD2} Let $g \in G(R)$ and apply \ref{thm_BTD1} 
to $Q:=\, {^g\!P}$. Thus, there exists $u \in  U(R)$, and $u' \in U'(R)$
such that $Q= \, {^{u u'}\!P}$. It follows that $g\me u u' \in N_G(P)(R)=P(R)$. Thus $g \in   U(R) \,   U'(R)\, P(R)$. 
\end{proof}
\sm

We can also extend [XXVI, 5.3] to LG-rings.

\begin{prop} \label{thm_BTD22} Let $R$ be an LG-ring, and let  $G$ be a reductive $R$-group scheme. Let $P$, $P'$, $Q$ be three  parabolic $R$--subgroups of
$G$. Then there exists  $g \in  G(R)$ such that $^g\!Q$ and $P$ as well as ${^g\!Q}$ and $P'$ are in transversal position.
\end{prop}

\begin{proof} We denote by $\Gen(Q/P)$ the open subscheme of $G$ representing the subfunctor which assigns to an $S$--scheme $S'$ the set of elements $g\in G(S')$ such that $^g Q_{S'}$  and $P_{S'}$ are in transversal position [XXVI, 4.2.4(iii)]. We have to show that  $X= \Gen(Q/P) \cap \Gen(Q/P') \subset G$ has an $R$-point.

We claim that $\Gen(Q/P)\to G$ is quasi-compact. Indeed, since quasi-compact morphisms allow fpqc descent \cite[Tag 02KQ]{Stacks} 
 and since $(G,P)$ is splitable \'etale-locally [XXVI, 1.14], we can assume that $G$ is split and that $P=P_I$ is a standard parabolic subgroup. In this case, $\Gen(Q/P) \to G$ arises by base change from the analogous morphism over $\ZZ$. But the latter is quasi-compact since $\Spec(\ZZ)$ is noetherian \cite[Tag 01OX]{Stacks}. Since quasi-compact morphism allow base change \cite[Tag 01K5]{Stacks}, we are done.
(Alternatively, one can use noetherian reduction as in the proof of \ref{lem_reg}, to prove that $\Gen(Q/P) \to G$ is quasi-compact.) In any case, since $\Gen(Q/P)$ and $\Gen(Q/P')$ are quasi-compact, so is their fibre product $X$.

We pick a parabolic subgroup $Q'$ of $G$ which is opposite to $Q$, ([XXVI, 4.3.5(i)]), put  $U= \rad^u(Q)$ and $U'= \rad^u(Q')$, and let $\io: U \times_R U' \to G$, $(u,u') \mapsto u \, u'$, be the immersion obtained by restricting the open immersion $U \times_R Q' \to G$ ([XXVI, 4.3.2.b(vi)]). We then  consider the fibre product
\begin{equation*}  \vcenter{\xymatrix{
      U \times_R U' \ar[r]^{\io} & G  \\
      V \ar[r] \ar@{^{(}->}[u] &
      \ar@{^{(}->}[u]  X= \mathrm{Gen}(Q/P) \cap \mathrm{Gen}(Q/P')
}}\quad .
\end{equation*}
The right immersion is open and quasi-compact, hence so is $V \hookrightarrow U \times_R U'$. 
Moreover, by [XXVI, 2.5], 
$U= \uW(M)$ for some locally free $R$--module $M$ of finite rank. The analogous fact then holds for $U \times_R U'$.

With the aim of applying  Proposition \ref{prop_baire}\ref{prop_baire-a} to $V \subset U\times_R U' = W(M)$, let $\gm \ideal R$ be a maximal ideal of $R$, and put $k = R/\gm$. By [XXVI, 5.3] we know that there exists $g\in G(k)$ such that $g$ fulfills the claim of the proposition for $R=k$. By \ref{thm_BTD}\ref{thm_BTD2} we can write $g$ in the form $g=uu'q$ with $u\in U(k)$, $u'\in U'(k)$ and $q\in Q(k)$. But then $uu'\in V(k)$. Hence, by \ref{prop_baire}\ref{prop_baire-a}, we get $V(R) \not= \emptyset$. Thus $X(R) \not= \emptyset$, as desired.
\end{proof}
\sm 

Corollary~\ref{cor_maxtorus} is the special case $P=Q$ of the following Proposition~\ref{prop_stan}.

\subsection{Parabolic subgroups in standard position [XXVI, 4.5.1]}\label{stan-rev} Let $S$ be a scheme and let $G$ be a reductive $S$--group scheme. Two parabolic subgroups $P$ and $Q$ of $G$ are said to be in {\em standard position\/},  if they satisfy the following equivalent conditions \ref{stan-revi}--\ref{stan-reviv}:
\begin{enumerate}[label={\rm (\roman*)}]
\item \label{stan-revi} $P \cap Q$ is smooth. \sm
\item \label{stan-revii} $P \cap Q$ is a subgroup of type {\rm (R)}. \sm
\end{enumerate}\begin{enumerate}[label={\rm (\roman*)$'$}]\setcounter{enumi}{1}
\item \label{stan-reviip} $P \cap Q$ is a subgroup of type {\rm (RC)}.
\end{enumerate}\begin{enumerate}[label={\rm (\roman*)}]\setcounter{enumi}{2}
\item\label{stan-reviii}  $P\cap Q$ contains (fpqc)-locally a maximal torus of $G$. \sm
\item \label{stan-reviv} $P \cap Q$ contains Zariski-locally a maximal torus of $G$.
 \end{enumerate}
Moreover, if $S=\Spec(R)$ for a semilocal ring, then \ref{stan-revi}--\ref{stan-reviv} is equivalent to
\begin{enumerate}[label={\rm (\roman*)}]\setcounter{enumi}{4}
\item\label{stan-revv} $P\cap Q$ contains a maximal torus.
\end{enumerate}
If $S$ is the spectrum of a field, every pair of parabolic subgroups is  in standard position by [XXVI, 4.1.1] or \cite[4.5]{BT}. \sm

Our goal is to generalize condition~\ref{stan-revv} to $S=\Spec(R)$ for $R$ an LG-ring, see \ref{prop_stan}\ref{prop_stanc}.  In order to do so, we need to retake part of the proof of [XXVI, 4.5.1]; at the same time we will add more details.

Regarding the equivalence of the conditions above, we note that the implications
\ref{stan-reviip} $\implies$ \ref{stan-revii} $\implies$ \ref{stan-revi} and
\ref{stan-reviv} $\implies$ \ref{stan-reviii} are trivial. The implication
\ref{stan-revi} $\implies$ \ref{stan-reviip} is easy: because of [XXVI, 4.1.1] and \cite[4.5]{BT}, every geometric fibre $K_{\bar s}$ of $K=P \cap Q$ is connected and contains a maximal torus $T_{\bar s}$ of $G_{\bar s}$ such that the roots of $K_{\bar s}$ with respect to $T_{\bar s}$ are a closed subset of the roots of $(G_{\bar s}, T_{\bar s})$. Thus, if $K$ is smooth, it is a subgroup of type (RC) by definition. We will show the remaining implications \ref{stan-reviii} $\implies$ \ref{stan-revi} and \ref{stan-reviii} $\implies$ \ref{stan-reviv} in the proof of \ref{prop_stan}, since they easily follow from the arguments in the proof of \ref{prop_stan}. The statements \ref{prop_stan}\ref{prop_stana} and \ref{prop_stan}\ref{prop_stanb} below are used in the proof of \ref{stan-reviii} $\implies$ \ref{stan-reviv} in [XXVI, 4.5.1], which is the main point of the proof the equivalences in [XXVI, 4.5.1]
\sm

Regarding \eqref{prop_stanb1} we note that $P\cap Q$ is a subgroup of type (RC) by \ref{stan-reviip} and hence has a unipotent radical $\rad^u(P\cap Q)$, defined in [XXII, 5.11.4].


\begin{prop} \label{prop_stan} Let $P$ and $Q$ be two parabolic subgroups of a reductive $S$--group scheme $G$ such
that $P\cap Q$ contains a maximal torus of $G$ {\rm (fpqc)}-locally, cf.~{\em \ref{stan-rev}}. Then the following hold.
\begin{enumerate}[label={\rm (\alph*)}]
\item \label{prop_stana} The {\em (fpqc)}--image sheaf of the multiplication morphism
\[
      f\co \rad^u(P) \rtimes   (P\cap Q) \to G, \quad (u,h) \mapsto  u\, h
\]
is representable by a parabolic  subgroup $P'$ of  $G$, satisfying
\[ P' = (P\cap Q).\rad^u(P). \]

\item  \label{prop_stanb} The induced map $f'\co \rad^u(P) \rtimes   (P\cap Q) \to P'$ is smooth and
\begin{equation} \label{prop_stanb1}
    \rad^u(P\cap Q)=\rad^u(P') \cap (P \cap Q).
 \end{equation}

\item \label{prop_stanc}  If $S=\Spec(R)$ with $R$ an $LG$--ring,
then  $P\cap Q$ contains  a maximal $S$-torus of  $G$.
\end{enumerate}
\end{prop}

\begin{proof} We put $K = P \cap Q$ and prove \ref{prop_stana} and \ref{prop_stanb} together. These statements are local with respect to the (fpqc)-topology. Hence, after passing to a suitable cover, we can assume that $K$ contains a split maximal torus $T$ and that $P$ and $Q$ have constant type.

We let $R$ be the root system of $(G,T)$, choose an order on $R$ and denote by $U_r$, $r\in R$, the associated root subgroups. First some general reminders. For each closed subset $R_0$ of $R$, we have an $S$--subgroup $H_{R_0}$ of type (R), which is characterized by its Lie algebra:  $\Lie(H_{R_0})=  \Lie(T) \oplus \bigoplus_{r \in R_0} \Lie(U_r)$, [XXII, 5.4.2, 5.4.7]. The
multiplication map
\begin{equation*}
\textstyle  \prod_{r \in R_0 \cap R_+} U_r  \times_S T \times_S  \prod_{r \in R_0 \cap - R_+} U_r  \; \longto \; H_{R_0}
\end{equation*}
is an open immersion, [XXII, 5.4.4], whose image we denote by  $\Omega_{R_+, R_0}$, the so-called {\em big cell\/} of $H_{R_0}$. We claim that $\Om_{R_+, R_0}$ is schematically dense in $H_{R_0}$. Indeed, each fibre $H_{R_0,s}$, $s\in S$, is an integral scheme in which $\Om_{R_+, R_0}\cap H_{R_0,s}$ is open and nonempty, hence dense, so that the reference \cite[IV$_3$, 11.10.10]{EGA} proves our claim \footnote{ This argument is taken from the web version of [XXII, 5.6.7, N.D.E. (38)]}.

Coming back to the situation at hand, in view of [XXVI, 1.14] we can assume that
$P=H_{R_1}$ and $Q=H_{R_2}$ for parabolic subsets  $R_1$ and $R_2$ of $R$. Since
according to \cite[Cor.\ 4.5]{BT}, the group $K$ has geometrically connected fibers, it is shown in [XXII, 5.4.5] that
\begin{equation}\label{prop_stan10}
               K=H_{R_1 \cap R_2},
\end{equation}
 and that $K$ is $S$--smooth and of type (R). Since $R_1 \cap R_2$ is a closed subset of $R$, it follows from the definition of groups of type (RC) that $K$ is such a group. Note that our arguments so far show the implication \ref{stan-reviii} $\implies$ \ref{stan-revi} of \ref{stan-rev} (as well as \ref{stan-reviii} $\implies$ \ref{stan-reviip}, which is however not needed).

We denote by $R_1^s=R_1 \cap  -R_1$ the symmetric part of $R_1$ and let
$R_1^a = R_1 \setminus R_1^s$ so that $U_{R_1^a} = \rad^u(P)$, [XXVI, 1.12].
Thus, using the notation established so far, the map $f$ of \ref{prop_stana} is defined on
\begin{equation}\label{prop_stan11}
     \rad^u(P) \rtimes K = U_{R_1^a} \rtimes H_{R_1 \cap R_2}.
\end{equation}
We put
\begin{equation*}\label{prop_stan111}
R'= (R_1 \cap R_2) \cup R_1^a= (R_1^s \cap R_2) \cup R_1^a.
\end{equation*}
One easily shows, see \cite[4.4]{BT}, that $R'$ is a closed subset of $R$ satisfying $R'\cup (-R') = R$, i.e., $R'$ is a parabolic subset of $R$. Therefore
$ P'=H_{R'}$ is a parabolic subgroup of $G$, [XXVI, 1.4]. Keeping in mind \eqref{prop_stan11}, we now claim that
\begin{equation}  \label{prop_stan2} \begin{split}
  &\text{\em $f$ factorizes through an $S$-homomorphism} \\
   & \quad f'\co U_{R_1^a} \rtimes H_{R_1 \cap R_2} \to P', \quad (u,h) \mapsto u \, h.
\end{split}
\end{equation}
In other words, we contend that the induced map $f_\sharp \co U_{R_1^a} \rtimes H_{R_1 \cap R_2} \to G/P'$ is the trivial map. Since $R_1^a \subset R'$, this is equivalent to $f_\flat \co H_{R_1 \cap R_2} \to G/P'$ being the trivial map. This  is the case when we restrict $f_\flat$ to the big open cell $\Omega_{R_+, R_1 \cap R_2}$ of $H_{R_1 \cap R_2}$  which is a schematically dense in $H_{R_1 \cap R_2}$. Hence, by
\cite[Tag 01RH]{Stacks}, it follows that $f_\flat$ and therefore also $f_\sharp$ is trivial, finishing the proof of \eqref{prop_stan2}.  Note that \eqref{prop_stan2} is part of the claim \ref{prop_stana} (surjectivity will be established later). Next we show the first part of \ref{prop_stanb}, namely
\begin{equation}  \label{prop_stan3}
\text{\em $f'$ is smooth.}
\end{equation}
It follows from \eqref{prop_stan111} and the definition of $U_{R_1^a}$, $H_{R_1 \cap R_2}$ and $P'$ that $\Lie(f'): \Lie\bigl(U_{R_1^a} \rtimes H_{ R_1 \cap R_2} \bigr) \to \Lie(P')$  admits a section (as $\calO_S$--module map), so is in particular surjective. The reference \cite[IV$_4$, 17.11.1(d)]{EGA} then shows that $f'$ is smooth along the unit section,  so is smooth everywhere. We have thus proven the first part of \ref{prop_stanb}. \sm

It follows, see e.g.~\cite[IV$_4$, 17.5.1]{EGA}, that $\Ker(f')$ is smooth and hence in particular flat,   so that the quotient $\bigl(U_{R_1^a} \rtimes H_{R_1 \cap R_2} \bigr)/\Ker(f')$ is representable by an $S$--group scheme $H'$ which is locally of finite presentation and equipped with a monomorphism $h'\co  H' \to P'$,  [XVI, 2.3].
 We consider the commutative diagram
 \[\vcenter{
\xymatrix@C=45pt{
  U_{R_1^a} \rtimes H_{R_1 \cap R_2}  \ar[rd]_{f'} \ar[rr] && H'
    \ar[ld]^{h'}
  \\
 &   P'}} .
\]
The quotient map $U_{R_1^a} \rtimes H_{R_1 \cap R_2} \to H'$ is smooth surjective and  $f'$ is smooth, hence  $h'\co H' \to P'$ is smooth in view of \cite[IV$_4$, 17.7.7]{EGA} (or see \cite[Tag 02K5]{Stacks}).
According to \cite[IV$_4$, 17.9.1]{EGA}, the smooth monomorphism $h'$ is an open immersion. Since $P'$ has smooth connected fibers, $h'$ is surjective and we conclude that $h'$ is an isomorphism. This then proves \ref{prop_stana} in full.

To finish the proof of \ref{prop_stanb}, it remains to show \eqref{prop_stanb1}, i.e., $\rad^u(K) = \rad^u(P') \cap K$. In view of \eqref{prop_stan10} and the definition of $P'$, this boils down to the equality
\[ (R_1 \cap R_2)^a = R'{}^a \cap R_1 \cap R_2,\]
which is a special case of Lemma~\ref{ablem}.


Once \ref{prop_stanb} established, the implication \ref{stan-reviii} $\implies$ \ref{stan-reviv} of \ref{stan-rev} follows: we apply [XXVI, 2.11] to the parabolic subgroup $P'$ and its subgroup $K$ of type (RC) (recall \ref{stan-rev}\ref{stan-reviii} $\implies$ \ref{stan-rev}\ref{stan-reviip}), and get that $K$ admits a Levi subgroup, hence contains Zariski-locally a maximal torus of $G$ by [XIV, 3.20].
\sm 

\ref{prop_stanc} By \ref{prop_stanb}, the parabolic subgroup scheme $P'$ and its subgroup $K$ of type (RC) satisfy the assumptions of Lemma~\ref{lemRC}, so that \ref{prop_stanc} follows by applying that lemma. \end{proof}

The notation of the following Lemma~\ref{ablem} is the same as that in the proof of \ref{prop_stan}, except that we consider subsets of an arbitrary free $\ZZ$--module $M$: For a subset $N\subset M$ we put $N^s = N \cap (-N)$ and $N^a = N \setminus N^s = N \setminus (-N)$.

\begin{lem}\label{ablem} Let $A$ and $B$ be subsets of a free $\ZZ$--module $M$. We consider $C= (A\cap B) \cup A^a = (A^s \cap B) \cup A^a$. Then
\[ C^a \cap A \cap B = (A \cap B)^a.\]
\end{lem}

\begin{proof} We will first establish several auxiliary statements, starting with \begin{equation}\label{ablem1}
(A^s \cap B)^a = A^s \cap B^a.
\end{equation}
Indeed, since $A^s = - A^s$ we have $x\in (A^s \cap B)^a \iff x\in A^s\cap B$ and $-x\notin (A^s \cap B) \iff x\in A^s \cap B$ and $-x\notin B\iff x\in A^s\cap B^a$.
Next, we claim
\begin{equation}\label{ablem2}
C^a = (A^s \cap B)^a \cup A^a = (A^s \cap B^a) \cup A^a.
\end{equation}
Because of \eqref{ablem1} we only need to prove the first equality. By definition of $C^a$, we have
\begin{equation}\label{ablem22}\begin{split}
 x\in C^a & \iff x\in (A^s\cap B) \cup A^a
 \\ & \quad \text{ and } -x\notin A^s\cap B \text{ and } -x\notin A^a.
\end{split}\end{equation}
This easily implies the inclusion $A^a \subset C^a$: if $x\in A^a$, then $x\in A$ and $-x\notin A= A^s \cup A^a$, in particular $-x\notin A^s \cap B$ and $-x\notin A^a$.  We also see that  $(A^s \cap B)^a = A^s \cap B^a  \subset C^a$ since $x\in A^s\cap B^a \iff x\in A^s \cap B$ and $-x\notin B$. Moreover $-x\in A^s$, so that $-x\notin A^a$ also holds. We have now proved $(A^s \cap B)^a \cup A^a \subset C^a$. For the proof of the other inclusion, consider $x\in A^s \cap B$, but $-x\notin A^s \cap B$ and $-x\notin A^a$, cf.~\eqref{ablem22}. Since $A^s = \pm A^s$ we get $x\notin B^s$, thus $x\in A^s \cap B^a= (A^s \cap B)^a$ by \eqref{ablem1}. This finishes the proof of \eqref{ablem2}. Next we observe
\begin{equation}\label{ablem4}
  (A\cap B)^a = (A^a \cap B) \cup (A \cap B^a)
\end{equation}
since $x\in (A\cap B)^a \iff x\in A \cap B$, but $-x \notin A$ or $-x\notin B \iff x\in A^a \cap B$ or $x\in A \cap B^a$. Finally, by \eqref{ablem2} and \eqref{ablem4},
\begin{align*}
 C^a \cap A \cap B &= \big( (A^s \cap B^a) \cup A^a \big) \cap (A \cap B)
 \\&= \big((A^s \cap B^a) \cup A^a\big) \cap B
  = (A^s \cap B^a ) \cup (A^a \cap B)
  \\ &= (A\cap B^a) \cup (A^a \cap B)  = (A\cap B)^a. \qedhere
\end{align*} \end{proof}

With Propositions~\ref{thm_BTD22} and \ref{prop_stan} in place, we can now also extend [XXVI, 5.4] to  LG-rings.

\begin{cor} \label{cor_BTD+}  Let $R$ be an LG-ring, let  $G$ be a reductive $R$-group scheme, and let $P$, $Q$ be two parabolic subgroups of $G$. Then there exists  $g \in  G(R)$ such that ${^g\!P} \cap Q$ is a  parabolic  subgroup of $G$.
\end{cor}

\begin{proof} Our proof follows the proof of [XXVI, 5.4]. We include it for the convenience of the reader.

Let $P'$ be an opposite parabolic $R$--subgroup
of $G$. 
Proposition~\ref{thm_BTD22}, applied to the triple $(P',P,Q)$ provides an $R$--parabolic $R$--subgroup  $P'_1$ of $G$ of same type as $P'$
such that $P'_1$ and $P$ (resp.\   $P'_1$ and $Q$)
are in transversal position, hence in particular in standard position. By Proposition~\ref{prop_stan}\ref{prop_stanc}, there exists a maximal torus $T$ of $G$ such that $T \subset P'_1 \cap Q$. Let $P_1$ be the opposite parabolic $R$--subgroup of $P'_1$ related to  $T$, [XXVI, 4.3.3]. It follows that $P_1 \cap Q$ is an $R$--parabolic subgroup [XXVI, 4.4.5]. Since $P$ and $P_1$ are both opposite to $P'_1$, they have same type, so that $P_1= \, {^g\!P}$ for some $g \in G(R)$,  according to Theorem \ref{thm_BTD}\ref{thm_BTD1}. Thus ${^g\!P} \cap Q$ is an $R$--parabolic
subgroup of $G$ as desired. \end{proof}
\sm

Having established Corollary~\ref{cor_BTD+} for LG-rings, we also get [XXVI, 5.5 and 5.7] for LG-rings replacing semilocal rings in {\em loc.~cit.}.

\begin{cor}[{[XXVI, 5.5(i)]} for semilocal rings]\label{cor55} Let $G$ be a reductive group scheme over an LG-ring, and let $P$ and $Q$ be two parabolic subgroups of $G$ such that $\bft(P)\subset \bft(Q)$. Then there exists $g\in G(S)$ such that $^g P \subset Q$. \end{cor}

\begin{cor}[{[XXVI, 5.7(i)]} for semilocal rings]\label{cor57} Let $G$ be a reductive group scheme over an LG-ring.

\begin{inparaenum}[\rm (a)] \item Let $t,t' \in \mathrm{Of}(\underline{\Dyn}(G)\big)(S)$. If there exist parabolic subgroups of type $t$ and $t'$, then there also exists a parabolic subgroup of type $t\cap t'$.

\item \label{cor57-b} There exists a smallest element $\bft_{\min}$ in the set of types $\bft(P)$, $P$ any parabolic subgroup of $G$.
\end{inparaenum}
\end{cor}

\subsection{The Tits index}\label{tti} Let $G$ be a semisimple group scheme over a {\em connected} LG-ring. In view of \ref{cor57}\eqref{cor57-b}, we can define the Tits index of $G$ as $(\Dyn(G), \bft_{\min})$, following \cite{Tits} for fields and the generalization to the semilocal case \cite[p.~202-203]{PS}.  Moreover, as the diligent reader will verify, all the facts used in \cite{PS} for semilocal rings hold in fact for LG-rings. In particular, \cite[Thm.~3]{PS} extends to connected LG-rings: {\em The Tits index of $G$ is one of those listed in Tits' original table\/} \cite[Table II]{Tits}, reproduced in \cite[App.]{PS}.

\section{Minimal parabolic subgroups, maximal split subtori}\label{sec:minipara}

The topic of this section is minimal parabolic subgroups, their Levi subgroups and maximal split tori in reductive group schemes over a {\em connected\/} scheme $S$. We investigate the relations between these subjects in Proposition~\ref{prop_flag} and in Lemma~\ref{lem_maxi} for arbitrary connected $S$. In case $S=\Spec(R)$ for $R$ a connected LG-ring 
and $G$ a reductive $R$--group scheme, we show in Theorem~\ref{thm_conj} that the group $G(R)$ acts transitively on the minimal parabolic subgroups and on the maximal split tori of $G$. This allows us to define the anisotropic kernel of $G$ in \ref{tak}.

For perspective we note here that Appendix~\ref{app:psr} contains results on parabolic and Levi subgroups based on the dynamic method. \sm

We start by reviewing the concepts used in Proposition~\ref{prop_flag}.

\subsection{Isotropic and irreducible reductive groups.}\label{rev-isog} Let $G$ be a reductive group scheme over an arbitrary scheme $S$.

Generalizing the well-known concept of an (an)isotropic reductive group over a field \cite[4.23]{BT} or over a connected semilocal scheme [XXVI, 6.13], we call $G$ {\em isotropic} if it admits a subgroup isomorphic to $\GG_{m,S}\,$; otherwise, $G$ is said to be {\em anisotropic}, \cite[7.1.1]{G2}. Because of \ref{cor_richardson}, these concepts coincides with those defined in \cite[\S5]{PS} for semisimple group schemes over a connected semilocal scheme.

Following \cite[3.5.1]{G2}, we say that $G$ is {\em reducible\/} (as reductive group scheme) if $G$ admits an everywhere proper parabolic subgroup $P$, i.e., $P_{\bar s} \subsetneq G_{\bar s}$ for all $s\in S$, and $P$ admits a Levi subgroup; otherwise, $G$ is said to be {\em irreducible\/}. The two notions are related by \cite[Thm.~7.3.1(2)]{G2}:

\begin{enumerate}[label={\rm (\alph*)}] \item\label{rev-isog-a}  \noindent {\em A reductive $S$--group scheme $G$ is isotropic if and only if $G$ is reducible or the central torus $\rad(G)$ is isotropic. In particular, a semisimple $S$--group scheme is isotropic if and only if it is reducible.}
\end{enumerate}
If $S$ is affine, the notion of reducibility of $G$ is equivalent to the existence of an everywhere proper parabolic subgroup [XXVI, 2.3], so the definition here agrees with the terminology of \cite[3.2]{GP07}.

We will generalize the concepts of (an)isotropic and (ir)reducible reductive group schemes in \ref{iaa}, and the criterion \ref{rev-isog-a} in \ref{cor_richardson}.

\subsection{Maximal split subtori, minimal parabolics, and faithful representations} A {\em maximal split subtorus\/} of $G$ is a  split subtorus of $G$ which is maximal among all split subtori of $G$. By [XXIV, 2.11], such a torus always exists if $S$ is a scheme with trivial Picard group, e.g., $S$ is an LG-scheme, and $G$ has constant type. Existence is also assured in the setting of Proposition~\ref{prop_flag}.

A parabolic subgroup of a reductive $S$--group $G$ is a {\em minimal parabolic subgroup\/} if for all parabolic subgroups $Q$ of $G$ with $Q\subset P$ we have $Q=P$, [XXVI, 5.6].

We use the term {\em faithful linear representation\/} of an $S$--group scheme $G$ in the sense of \cite{G3}, i.e., it is a group monomorphism $G \to \uGL(E)$ where $E$ is a locally free $\calO_S$--module of finite rank.

Now that we have explained all the concepts used in the following Proposition~\ref{prop_flag}, we can finally state it. In fact, Proposition~\ref{prop_flag} is a refinement of \cite[Prop.~7.4.1]{G2}, which proves the implications \eqref{prop_flag0}.

\begin{prop}\label{prop_flag}  Let $S$ be a {\em connected\/} scheme,  $G$ a reductive $S$-group scheme, and $T_0$  a split subtorus of $G$. Moreover, let $P$ be a parabolic subgroup of $G$ for which $\rmC_G(T_0) := \Cent_G(T_0)$ is a Levi subgroup of $P$; such a parabolic subgroup exists by {\rm [XXVI, 6.2]}.
We consider  the following assertions:

\begin{enumerate}[label={\rm (\roman*)}]
\smallskip

\item \label{prop_flag1}   $T_0$ is a maximal split $S$-subtorus of  $G$;
\smallskip

\item \label{prop_flag2}   $T_0$ is a maximal split $S$-subtorus of  $\rmC_G(T_0)$;
\smallskip

\item \label{prop_flag3}   The  reductive $S$-group $\rmC_G(T_0)/T_0$ is anisotropic;
\smallskip

\item \label{prop_flag4} The  reductive $S$-group $\rmC_G(T_0)/T_0$ is irreducible;
\smallskip

\item  \label{prop_flag5}  The  reductive $S$-group $\rmC_G(T_0)$ is irreducible;
\smallskip

\item\label{prop_flag6}  $P$  is a minimal parabolic subgroup of  $G$.
\end{enumerate} \sm
We then have the implications
\begin{equation}\label{prop_flag0}
\ref{prop_flag1} \Longleftrightarrow \ref{prop_flag2}
 \Longleftarrow \ref{prop_flag3}
 \Longrightarrow  \ref{prop_flag4}
 \Longleftrightarrow \ref{prop_flag5}
\Longleftrightarrow \ref{prop_flag6}
\end{equation}
Furthermore, if $G$ admits a faithful linear representation, then
$\ref{prop_flag2}
 \Longleftrightarrow \ref{prop_flag3}$.
\end{prop}

\begin{proof}  It suffices to prove the implication \ref{prop_flag2}
$\implies$ \ref{prop_flag3} under the assumption that $G$ admits a faithful linear representation. Let $T_0$ be a maximal split $S$-subtorus of $\rmC_G(T_0)$.
We argue by contradiction and therefore assume that $\rmC_G(T_0)/T_0$ contains a  subtorus $\GG_{m,S}$. Denoting by $E\subset \rmC_G(T_0)$  the preimage of this subtorus $\GG_{m,S}$ under the quotient map, we get an exact sequence of $S$--group schemes
\begin{equation}\label{prop_flag00}
 1 \longto T_0 \longto E \longto \GG_{m,S} \longto 1.
\end{equation}
By [XVII, 7.1.1], $E$ is of multiplicative type. Clearly, it is also of finite type and it has a faithful linear representation. Therefore, by \cite[3.3]{G3},
the group scheme $E$ is isotrivial, i.e., it is split by a finite \'etale cover, which we can assume to be a (connected) Galois cover $S'\to S$, whose Galois group we denote by $\Ga$ \footnote{\label{footnote1}  We use here: {\em Given a finite \'etale cover $T\to S$ of a connected scheme $S$, there exists a connected Galois cover $S'\to S$ dominating $T\to S$.} Proof: Since finite \'etale covers form a Galois category with respect to the base change functor induced by any geometric point of $S$ \cite[Tag 0BNB]{Stacks}, the axioms of a Galois category \cite[Tag 0BMZ]{Stacks} imply that $T\to S$ is a product of finite \'etale connected covers. The existence of a Galois cover as claimed then follows from \cite[Tag 0BN2]{Stacks}, see the discussion in \cite[Tag 03SF]{Stacks}.}.
According to the dictionary [X, 1.1], the extension \eqref{prop_flag00}  corresponds to an extension
\[ 0 \longto \ZZ \longto \widehat E(S') \longto \ZZ^{r_0} \longto 0
\]
of $\Ga$--modules, where $\ZZ^{r_0}= \Hom_{S'-{\rm gr}}(E, \GG_{m,S'})$.
The $\Ga$--module $\ZZ$ is the trivial module. Hence $H^1(\Gamma, \ZZ)=\Hom_{\rm Groups}(\Ga, \ZZ) = 0$
because $\Ga$ is finite. Thus, this extension splits, so that $E$ is a split $S$--subtorus of $\rmC_G(T_0)$. This contradicts the maximality of $T_0$.
\end{proof}

\begin{remarks}\label{flagrem} \begin{inparaenum}[(a)] \item \label{flagrem-a}
The implication
$\ref{prop_flag2} \Longrightarrow \ref{prop_flag3}$
is false without an additional assumption. Indeed, according to [X, 1.6],
there exists a non-split rank $2$ torus $T$ over a singular projective connected
complex curve $S$ such that $T$ admits a split subtorus $T_0\cong \GG_m$ and an $S$--exact sequence $1 \to T_0 \to T \to   \GG_m \to 1$. Thus, $T_0$ is a maximal split subtorus of the reductive $S$--group $T=\rmC_T(T_0)$, yet $\rmC_T(T_0)/T_0$ is isotropic.  We note that $T$ is locally split in the sense of \ref{bnr}.
\sm

\item The implication
$\ref{prop_flag4} \Longrightarrow \ref{prop_flag3}$
is already false for $G=\GG_m= T$ and the obvious split subtorus $T_0 = \{1\}$.
One may then be tempted to conjecture that $\ref{prop_flag4} \Longrightarrow \ref{prop_flag3}$ holds at least in the semisimple case. But this is false too: indeed, let $G= \SL_3$ and let $T_0$ be the image of $\GG_{m,S}$ under the diagonal homomorphism $t \mapsto \diag(t, t^2, t^{-3})$. The centralizer $\rmC_G(T_0)$ is the standard diagonal maximal torus of $G$ of rank $2$. Hence $\rmC_G(T_0)/T_0\cong \GG_{m,S}$ is irreducible, but also isotropic.
\sm

\item The reader may have noticed that our proof of  \ref{prop_flag2}
$\implies$ \ref{prop_flag3} in \ref{prop_flag} only uses that $\rmC_G(T_0)$ has a faithful linear representation. This is however equivalent to our assumption on $G$: if  $\rmC_G(T_0)$ has a faithful linear representation, then so does $\rad(G)$, the unique maximal torus of the centre $\euZ(G)$ of $G$, 
and therefore also $G$ by \cite[Thm.~4.1]{G3}.
\end{inparaenum} \end{remarks}

\subsection{The dynamic description of parabolic and Levi subgroups}\label{dps}
Let $S$ be a scheme and let $H$ be a reductive $S$--group scheme. We recall  the ``dynamic method'' in the sense of \cite[4.1, 5.2]{Co1}.   

We will refer to a homomorphism $\la \co \GG_{m,S} \to H$ of $S$--groups as a {\em cocharacter} (or a {\em $1$--parameter subgroup\/} \cite[Def.~4.1.6]{Co1}). Any cocharacter gives rise to a conjugation action of $\GG_m$ on $H$, and so for a fixed $h\in H(S)$ to an orbit map $\mathrm{orb}_h \co \GG_m \to H$, which assigns to an $S$--scheme $T$ and $t\in \GG_m(T)$ the element $\la(t) h|_T \la(t)\me \in H(T)$.
If $\mathrm{orb}_h$ extends to a morphism $\GG_a \to H$, i.e., an element of $H(\GG_a)$, such a morphism is unique (since $H$ is separated) and we will abbreviate this by ``$\la(t) h \la(t)\me \in H(\GG_a)$'' (the abbreviation ``$\lim_{t\to 0} \la(t) h \la(t)\me$ exists'' is used in \cite[\S4.1]{Co1}).
One then defines an $S$--functor $\underline{\rmP_H(\la)}$ which assigns to an $S$--scheme $T$ the group
\[
\underline{\rmP_H(\la)}(T) = \big\{ h\in H(T) : \la(t) \, h_{\GG_{m,T}}\, \la(t)\me  \in H(\GG_{a,T})\big\}
\]
It is known that $\underline{\rmP_H(\la)}$ is represented by a parabolic subgroup scheme of $H$, denoted $\rmP_H(\la)$. The centralizer $\Cent_H(\la)= \rmC_H(\la)$ of $\la$ is a Levi subgroup of $\rmP_H(\la)$; its unipotent radical is given by those elements for which the limit is $1$ \cite[4.1.7]{Co1}.

If $S$ is a connected scheme, and $(P,L)$ is a pair consisting of a parabolic subgroup of  $G$ and a Levi subgroup $L$ of $P$, then there exists a cocharacter $\la$ such that $(P,L) = \big(\rmP_G(\la), \rmC_G(\la)\big)$, \cite[Thm.~7.3.1(1)]{G2}. We will refine this result in \ref{thm_richardson}. We point out that even over a connected $S$ there may exist parabolic subgroups which do not contain Levi subgroups and hence are not amenable to a dynamic description, \cite[Ex.~5.4.9]{Co1}. However, this cannot happen if $S$ is affine [XXVI, 2.3] (or \cite[5.4.8]{Co1}). Thus, over an affine base every parabolic subgroup $P$ has the form $P=\rmP_H(\la)$ for a suitable cocharacter $\la$ (\cite[7.3.2]{G2}).

In \ref{lem_maxi}, we will use the dynamic method to describe minimal parabolic subgroups and their Levi subgroups over a connected base.

\begin{lem}[Minimal parabolic subgroups and their Levi subgroups]\label{lem_maxi} 
Let $G$ be a reductive group scheme over a\/ {\em connected\/} scheme $S$.
\sm

\begin{enumerate}[label={\rm (\alph*)}]

\item \label{lem_maxi1} $G$ admits a minimal parabolic subgroup. \sm

\item \label{lem_maxi2}  Let $L$ be a Levi subgroup of a minimal  parabolic $R$-subgroup $P$ of $G$, and let $T_0$ be a maximal split $S$-subtorus of the torus $\rad(L)$; it is unique by Lemma~{\rm \ref{lem_torus}\ref{lem_torus2}}. Then $L=\Cent_G(T_0)$.  \sm

\item \label{lem_maxi3} We assume furthermore  that $G$ admits a
faithful linear representation, and let $T_0$ be a maximal split subtorus of
$G$. Then $\Cent_G(T_0)$ is a Levi subgroup of a minimal parabolic subgroup of $G$.
\end{enumerate}
\end{lem}

\begin{proof} Since $S$ is connected, the type of $G$ is constant [XXII, 2.8], and so is the type of any parabolic subgroup of $G$. \sm

\ref{lem_maxi1} The length of a strictly increasing chain of parabolic subgroups of $G$ is bounded by the relative dimension of $G$. So it is obvious that a minimal $S$--parabolic subgroup of $G$ exists. \sm

\ref{lem_maxi2} According to \cite[Thm.~7.3.1]{G2}, there exists an $S$--group homomorphism $\lambda\co \GG_{m,S} \to G$ such that $P=\rmP_G(\lambda)$ and $L=\rmC_G(\lambda)$ (notation of \ref{dps}). We can suppose that $\la$ is non-trivial. Hence, by \ref{lem_quo-ex}, $\GG_{m,S} / \Ker(\lambda) \cong \GG_{m,S}$. Without loss of generality, we can then assume that
$\lambda$ is a monomorphism. In particular, $\lambda$ factors through the unique
maximal split $S$-subtorus $T_0$ of $\rad(L)$. We have $L \subset \rmC_G(T_0) \subset  \rmC_G(\lambda)=L$ so that $L=\rmC_G(T_0)$ as desired. \sm

\noindent \ref{lem_maxi3} This follows from Proposition~\ref{prop_flag}, \ref{prop_flag1} $\Longrightarrow$ \ref{prop_flag6}.
\end{proof}

\begin{thm}[Transitivity of $G(R)$ on minimal parabolic subgroups and maximal split tori]\label{thm_conj} Let $R$ be a connected LG-ring,  and let  $G$ be a reductive $R$-group scheme.

\begin{enumerate}[label={\rm (\alph*)}]

\item \label{thm_conj1} The group $G(R)$ acts transitively on
the minimal $R$--parabolic subgroups of $G$. \sm

\item \label{thm_conj2}
We assume furthermore that  $G$ admits
a faithful linear representation (e.g.\ $G$ is semisimple).
Then the  group $G(R)$ acts transitively on
the maximal  split $R$--subtori of $G$. \sm
\end{enumerate}
\end{thm}

\begin{proof} \ref{thm_conj1} Let $P$ and $Q$ be minimal parabolic subgroups of $G$. Corollary \ref{cor_BTD+} shows that there exists $g \in G(R)$ such that ${^g\!P} \cap Q$ is an $R$--parabolic subgroup of $G$. By minimality, we have
${^g\!P} \cap Q=Q={^g\!P}$. \sm

\noindent\ref{thm_conj2} Let $T_1$ and $T_2$ be maximal split $R$-subtori of $G$. We want to show that $T_1$ and $T_2$ are $G(R)$--conjugate. According to Lemma \ref{lem_maxi}\ref{lem_maxi3}, the centralizers $L_i=\rmC_G(T_i)$, $i=1,2$, are Levi subgroups of minimal parabolic subgroups $P_i$ of $G$. Applying \ref{thm_conj1}, reduces to the case $P_1=P_2$. Since Levi subgroups of a fixed parabolic subgroup are $G(R)$--conjugate by [XXVI, 1.8], we can further assume that $L_1=L_2$. According to Proposition \ref{prop_flag}, \ref{prop_flag1} $\Longrightarrow   \ref{prop_flag2}$, $T_1$ and $T_2$ are maximal $R$--split tori of  $L_1=L_2$ and a fortiori a maximal split tori of the torus $\rad(L_1)$. Since we have uniqueness for  maximal split subtori of a given subtorus (Lemma \ref{lem_torus}\ref{lem_torus2}), we conclude that $T_1=T_2$. \end{proof}
\sm

Theorem~\ref{thm_conj}\ref{thm_conj2} is proven in [XXVI, 6.16] for a reductive $R$--group scheme over a connected semilocal $R$.

\begin{remarks}[Witt-Tits decomposition]\label{witide} {\em Let $R$ be a connected LG-ring, and let $G$ be a reductive $R$--group scheme. Then the conclusions of\/ \cite[4.3.1, 4.4.3, 5.2.1]{G2} regarding the Witt-Tits decomposition of $H^1(\Spec(R), G)$ hold.}

Indeed, the reader will easily check that the proofs of the quoted references, stated in \cite{G2} for semilocal rings, only require that $R$ is a connected ring for which $G(R)$ acts transitively on the set of minimal parabolic subgroups of $G$. But this is \ref{thm_conj}\ref{thm_conj1}.
\end{remarks}

\subsection{The anisotropic kernel}\label{tak} Let $R$ be a connected LG-ring and let $G$ be a reductive $R$--group scheme.

By \ref{lem_maxi}\ref{lem_maxi1}, $G$ admits minimal parabolic subgroups and by \ref{thm_conj}\ref{thm_conj1}, they are all conjugate. Let us fix one of them, say $P_{\min}$. Levi subgroups of $P_{\min}$ exist by [XXVI, 2.3], and they are all conjugate by [XXVI, 1.8]. Let us fix one of them, say $L_{\min}$. The derived subgroup $\scD(L_{\min})$ of $L_{\min}$ is a semisimple $R$--group scheme. It follows from [XXVI, 1.20] and \cite[Lem.~3.2.1(2)]{G2} that $\scD(L_{\min})$ is irreducible, hence anisotropic by \ref{rev-isog}\ref{rev-isog-a}. Summarizing, up to isomorphism, there exists up to isomorphism a unique semisimple anisotropic $R$--subgroup of $G$, defined as
\begin{equation} \label{tak1}  G\an = \scD(L_{\min})
\end{equation}
and called the {\em anisotropic kernel of $G$}. The terminology follows Tits \cite{Tits} for the case of a field and Petrov-Stavrova \cite{PS} for the case of semilocal rings.

\subsection{Example} \label{aniex} Let $G$ be a reductive group scheme over a connected LG-ring, and suppose $T_0$ is a split subtorus of $G$ for which $\Cent_G(T_0)/T_0$ is anisotropic. Then
\begin{enumerate}[label={\rm (\roman*)}]
 \item \label{aniex-i} $T_0$ is a maximal split subtorus of $G$;

 \item \label{aniex-ii} there exist parabolic subgroups of $G$ admitting $\Cent_G(T_0)$ as Levi subgroup, and any such parabolic subgroup is minimal;

 \item \label{aniex-iii} $\scD\big(\Cent_G(T_0)\big)$ is an anisotropic kernel of $G$. \end{enumerate}
Indeed, Proposition~\ref{prop_flag} implies \ref{aniex-i} and \ref{aniex-ii}, while \ref{aniex-iii} follows from the definition \eqref{tak1} of an anisotropic kernel. We will specialize this example in \ref{ako} to determine the anisotropic kernel of $\uGL_1(A)$, $A$ an Azumaya $R$--algebra. \sm

In the remainder of this section we investigate the interplay between arbitrary parabolic subgroups of a reductive group over a connected LG-ring, Levi subgroups  and split subtori.

\begin{prop}[{[XXVI, 6.8]} for $R$ semilocal connected] \label{centlg} Let $R$ be a connected LG-ring, let $G$ be a reductive $R$--group scheme, let $P\subset G$ be a parabolic subgroup scheme and let $L$ be a Levi subgroup of $P$, which exists by {\rm [XXVI, 2.3]}. We put $Q=\rad(L)$ and let $Q_0$ the maximal split subtorus of $Q$, which exists by\/ {\rm \ref{lem_torus}\ref{lem_torus4}}. Then
\begin{equation}\label{centlg-1} L = \Cent_G(Q_0).
\end{equation}
\end{prop}

The proof of \ref{centlg} for a connected LG-scheme is the same as the proof of [XXVI, 6.8], up to replacing the reference [XXVI, 3.20] by \ref{thm_maxtorus}. We leave the details to the reader.
\lv{
\begin{proof} The proof is a straightforward adaptation of the proof of [XXVI,  6.8]. We include it here for the convenience of the reader.

Let $L' = \Cent_G(Q_0)$. This is a reductive subgroup of $G$ containing $L$. By [XXVI, 6.2], $L'$ is the Levi subgroup of some parabolic subgroup $\wtl P$ of $G$ containing $P$. We can therefore apply [XXVI, 1.20] to the pair $(\wtl P, L')$ and get that $P'= P \cap L'$ is a parabolic subgroup of $L'$ with Levi subgroup $L$. Assume $L \ne L'$, i.e., $L \subsetneq L'$. Since $L$ is a maximal reductive subgroup of $P$, cf.\~[XXVI, 1.7], it follows that $L'\not\subset P$, and therefore $P'=P \cap L \ne L'$.

Let $G_1 = \scD(G)$ be the derived subgroup of $G$, and let $P_1 = P' \cap G_1$. By [XXVI, 1.19], $P_1$ is a parabolic subgroup of $G_1$, $L_1 = L \cap G_1$ is a Levi subgroup of $P_1$ whose radical $Q_1 = \rad(L_1)$ satisfies $Q_1  = (\rad(L) \cap G_1){}^0 = (Q\cap G_1){}^0$. By \ref{thm_maxtorus},
\comments{The reference \ref{thm_maxtorus} replaces Demazure's reference to [XIV, 3.20]}
$L_1$ admits a maximal torus $T_1$ which is isotrivial by [XXIV, 4.1.5, 4.1.6]. Hence, by [IX; 2.11(ii)], the subtorus $Q_1$ of $T_1$ is isotrivial too. Since $P_1 \ne G_1$, we can apply [XXVI, 6.6, 6.7], and obtain that the maximal split subtorus $(Q_1)_0$ of $Q_1$ satisfies $(Q_1)_0 \ne e$, and therefore $(Q_0 \cap G_1){}^0 \ne e$. A fortiori, $Q_0 \not\subset \rad(L')$ since $\rad(L') \cap G_1$ is finite. But this contradicts the definition of $L'$. \end{proof}
}
\ms

Recall [XXVI, 6.2]: {\em If $S$ is any scheme, and $Q$ a split subtorus of a reductive $S$--group scheme, then there exists a parabolic subgroup of $G$ for which $\Cent_G(Q)$ is a Levi subgroup.} \sm

Combining this result with \ref{centlg}, we get the following corollary, which is [XXVI, 6.9] in the semilocal case.

\begin{cor}\label{centlgco} Let $S$ be a connected LG-scheme, $G$ a reductive $S$--group, and $T$ a {\em critical\/} subtorus in the sense that $T= \rad\big( \Cent_G(T)\big)$. We denote the maximal split subtorus of $T$ by $T_0$. Then the following are equivalent:
\begin{enumerate}[label={\rm (\roman*)}]
 \item\label{centlgco-i} $\Cent_G(T)$ is a Levi subgroup of a parabolic subgroup of $G$;

 \item \label{centlgco-ii} $\Cent_G(T) = \Cent_G(T_0)$;

 \item \label{centlgco-iii} $\Lie(G)^T = \Lie(G)^{T_0}$.
\end{enumerate}\end{cor}

More corollaries of Proposition~\ref{centlg} are established in [XXVI, 6.10--6.12] for reductive groups over a connected semilocal scheme. Their proofs carry over without change to groups over a connected LG-scheme. We state these corollaries below for the reader's convenience.

\begin{cor}[{[XXVI, 6.10]} for $R$ semilocal connected] \label{6.10} Let $S$ be a connected LG-scheme, and let G be a reductive $S$--group. Then the following conditions are equivalent for a subgroup scheme $L$ of $G$:
\begin{enumerate}[label={\rm (\roman*)}]
  \item There exists a parabolic subgroup with Levi subgroup $L$.

  \item There exists a split subtorus of $G$ whose centralizer is $L$.

  \item There exists a cocharacter $\GG_{m,S} \to G$ whose centralizer is $L$.
\end{enumerate}\end{cor}
\sm

\begin{cor}[{[XXVI, 6.11]} for $R$ semilocal connected] \label{6.11} Let $S$ be a connected LG-scheme, and let G be a reductive $S$--group. Given a torus $T$, we denote by $T_0$ its unique maximal split subtorus, whose existence is guaranteed by {\em \ref{lem_torus}\ref{lem_torus4}}. Then the maps
\[ L \mapsto \rad(L)_0, \qquad Q \mapsto \Cent_G(Q) \]
are order-reversing  inverse bijections between the set of Levi subgroups of parabolic subgroups of $G$ and the set of split subtori $Q$ of $G$ satisfying $Q = \rad( \Cent_G(Q))_0$.
\end{cor}

\section{Cancellation Theorems}\label{sec:canc}

In this section we prove several cancellation theorems that are classically known  over fields, in some cases even over semilocal rings, but that we establish here over LG-rings: cancellation of modules and Azumaya algebras in tensor products \ref{canfgp}, cancellation of hermitian forms in \ref{hercan} and cancellation of quadratic forms in \ref{canqf}. All of these cancellation results are applications of the cohomological injectivity result of Theorem~\ref{thm_BTD}\ref{thm_BTD3}. We start with presenting the principle \ref{prin} that describes the basis of cancellation.

\begin{principle}\label{prin} Let $R$ be an LG-ring. We consider the diagram \eqref{prin-1} of $R$--group schemes and $R$--group homomorphisms
\begin{equation} \label{prin-1}
 G \xrightarrow{\;\; \De \;\; } L \xrightarrow{\; \inc_L \;} P \xrightarrow{\inc_P} H
\end{equation}
where $G$ and $H$ are reductive $R$--group schemes, $P$ is a parabolic subgroup of $H$ with Levi subgroup $L$, the map $\De$ is split as group homomorphism, and  where $\inc_L$ and $\inc_P$ are the natural inclusions. Then the canonical map
\begin{equation} \label{prin-2}
 \al\co \rmH^1(R, G) \longto \rmH^1(R,H),
\end{equation}
induced by the composition of the maps in \eqref{prin-1}, is injective.
\end{principle}

\begin{proof} The sequence~\ref{prin-1} of group homomorphisms induces maps in cohomology:
\begin{equation} \label{prin-3} \vcenter{\xymatrix@C=50pt{
     \rmH^1(R, G) \;  \ar@{..>}[d]\ar@{^{(}->}[r]^>>>>>>>>{\De^*} &
         \rmH^1(R, L) \ar[d]_\cong^{\inc_L^*}
  \\ \rmH^1(R, H) & \rmH^1(R, P) \ar[l]^<<<<<<<{\inc^*_P}
}}\quad .\end{equation}
In the diagram \eqref{prin-3} the map $\De^*$ is injective since $\De$ is a split group homomorphism, the map $\inc_L^*$ is a bijection by [XXVI, 2.3] and the map $\inc_P^*$ is injective by Theorem~\ref{thm_BTD}\ref{thm_BTD3}. \end{proof}
\sm

As a first application of the Cancellation Principle~\ref{prin}, we prove a cancellation result for finite projective modules and for Azumaya algebras over LG-rings. Cancellation of Azumaya algebras was proven for semilocal rings by Roy-Sridharan \cite[Prop.~3.2]{RS}, later by Knus \cite[Thm.~3.3]{Knus-Azu} and again by Ojanguren-Sridharan in \cite[Cor.~1]{OS}, see also \cite[III, 5.2.3(2)]{K}.


\begin{appli}[Cancellation of modules and Azumaya algebras] \label{canfgp}
Let $R$ be an LG-ring. \sm

\begin{inparaenum}[\rm (a)]
\item\label{canfgp-a} 
Let $M_1$ and $M_2$ be finite projective $R$--modules, and let $N$ be a faithfully projective $R$--module. Then
\begin{equation}   \label{canfgp1}
 M_1 \ot_R N \cong M_2 \ot_R N \implies M_1 \cong M_2
\end{equation}
(isomorphism of $R$--modules). In particular, for any $n\in \NN_+$,
\begin{equation}
  \label{canfgp2} M_1^{(n)} \cong M_2^{(n)} \implies M_1 \cong M_2.
\end{equation}

\item \label{apw} Let $A$, $B$ and $C$ be Azumaya $R$--algebras. Then
\begin{equation}\label{apw1}
 A \ot_R C \cong B \ot_R C \implies A \cong B.
 \end{equation}
In particular, for $n\in \NN_+$,
\begin{equation}\label{apw2}
 \Mat_n(A) \cong \Mat_n(B) \implies A \cong B.
 \end{equation}
\end{inparaenum}
\end{appli}

\begin{proof} \eqref{canfgp-a}
We apply the Principle~\ref{prin} with the following choices:
\begin{enumerate}[label={\rm (\roman*)}]
 \item $G= \uGL_d$ is the $R$--group scheme representing the $R$--functor which associates with $T\in \Ralg$ the group $\GL_d(T)$ of invertible matrices in $\Mat_d(T)$;
 \item $H=\uGL_{nd}$;

 \item $L$ the subgroup scheme of $\uGL_{nd}$, representing the $R$--functor which associates with $T\in \Ralg$ the subgroup
    \[ L(T) = \{ \diag(x_1, \ldots, x_n) : x_i \in \GL_d(T)\} \]
  of $\uGL_{nd}(T)$;

 \item $P$ is the product of $L$ and upper triangular matrices in $\uGL_{nd}$;

 \item $\De$ is the diagonal group homomorphism
  \[ \De \co \uGL_d \to L, \quad x \mapsto \diag(x, \ldots, x)\]
  which is split by the projection onto  the first factor
  \[ \pr_1 \co L \to \uGL_d, \quad \diag(x_1, \ldots, x_n) \mapsto x_1.\]
\end{enumerate}

We first prove \eqref{canfgp2}. In view of the rank decomposition of finite projective modules and of \ref{LG-def}\eqref{LG-defa}, it is no harm to assume that $M_1$ and hence also $M_2$ have constant rank, say they are both of rank $d$. Thus, they represent cohomology classes $[M_1]$ and $[M_2]$ in $\rmH^1(R, \uGL_d)$. The assumption in \eqref{canfgp2} is that $\al([M_1]) = \al([M_2]) \in H^1(R, \uGL_{nd})$. Therefore $[M_1] = [M_2]$ by injectivity of $\al$.

We now prove \eqref{canfgp1} as a consequence of \eqref{canfgp2}. Let  $Q$ be an $R$--module such that $N \ot_R Q$ is free of finite rank, say $N\ot_R Q \cong R^{(n)}$, see \ref{fapmod}. Observe that $M_1 \ot_R N \ot_R Q \cong M_1 \ot_R R^{(n)} \cong M_1^{(n)}$. Therefore, the assumption in \eqref{canfgp1} implies $M_1^{(n)} \cong M_2^{(n)}$, which forces $M_1 \cong M_2$ by \eqref{canfgp2}.
\sm

\eqref{apw} The proof of \eqref{apw} is a straightforward modification of the proof of \eqref{canfgp-a}. Quotioning the sequence
\[  \uGL_d \xrightarrow{\;\; \De \;\; } L \xrightarrow{\; \inc_L \;} P \xrightarrow{\inc_P} \uGL_{nd} \]
of \eqref{canfgp-a} 
by the central $\GG_m$, we obtain the sequence of $R$--group schemes
\[
  \uPGL_d \xrightarrow{\; \; \bar \De  \; \;} \ol L = L/\GG_m \;
     \xrightarrow{\; \inc_{\bar L}\;} \ol P = P / \GG_m \;
        \xrightarrow{\; \inc_{\bar P}\;} \uPGL_{nd}
\]
which satisfies the assumptions of the Principle~\ref{prin}, cf.~\cite[3.2.1]{G2}.
Thus the canonical map $\bar \al \co \rmH^1(R, \uPGL_d) \to \rmH^1(R, \uPGL_{nd})$ is injective.
The cohomology set $\rmH^1(R, \uPGL_d)$ represents the isomorphism classes of Azumaya algebras of constant degree $d$. The map $\bar \al$ sends the class of an Azumaya $R$--algebra $D$ of degree $d$ to the class $[\Mat_n(D)] \in \rmH^1(R, \uPGL_{nd})$.

For the proof of \eqref{apw2} we can assume that $A$ and $B$ have constant rank, say rank $d$. Then \eqref{apw2} follows from injectivity of the map $\bar \al$ above. To prove \eqref{apw1}, we can suppose that $C$ has constant rank $n$. Then $C\ot_R C \op \cong \End_R(C)$, where on the right hand side we view $C$ as projective $R$--module of rank $n$. Since $C\cong R^n$ by \ref{prop_baireco}\eqref{prop_baireco-b}, we get $C\ot_R C\op\cong \Mat_n(R)$. Now \eqref{apw1} follows:
\begin{align*}
 & A \ot_R C \cong B \ot_R C \implies A \ot_R C \ot_R C\op \cong B \ot_R C \ot_ R C\op
\\ &\implies \Mat_n(A) \cong \Mat_n(B) \implies A \cong B. \qedhere
\end{align*}
\end{proof}
\sm

Based on \ref{canfgp}\eqref{apw}, we will be able to say more about Azumaya algebras over LG-rings in \S\ref{sec:azu}.

\begin{remark}[Additive Cancellation of modules] \label{rem-adc} The additive version of \eqref{canfgp1} is known to be true too, even in a more general setting than modules over LG-rings \cite[Thm.~2.5]{EG}: Let $R$ be an LG-ring, let $S\in \Ralg$ be a direct limit of finite $R$--algebras, and let $M_1$, $M_2$ and $N$ be $S$--modules. If $N$ is finitely generated, then
\begin{equation}\label{adc1}
   M_1 \oplus N \cong M_2 \oplus N \quad \implies \quad M_1 \cong M_2.
\end{equation}
The proof of {\em loc.~cit.} establishes that $\End_S(N)$ has $1$ in the stable range of $\End_R(N)$ and then applies a result of Evans \cite[Thm.~2]{Evans} to show \eqref{adc1}. Additive cancellation goes back to the classical results of Bass \cite[6.6, 9.3]{Bass}; in the semilocal case it is  proven in \cite[VI; (1.3.3)]{K}.
\sm

The point of this remark is that the special case $S=R$ and $M_1$, $M_2$ and $N$ finite projective $R$--modules of \eqref{adc1} can easily be proven using the Cancellation Principle~\ref{prin}.
Indeed, there exists a finite projective $R$--module $N'$ such that $M \oplus N' \cong R^n$ for some $n\in \NN_+$. Therefore $M_1 \oplus R^n \cong M_2 \oplus R^n$. By induction on $n$, it suffices to prove the case $n=1$,
\begin{equation} \label{adc2}
M_1 \oplus R \cong M_2 \oplus R \implies M_1 \oplus M_2.
\end{equation}
It follows that $\rank_R M_1 = \rank_R M_2$. Applying the standard rank decomposition of finite projective modules and \ref{LG-def}\eqref{LG-defa}, we can assume that $M_1$ and $M_2$ have constant rank $r$.

Let $L$ be the Levi subgroup of the parabolic subgroup $P$ of the reductive $R$--group scheme $H=\uGL_{r+1}$, given with obvious meaning by
\[
   L = \begin{pmatrix} \uGL_r & 0 \\ 0 & \GG_m \end{pmatrix}, \qquad
   P= \begin{pmatrix} \uGL_r & *  \\ 0 & \GG_m \end{pmatrix}.
\]
We can apply the Cancellation Principle~\ref{prin} with $\De \co G=\uGL_r \to L$ the obvious embedding and $L \subset P \subset \uGL_{r+1} = H$ as above. Hence,
we get an injective map
\begin{equation} \label{adc4}
\rmH^1(R, \uGL_r) \, \hookrightarrow \, \rmH^1(R, \uGL_{r+1}),
\end{equation}
which sends the isomorphism class of a finite projective $R$--module $Q$ of rank $r$ to the isomorphism class of $Q\oplus R$. Thus, injectivity of \eqref{adc4} yields a proof of \eqref{adc2}. \end{remark}
\sm

The final two results of this section concern cancellation of hermitian and quadratic forms. For the convenience of the reader, Appendix \ref{sec:hqr} reviews the concepts needed for \ref{hercan} and \ref{canqf} and their proofs.


\begin{appli}[Cancellation of hermitian forms]\label{hercan}
Let $R$ be an LG--ring, and  let $S/R$ be a quadratic \'etale extension with standard involution $\si$. Then cancellation holds for $(S,\si)$--hermitian spaces: if $h_1$, $h_2$ and $h_3$ are hermitian spaces such that $h_1 \perp h_3$ and $h_2\perp h_3$ are isometric, then already $h_1$ and $h_2$ are isometric:
\begin{equation}  \label{hercan1}
h_1 \perp h_3 \cong h_2 \perp h_3 \implies h_1 \cong h_2.
\end{equation}
\end{appli}

\begin{proof} The form $h_3 \perp (-h_3)$ is hyperbolic \cite[I; 3.7.3]{K}.
Since the assumption implies that $h_1 \perp (h_3 \perp -h_3) \cong h_2 \perp (h_3 \perp - h_3)$, it is no harm to assume that $h_3$ is hyperbolic, say $h_3 = \HH(U_3)$.
We can then apply the Cancellation Principle~\ref{prin} with the following choices:
\begin{enumerate}[label={\rm (\roman*)}]
 \item $G= \uU(h_1)$ and $H=\uU(h_1 \perp h_3)$ are the unitary $R$--group schemes of the regular hermitian forms $h_1$ and $h_1 \perp h_3$ respectively, which are reductive $R$--group schemes by \ref{ung}\eqref{ung-d}.

 \item $L = \uU(h_1) \times_R \frR_{S/R}\big(\uGL(U_3)\big)$;

\item $P$ is the stabilizer of $U_3$ in $H$, which by \ref{dec} is a parabolic subgroup of $H$ with Levi subgroup $L$.

 \item $\De \co \uU(h_1) \to \uU(h_1) \times_R \uU(h_3)$ is the canonical embedding.
\end{enumerate}
Thus the natural map $\rmH^1\big(R, \uU(h_1)\big) \to \rmH^1\big(R, \uU(h_1 \perp h_3)\big)$ is injective. Taking the bijection of \ref{ung}\eqref{ung-e} as identification, it sends the isometry class of a hermitian space $h'$ to the isometry  class of $h' \perp h_3$, which implies \eqref{hercan1}. \end{proof}
\sm

\textbf{Remark.} For a semilocal ring $R$, a more general result than \ref{hercan} is proven \new in \cite[(3.4)]{Keller} and reproduced in \cite[VI; (5.7.5)]{K} \enew --- it concerns cancellation of unitary (= hermitian) spaces over unitary rings.

\begin{appli}[Cancellation of quadratic forms]\label{canqf} Let $R$ be an LG-ring, let $q_1$ and $q_2$ be nonsingular quadratic forms and let $q_3$ be a {\em regular\/} quadratic form. Then
\begin{equation}\label{canqf1} q_1 \perp q_3 \cong q_2 \perp q_3 \implies q_1 \cong q_2.
\end{equation}
\end{appli}

\begin{proof} Applying the rank decomposition \ref{qf}\eqref{qf-redc} of the quadratic forms $q_i$, $i=1,2,3$, together with \ref{LG-def}\eqref{LG-defa}, we can assume that the $q_i$ have constant rank. We can further suppose that all three ranks are positive, since otherwise the claim is obvious. Moreover, because
the quadratic form $q_3 \perp -q_3$ is hyperbolic by \cite[I, (4.7)]{Ba},
we can assume that $q_3$ is hyperbolic. In particular, $q_3$ is regular, so that
$q:=q_1 \perp q_3$ is nonsingular by \eqref{qf-perp1}.

At this point, the reader might be inclined to apply the principle \ref{prin} to $\uO(q_1)$ and $\uO(q)$, which by \ref{ogs-co} describe the isometry classes of nonsingular quadratic forms. This is however not possible because $\uO(\cdot)$ is not reductive in general. A way out of this problem is to use $\uSO(q_1)$ and $\uSO(q)$ instead, which are indeed reductive group schemes by \ref{sogsc}, and then investigate the relation between $\rmH^1(R, \uSO(\cdot))$ and $\rmH^1(R, \uO(\cdot))$.

Following this approach, we apply \ref{prin} with $G=\uSO(q_1)$ and $H=\uSO(q)$, using the parabolic and Levi subgroups of $H$ exhibited in \ref{sopl}. Hence
\[ \al\co \rmH^1(R, \uSO(q_1)) \to \rmH^1(R, \uSO(q)) \]
is injective. To link $\rmH^1(R, \uSO(\cdot))$ and $\rmH^1(R, \uO(\cdot))$, we
use that there exist an $R$--group scheme $A$, an $R$--group homomorphism $D$ and an exact sequence of $R$--group schemes
\begin{equation}\label{canqf2}
1 \longto \uSO(q) \xrightarrow{\;i\;} \uO(q) \xrightarrow{\;D\;} A \longto 1
\end{equation}
where $i$ is the canonical inclusion and where
\[ A = \begin{cases} \bmu_{2,R}, &\text{if $q$ has odd rank}, \\
                      \ZZ/2\ZZ_R, &\text{if $q$ has even rank}.
\end{cases}
\]
Indeed, this follows from \ref{sogsc}\eqref{sogsc-odd} in the odd rank case and from \eqref{sogsc-even1} in the even rank case.

Part of the long exact cohomology sequence associated with \eqref{canqf2} is
the sequence of pointed sets
\[
\orth(q) \xrightarrow{\;D(R)\;} A(R) \xrightarrow{\;\de\;} \rmH^1(R, \uSO(q))
     \xrightarrow{i^*} \rmH^1(R, \uO(q)) ,
\]
in which $D(R)$ is surjective by \ref{LGqdi}\eqref{LGqdi-b}, implying that $i^*$ has trivial kernel.
\lv{
By surjectivity of $D(R)$, $\Ker(\de) = A(R)$, so $\de$ is the trivial map, whose image is the base point of $\rmH^1(R, \uO(q))$, which in turn is the kernel of $i^*$.}
Since the parities of the ranks of $q_1$ and $q$ agree, we have the analogous exact sequence \eqref{canqf2} for $q_1$ replacing $q$, with the same $A$. Moreover, we get a commutative diagram of pointed sets
\[\xymatrix{
   \rmH^1(R, \uSO(q_1)) \ar[r]^{i_1^*} \ar[d]_\al
    & \rmH^1(R, \uO(q_1))\ar[r]^{D^*} \ar[d]^\be
        & \rmH^1(R, A) \ar@{=}[d]
\\
\rmH^1(R, \uSO(q)) \ar[r]^{i^*} & \rmH^1(R, \uO(q))\ar[r]^{D^*}
        & \rmH^1(R, A)
}\]
where $i_1^*$ is associated to the inclusion $i_1 \co \uSO(q_1) \to \uO(q_1)$ and where $\be$ maps the isometry class of $q_2$ to that of $q_2 \perp q_3$ (recall \ref{ogs-co}). Our claim then is that $\be$ has trivial kernel. But this follows from a simple diagram chase, using that $\al$ is injective and that $i^*$ has trivial kernel.
\lv{
We have $D^*([q_1]) = D^*([q_1\perp q_3]) = D^*([q_2 \perp q_3])= D^*([q_2])$. Hence $[q_2] \in \Ker(\inc_1^*)$ and so there exist cohomology classes $c_i \in H^1(R, \uSO(q_1))$ such that $\inc_1^*(c_i)= [q_i]$ (of course, $c_1$ is the base point of $\rmH^1(R, \uSO(q_1))$). Since $\be([q_1]) = \be([q_2])$ we get
\begin{align*}
  (i^*\circ \al)(c_1)&= (\be \circ i_1^*)(c_1) = \be([q_1])
    = \be([q_2]) \\& = (\be \circ i_1^*)(c_2) = (i^*\circ \al)(c_2).
\end{align*}
Since $i^*$ has trivial kernel, $[q_1] = \al(c_1) = \al(c_2) = [q_2]$ follows.
}
\end{proof}
\sm

\textbf{Remarks.} For regular quadratic forms over fields of characteristic not $2$, the result~\ref{canqf} goes back to Witt, and is therefore referred to as {\em Witt cancellation}, even in more general settings.

Witt cancellation is proven in \cite[Thm.~8.4]{EKM} for quadratic forms over arbitrary fields, in \cite[III, Cor.~4.3]{Ba} for regular quadratic forms $q_1$, $q_2$ and $q_3$ over semilocal rings, and in \cite[II, 6.4]{Dias} for LG-rings, again for regular forms, following Baeza's approach which in turn goes back to Knebusch.

Witt cancellation is not true if all $q_i$ are merely nonsingular, see \cite[p.~49]{EKM} for a counterexample.

\section{Azumaya algebras over LG-rings} \label{sec:azu}
In this section we consider Azumaya algebras over rings. One of its goals is Corollary~\ref{raco} on the Brauer decomposition of Azumaya algebras over connected LG-rings. This is a consequence of Theorem~\ref{LGdeM} which says that indecomposable finite projective modules of Azumaya algebras over a connected LG-ring are isomorphic. Another highlight of this section is Proposition~\ref{prop_isotrivial} proving that an Azumaya algebra $A$ of constant degree over an LG-ring admits a splitting ring which is a maximal \'etale subalgebra of $A$.

We start this section with Corollary~\ref{hil}, which is an application of cancellation of tensor products of Azumaya algebras and establishes ``Hilbert 90''. In this corollary, $\Br(R)$ denotes the Brauer group of a ring $R$ as defined in \cite[\S7.3]{Ford}, which is the Brauer-Azumaya group of \cite[3.1.3]{Collio}.

\begin{cor}\label{hil}
Let $R$ be an LG-ring, and let $A$ be an Azumaya $R$--algebra.%
\sm

\begin{inparaenum}[\rm (a)]
  \item\label{hil-a} If $B$ is another Azumaya $R$--algebra with $\deg A = \deg B$, then $A$ and $B$ are isomorphic as $R$--algebras if and only if their Brauer classes agree,
\begin{equation} \label{hil-a1}
           A \cong B \iff [A] = [B] \, \in \Br(R).
\end{equation}

\item\label{hil-b} {\em (``Hilbert 90'')} $\rmH^1 \big(R, \uGL_1(A)\big) = \{1\}$.
\end{inparaenum}
\end{cor}

\begin{proof}  \eqref{hil-a} It suffices to prove that $[A] = [B] \implies A \cong B$. The equality
$[A]=[B]$ means that there exist faithfully projective $R$--modules $P$ and $Q$ such that $A \ot_R \End_R(P) \cong B \ot_R \End_R(Q)$ as $R$--algebras. Using the rank decomposition of $P$ and $Q$, we can easily reduce to the case that both $P$ and $Q$ have constant rank. Hence, by  \ref{prop_baireco}\eqref{prop_baireco-b},  there exist $m,n\in \NN_+$ such that $\Mat_m(A) \cong \Mat_n(B)$ (isomorphism of $R$--algebras). Comparing degrees, we get $m=n$,
hence $A\cong B$ by \eqref{apw2}. \sm

\eqref{hil-b} Let $\uPGL(A)$ be the automorphism group scheme of $A$, \cite[2.4.4.2]{CF}. The standard exact sequence of $\Spec(R)$--group schemes,
\[ 1 \to \GG_m \to \uGL_1(A) \to \uPGL(A) \to 1 
\] 
gives rise to a long exact sequence of pointed cohomology sets, part of which is
\begin{equation} \label{hil-b1} 
  \rmH^1(R, \GG_m) \to \rmH^1(R, \uGL_1(A)) \xrightarrow{\al}  \rmH^1(R, \uPGL(A)) \xrightarrow{\de}  \rmH^2(R, \GG_m).
\end{equation} 
On the left end, $\rmH^1(R, \GG_m) = \Pic(R) = 0$ by \eqref{LG-defe1}, which implies that $\al$ has trivial kernel. The set $\rmH^1(R, \uPGL(A))$ classifies twisted forms of $A$. If $B$ is such a twisted form, it has the same degree as $A$. By \cite[4.2.12.(iii)]{Gir}, the map $\de$ sends the isomorphism class of $B$ to $[A]-[B] \in \Br(R) \subset \rmH^2(R, \GG_m)$. Hence, part \eqref{hil-a} implies that $\de$ has trivial kernel, i.e., $\al$ has trivial image. This implies 
\eqref{hil-b}. \end{proof}

\begin{prop} \label{prop_isotrivial} Assume that $R$ is  LG-ring.
Let $A$ be an Azumaya $R$--algebra of constant degree $d$. \sm

\begin{enumerate}[label={\rm (\roman*)}]
 \item \label{prop_isotrivial1}
 $A$ admits a maximal \'etale $R$--subalgebra $S$ of finite degree $d$
 and $A \otimes_R S \cong \Mat_d(S)$.

\item \label{prop_isotrivial2} If $R$ is connected, then
$A$ admits a {\em connected\/} finite \'etale $R$--subalgebra $S'$
 such that $A \otimes_R S' \cong \Mat_d(S')$.
\end{enumerate}
\end{prop}

\begin{proof} \ref{prop_isotrivial1}
According to Theorem \ref{thm_maxtorus}, the reductive $R$--group scheme $\uGL_1(A)$ admits a maximal torus $T$. By \cite[Prop.\ 3.2]{G3}, $T$ occurs from an embedding  $S \hookrightarrow A$ of a maximal finite \'etale algebra $S$ of degree $d$ and $T=\uGL_1(S)$. The canonical $S$--module $A$ is finite projective of constant rank $d$ and $A\ot_R S \cong \End_S(A)$ as Azumaya $S$--algebras (\cite[7.4.2]{Ford}, \cite[III, Prop.\ 6.1]{KO}). In particular, $A\ot_R S$ is a neutral Azumaya $S$--algebras of degree $d$. Next, $S$ is an LG-ring by Example \ref{LG-ex}\eqref{LG-ex-c}. Since $A\ot_R S$ and $\Mat_d(S)$ are both neutral Azumaya $S$--algebras of the same degree, Corollary \ref{hil}\eqref{hil-a} enables us to conclude that $A \otimes_R S \cong \Mat_d(S)$. \sm

\noindent \ref{prop_isotrivial2}
Because $R$ is connected, we have a decomposition $S=S_1 \times \dots \times S_c$ with $S_i$ finite \'etale connected. 
Since $A \ot_R S \cong \Mat_d(S)$, we obtain that $A \otimes_R S' \cong \Mat_d(S')$ for  $S'=S_1$.
\end{proof}


\subsection{Indecomposable finite projective modules}\label{ifpm}
Let $R$ be an arbitrary ring and let $A$ be an Azumaya $R$--algebra. Recall that an $A$--module $M$ is called {\em decomposable\/} if there exists a family $(M_i)_{i\in I}$, $|I|\ne 1$,  of submodules of $M$ such that $M=\bigoplus_{i\in I} M_i$ and every $M_i \ne 0$. Otherwise, it is called {\em indecomposable\/}.
Hence $M=\{0\}$ is decomposable by taking $I=\emptyset$, and an $A$--module $M$ is indecomposable if and only if $M\ne 0$ and $\idemp\big( \End_A(M)\big) = \{0,1\}$.  Here $\idemp(B) = \{b \in B: b^2=b\}$ is the set of idempotents of an $R$--algebra $B$.
\sm

Following \cite[V, p.~131]{DI}, we define an equivalence relation $\sim$ on the set of indecomposable finite projective left $A$--modules by $P\sim Q \iff$ there exists an invertible $R$--module $E$ such that $P \cong Q \ot_R E$ as (left) $A$--modules. Obviously, if $\Pic(R) = \{0\}$, then ``equivalence'' reduces to ``isomorphism".
We will use the following result.

\begin{thm}[{\cite[V; Thm.~1.1]{DI}}] \label{dbi} Let $A$ be an Azumaya algebra over a connected ring $R$. Then
\[ P \mapsto \End_A(P)\op
\] induces a bijection between the set of equivalence classes of indecomposable finite projective $A$--modules and the set of isomorphism classes of Azumaya $R$--algebras $B$ satisfying
\begin{equation} \label{dbi-1}
     A \sim_{\rm Br} B \quad \text{and} \quad \idemp(B) = \{0,1\},
 \end{equation}
where $\sim_{\rm Br}$ indicates Brauer equivalence. \end{thm} \sm

The inverse map assigns to $B$ satisfying \eqref{dbi-1} the module $P$ constructed as follows: By definition of $A \sim_{\rm Br} B$ there exists a faithfully projective $R$--module $P$ satisfying
\begin{equation} \label{dbi-2}
  A \ot_R B\op \cong \End_R (P).
 \end{equation}
We use the canonical $R$--algebra homomorphism $A \to A \ot_R B\op$ and \eqref{dbi-2} to endow $P$ with an $A$--module structure. It is finite projective as $A$--module by \cite[4.4.1]{Ford}. Viewing $B$ as a subalgebra of $\End_R(P)$  the double centralizer theorem implies
\begin{equation}  \label{dbi-3}
 B \cong \End_A(P).
\end{equation}
Hence $\idemp\big(\End_A(P)\big) = \{0,1\}$ which says that $P$ is an indecomposable $A$--module by the characterization of indecomposability mentioned in  \ref{ifpm}.

\begin{cor}\label{dbic}
  Let $R$ be a connected ring. Then every Brauer class $\al \in \Br(R)$ has a representative $B$ satisfying $\idemp(B) = \{0,1\}$.
\end{cor}

\begin{proof}
  By Theorem~\ref{dbi}, it suffices to show that for every Azumaya $R$--algebra $A$ there exists an indecomposable finite projective $A$--module. This can be seen by decomposing the finite ``projective'' left $A$--module $_AA$.
  \lv{
  (likely folklore):
The left $A$--module $_A A$ is finite projective. If $A$ in not indecomposable = decomposable, then $A=I_1 \oplus I_2$ where $I_1$ and $I_2$ are left ideals satisfying $0 \ne I_j \ne A$ for $j=0,1$. Hence $I_j$ are finite projective $A$--modules and hence also as $R$--modules with $0 < \rank_R I_j < \rank_R A$. Continuing with decomposing the left ideals if necesary, we arrive after a finite number of steps at an indecomposable finite projective $A$--module $P$.
}
\end{proof}

\begin{lem}\label{diba} Let $R$ be a connected ring and let $A$ be an Azumaya $R$--algebra. Then the following conditions are equivalent:
\begin{enumerate}[label={\rm (\roman*)}]
  \item \label{diba-i} Any two indecomposable finite projective $A$--modules are equivalent.

  \item \label{diba-ii} Up to isomorphism, the Brauer class of $A$ contains a unique $B$ satisfying $\idemp(B) = \{0,1\}$.
\end{enumerate}
If these conditions hold, then, for $B$ as in \ref{diba-ii}, we have  that
\begin{enumerate}[label={\rm (\roman*)}]\setcounter{enumi}{2}

 \item \label{diba-iii} $\deg(B) = \ind([A])$, in particular $\deg(B) \big| \deg (A)$. \sm

 \item \label{diba-iv} Up to equivalence, the left $B$--module $_BB$ is the unique indecomposable finitely generated projective left $B$--module. \sm

 \item \label{diba-v} In addition to \ref{diba-i} and \ref{diba-ii},  suppose $\Pic(R) = \{0\}$, and let $M$ and $N$ be finite projective $A$--modules satisfying $\rank_R M \ge \rank_R N$. Then $N$ is a homomorphic image of $M$.
\end{enumerate}
\end{lem}

\begin{proof}
  The equivalence \ref{diba-i}$ \iff$ \ref{diba-ii} is a special case of Theorem~\ref{dbi}. \sm

\ref{diba-iii} We can write $A$ as a finite sum of indecomposable finitely generated projective $A$--modules, say $A = P_1 \oplus \cdots \oplus P_n$. By \ref{diba-i} we can assume that there exists an indecomposable finite projective $A$--module $P$ such that $P_i = P \ot_R E_i$ for some invertible $R$--module $E_i$. Thus, $\rank_R P_i = \rank_R P$ and we get
\begin{equation}  \label{diba-1}
\rank_R A = n \rank_R P\quad \text{and} \quad \rank_R A \cdot \rank_R B = (\rank_R P)^2
\end{equation}
where the second equation in \eqref{diba-1} follows from \eqref{dbi-2}. The two equations in \eqref{diba-1} imply
\begin{equation}
  \label{diba-2} \rank_R P = n \rank_R B, \quad \rank_R A = n^2 \rank_R B.
\end{equation}
The second equation then forces $\deg B \big| \deg (A)$.

Let $[A]=\al \in \Br(R)$. Recall that by definition $\ind (\al) = \gcd\{ \deg A' : A' \sim_{\rm Br} A \}$. The equivalence \ref{diba-i}$ \iff$ \ref{diba-ii} shows that the condition~\ref{diba-i} also holds for $A'$. Since $A \sim_{\rm Br} A' \implies B \sim_{\rm Br} A'$, it then follows that $\deg(B) \big| \deg(A')$, and therefore $\deg(B) = \ind(\al)$. \sm

\ref{diba-iv} As observed in \ref{diba-iii}, the condition \ref{diba-i} holds for $B$ replacing $A$. Hence we get \eqref{diba-2} for $A=B$, which forces $n=1$ and therefore \ref{diba-iv}. \sm

\ref{diba-v} Since $\Pic(R)= 0$, ``equivalence'' in \ref{diba-i} becomes ``isomorphism''. We fix an indecomposable finite projective $A$--module $K$, and then get $M \cong K^a$ and $N \cong K^b$ for positive integers $a$ and $b$. The rank assumption on $M$ and $N$ implies that $a\ge b$. The claim then follows.
\end{proof}

\begin{thm} [{\cite[Thm.~1]{DeM}} for $R$ semilocal] \label{LGdeM}
Let $R$ be a connected LG-ring. Then the condition~{\rm \ref{diba}}\ref{diba-i}  holds with ``equivalence" replaced by ``isomorphism''. Hence the conditions \ref{diba-ii}--\ref{diba-v} of\/ {\rm \ref{diba}} are satisfied too.
\end{thm}

\begin{proof} Let $P$ and $Q$ be indecomposable finite projective $A$--modules.
Since $R$ is connected, $\rank_R P$ and $\rank_R Q$ are constant, so that Corollary~\ref{prop_baireco} says that $P$ and $Q$ are free as $R$--modules. Without loss of generality we can therefore suppose that $\rank_R P \ge \rank_R Q$. We claim:
\begin{equation} \label{LGdeM1}
\text{$Q$ is a homomorphic image of $P$.}
\end{equation}
Assuming \eqref{LGdeM1} for a moment, we can quickly finish the proof. Indeed, because $Q$ is projective, any epimorphism $P \to Q$ splits. But $P$ is indecomposable. Therefore $P \cong Q$.

It remains to prove \eqref{LGdeM1}. This can be done by observing that \ref{diba}\ref{diba-i} and hence \ref{diba}\ref{diba-v} holds for the localization $R_\gm$ in a maximal ideal $\gm \ideal R$, thanks to \cite[Thm.~1]{DeM}. Applying then  \cite[Thm.~2.6(i)]{EG}, proves the claim. A more direct proof of \eqref{LGdeM1} goes as follows. Since $P$ and $Q$ are free of finite rank as $R$--modules, $\Hom_R(P,Q)$ is an affine space; it contains the open quasi-compact subscheme $U'= \{ \vphi \in \Hom_R(P,Q) : \vphi \text{ is surjective }\}$ (note that
the complement of $U'$ is given by the vanishing of finitely many minors). The $A$--linear maps $\Hom_A(P,Q)$  is  a subspace of $\Hom_R(P,Q)$, hence is again an affine space. It contains $U = \{ \vphi \in \Hom_A(P,Q): \vphi \text{ is surjective}\} = \Hom_A(P,Q) \cap U'$ as an open quasi-compact subscheme. If $\gm\ideal R$ is any maximal ideal of $R$, it is well-known that \ref{diba}\ref{diba-i} and hence \ref{diba}\ref{diba-v} hold for $R$ replaced by the field $R/\gm$, i.e., $U(R/\gm) \ne \emptyset$. But then Proposition~\ref{prop_baire}\ref{prop_baire-a} shows that $U(R) \ne \emptyset$, which is \eqref{LGdeM1}.
\lv{ OLD PROOF:
Let $\gm \ideal R$ be a maximal ideal and let $R_\gm$ be the localization of $R$ at $\gm$. Then $A_\gm = A\ot_R R_\gm \cong A/\gm A$ is an Azumaya $R_\gm$--algebra. Also, $P_\gm = P\ot_A A_\gm \cong P\ot_R R_\gm$ and $Q_\gm$ are finite projective $A_\gm$--modules satisfying $\rank_{R_\gm} P_\gm \ge \rank_{R_\gm} Q_\gm$. By \cite[Thm.~1]{DeM}, the condition \ref{diba}\ref{diba-i} holds for the Azumaya $R_\gm$--algebra $A_\gm$. Hence \ref{diba}\ref{diba-v} says that $Q_\gm$ is a homomorphic image of $P_\gm$. As this holds for every maximal ideal $\gm$ of $R$, we can apply \cite[Thm.~2.6(i)]{EG} with $(R,S, M,N)$ there replaced by $(R_\gm, A_\gm, P, Q)$ here. Thus, $Q$ is a homomorphic image of $P$.}
\end{proof}

\subsection{The Wedderburn property} \label{twp} \new One says that a connected ring $R$ has the {\em Wedderburn property} if for every Azumaya $R$-algebra $A$ there is up to isomorphism a unique representative $B$ in the Brauer class of $A$ such that $\mathrm{idemp}(B) = \{0, 1 \}$ and $B\op\cong  \End_A(M)$ for some $A$-progenerator module $M$, \cite{AW14} or  \cite[\S 7.6]{Ford}. Condition~\ref{diba-ii} of \ref{diba} and Theorem~\ref{LGdeM} say that a connected LG-ring $R$ has the Wedderburn property.  This was previously known for connected semilocal rings \cite[Cor.~1]{DeM}. 

Obvious examples of rings having the Wedderburn property are connected rings for which every Azumaya algebra is a matrix algebra, i.e., connected rings $R$ with $H^1(R, \uPGL_n) = 1$. An example of such a ring is any connected ring $R$ satisfying the two conditions \ref{twbi} and \ref{twbii} below, 
\begin{enumerate}[label={\rm (\roman*)}]
  \item \label{twbi} $\Br(R) = 1$ and 
  \item\label{twbii} any finitely generated projective $R$--module (necessarily of constant rank $n$) is free, i.e., $H^1(R, \uGL_n) = 1$. 
\end{enumerate}
Indeed, let $A$ be an Azumaya $R$--algebra of degree $n$. Since $[A]=0\in \Br(R)$, the algebra $A$ is Brauer equivalent to $R$ so that $A \cong \End_R(P)$ for a projective $R$--module $P$ of rank $n$ by \cite[Prop.\ III.5.6]{KO}. Since $P$ is free by \ref{twbii}, we conclude that $A\cong \Mat_n(R)$. Thus $H^1(R,\uPGL_n)=1$. 

For example, any PID satisfies condition~\ref{twbii}, while condition~\ref{twbi} is fulfilled for $R=\ZZ$ (\cite[Prop.~2.4]{Gro}, see also \cite[14.3.8]{Ford})
and $R=k[X]$, $k$ an algebraically closed field (see for example \cite[7.5.1, 13.6.2(a)]{Ford}). 
\enew 
\sm

Examples of rings that do not have the Wedderburn property are given in \cite[Ex.~1.5]{AW14}, \cite[page 46]{Bass} and \cite{Childs}.
\ms

\new
The existence of a Brauer decomposition, which we address in Corollary~\ref{raco} for LG-rings, is established over fields in \cite[Prop.~2.8.13]{GiS}. 
\enew

\begin{cor}[Brauer Decomposition for LG-rings] \label{raco}
Let $R$ be a connected LG-ring, let  $A$ be an Azumaya $R$--algebra with $\idemp(A)  = \{0,1\}$, and let $\ind([A]) = p_1^{m_1} \dots p_c^{m_c}$ be the prime factor decomposition of $\ind([A])$. Then there exist Azumaya $R$--algebras $B_1, \ldots, B_c$ satisfying
 \begin{equation}\label{raco0}
   \deg(B_i) = p_i^{m_i}\; \text{ and } \; \idemp(B_i) = \{1,0\}
 \end{equation}
 for $i=1, \ldots, c$,  as well as
\begin{equation}\label{raco1}  A \cong B_1 \ot_R \cdots \ot_R B_c.
\end{equation}
The conditions \eqref{raco0} and \eqref{raco1} determine the family $(B_1, \ldots, B_c)$ up to isomorphism.
\end{cor}

\begin{proof} \new
We first prove the existence of a Brauer decomposition writing $\al = [A]$ and $\per(\al) = |\lan \al \ran |$, the order of the finite subgroup of $\Br(R)$ generated by $\al$. Let $\per(\al) = q_1^{\ell_1} \cdots q_d^{\ell_d}$ be the prime factor decomposition or $\per(\al)$, i.e., $q_1, \ldots, q_d$ are distinct primes and $\ell_i \in \NN_+$. The prime factor decomposition in finite cyclic subgroups yields $\al = \al_1 + \cdots + \al_d$ with $\al_i \in \lan \al \ran$ and $\per(\al_i) = q^{\ell_i}$. By \cite[Thm.~6]{AW15}, $\per(\al_i)$ and $\ind(\al_i)$, have the same prime divisors, so that $\ind(\al_i) = q_i^{n_i}$ for some $n_i \in \NN$. But by \cite[Thm.~3]{AW17} we have   
\begin{equation}\label{raco2}
  \ind(\al) = \ind(\al_1) \cdots \ind(\al_d) = q_1^{n_1} \cdots q_d^{n_d}, 
\end{equation} 
and therefore $d=c$, $q_i = p_i$ and $n_i = m_i$ for $1 \le i \le c$. 

Next, by \ref{LGdeM}, we can apply \ref{diba}. Thus, there exist Azumaya algebras $B_i$ satisfying $[B_i] = \al_i$, $\idemp(B_i) = \{0,1\}$ and $\deg(B_i) = \ind(\al_i) = p_i^{m_i}$ for $i = 1, \ldots, c$, proving \eqref{raco0}. 
(The existence of such Azumaya algebras is a special case of Gabber's Theorem \cite[II, Thm.~1]{Ga}). Observe that $\deg(A) = \ind(A)$ by \ref{diba}\ref{diba-iii}. 
 We then obtain 
\begin{align*}
 & \deg (B_1 \ot_R \cdots \ot_R B_c) = \textstyle \prod_{i=1}^c \deg(B_i)
  \\& \quad = \textstyle \prod_{i=1}^c \ind(\al_i)
   \overset{\eqref{raco2}}{=}  \ind(\al) = \deg(A)
\end{align*}
where $(*)$ holds by \cite[Thm.~3]{AW17}. Similarly,
\begin{align*}
  [A] &= \al = \al_1 + \cdots + \al_c = [A_1] + \cdots + [A_c]
  \\ &= [B_1] + \cdots + [B_c] = [B_1 \ot_R \cdots \ot_R B_c],
\end{align*}
so that Corollary~\ref{hil}\eqref{hil-a} yields the isomorphism \eqref{raco1}. 
\enew
 
Unicity: Let $(B'_1, \ldots, B'_c)$ be a family of Azumaya $R$-algebras satisfying \eqref{raco0} and \eqref{raco1}. By  \ref{diba}\ref{diba-iii} we then get $\ind([B'_i]) = \deg(B_i') = p_i^{m_i}$, thus $\per([B_i])= p_i^{\ell_i}$ for some $\ell_i \le m_i$ (\cite[Thm.~3]{AW15}). It follows that $[A] = [B_1'] + \cdots + [B_c']$ is the decomposition of $[A]$ into primary components. Hence $[B'_i] = [B_i]$ for $i=1, \ldots, c$. As $\deg(B'_i) = p_i^{m_i} = \deg(B_i)$, another application of Corollary~\ref{hil}\eqref{hil-a} yields $B'_i \cong B_i$. \end{proof}

\sm

\new
\textbf{Remark on generalization.} The Brauer decomposition holds for arbitrary Azumaya algebras in a modified form. 

Namely, let $A$ be an Azumaya algebra over a connected LG-ring. By \ref{dbi} and {\rm \ref{LGdeM}} we know $A=A_0 \ot_R \Mat_n(R)$ for a unique (up to isomorphism) Azumaya algebra $A_0$ with $\idemp(A_0) = \{0,1\}$ and a unique $n\in \NN$. We then apply \ref{raco} to $A_0$ and get $B_1, \ldots, B_c$ satisfying \eqref{raco0} and \eqref{raco1}. Hence $A \cong \Mat_n(R) \ot_R B_1 \ot_R \cdots \ot_R B_c$. 
\enew
\ms

The results in \ref{dbi}--\ref{raco} show the importance of Azumaya algebras $A$ satisfying $\idemp(A) = \{0,1\}$. While this is a purely algebraic condition, it is perhaps not surprising that these algebras also have a natural group theoretic characterization, which in fact holds beyond the case of connected LG-rings. We will prove this in \ref{aniso}, using the following Lemmas~\ref{anisole} and \ref{anid} which (in our view) are of independent interest.

For an Azumaya algebra $\calA$  over a scheme $S$, the $S$--group scheme $\uGL_1(\calA)$ is for example defined in \cite[2.4.2.2]{CF}.

\begin{lem}\label{anisole} Let $\calA$ be an Azumaya algebra over a {\em connected\/} scheme $S$, and let $e$ be an idempotent of $A = \calA(S)$. 
\sm

\begin{inparaenum}[\rm (a)] \item \label{anisole-a} Then
\begin{equation} \label{ansiole1} \begin{split}
\la_e \co \GG_{m,S} &\longto \, \uGL_1(\calA)=: G, \\
   t(\in \GG_{m,S}(T))  &\;\mapsto \; t e_T + t\me (1_{\calA(T)} - e_{T})
\end{split} \end{equation}
($T$ an $S$--scheme) is a cocharacter such that
\begin{equation}\label{anisole2}\begin{split}
 \text{$\la_e$ is central} \quad &\iff \quad e\in \{0, 1\} \\
 &\iff \quad G=\rmP_G(\la) = \rmC_G(\la).
\end{split}\end{equation}

\item \label{anisole-b} $e=0 \iff e_T = 0$ for some $S$--scheme $T\ne \emptyset$; analogously for $e=1$.
\end{inparaenum}\end{lem}

\begin{proof} \eqref{anisole-a} It is immediate that $\la_e$ is a cocharacter, which is central if $e\in \{0,1\}$.
In general, let $\calA_{ij}$, $i,j\in \{1,0\}$  be the Peirce spaces of $\calA$ with respect to the orthogonal idempotents $e_1 = e$ and $e_0 = 1-e$. Thus $\calA_{ij} = e_i \calA e_j$. The centralizer of $\la_e$ in $\calA$ is $\rmC_{\calA}(\la_e) = \calA_{11} \oplus \calA_{00}$.

Assume now that $\la_e$ is central. Then $\calA = \rmC_{\calA}(\la) = \calA_{11}\times \calA_{00}$. It follows that $\calA_{11}$ and $\calA_{00}$ are ideals of $\calA$. They induce ideals $\calI_1$ and $\calI_0$ of $\calO_S$ by $\calI_i = \calA_{ii} \cap \calO_S$, $i=1,0$,  which satisfy $\calA_{ii} = \calI_i \calA$. Since $\calA = \calA_{11} \times \calA_{00}$ we have $\calO_S = \calI_1 \times \calI_0$. Because $S$ is connected, either $\calI_i = \calO_S$, i.e., $e=1$, or $\calI_1 = 0$, i.e., $e = 0$.

Clearly, $\la_e$ is central $\iff G=\rmC_G(\la)$. Since $\rmC_G(\la) \subset \rmP_G(\la) \subset G$ in general, this proves the second equivalence. \sm

\eqref{anisole-b} The parabolic subgroup $\rmP_G(\la_e)$ has constant type since $S$ is connected [XXVI, 3.3]. Hence, if $e_T = 0$, then $\rmP_G(\la_e) (T) = G(T)$. Because $\rmP_G(\la_e)$ has constant type, we get $\rmP_G(\la) = G$ and therefore $e\in \{0,1\}$ by \eqref{anisole2}. The assumption $e=1$ contradicts $e_T = 0$. Thus $e=0$. The other direction is obvious. \end{proof}

\begin{lem} \label{anid} Let $\calA$ be an Azumaya algebra over a {\em connected\/} scheme $S$, and suppose $\la \co \GG_{m,S} \to \GL_1(\calA)$ is a non-central cocharacter. Then $\{0,1\} \subsetneq \idemp\big(\calA(S)\big)$.
\end{lem}

\begin{proof} Let $\calB = \calA^\la$ be the subalgebra of $\calA$ centralizing the image of $\la$, and let $\calC$ be the centre of $\calB$. By \cite[Prop.~3.2]{G3} we know that $\calC$ is a finite \'etale $\calO_S$--algebra of positive rank. We will eventually show that $\calC(S)$ contains a nontrivial idempotent $e$ defined by using the decomposition of $\calC$ into its connected components. But first we need to eliminate the case that $\calC$ is connected.

By construction of $\calC$, the image of $\la$ lies in $\calC$, and this also holds for the central cocharacter $\De$. Thus  $\la$, $\De \in \Hom_{S-{\rm gr}}(\GG_{m,S}, \uGL_1(\calC))$. To compare the two cocharacters, let $C$ be the scheme associated with the finite \'etale $\calO_S$--algebra $\calC$ and let $\frR_{C/S}(\cdot)$ be the Weil restriction. Then $\GL_1(\calC) = \frR_{C/S}(\GG_m,C)$, $(\GG_{m,S})_C = \GG_{m,C}$ and the fundamental identity of the Weil restriction $\frR_{C/S}$ becomes (after a canonical identification)
\[ \Hom_{S-{\rm gr}}\big(\GG_{m,S}, \frR_{C/S}(\GG_{m,C})\big) =
   \Hom_{C-{\rm gr}}(\GG_{m,C}, \GG_{m,C}).\]
Suppose that $C$ is connected. Then $\Hom_{C-{\rm gr}}(\GG_{m,C}, \GG_{m,C})\cong \ZZ$ with basis $\De$. Therefore $\la = \De^n$ for some $n\in \ZZ$, in particular $\la$ is central, contradiction. Therefore $C$, equivalently $\calC$, is not connected.

Let $\calC_1 \times \cdots \times \calC_m$ be the decomposition of $\calC$ corresponding to the decomposition of $C$ into its connected components, and let $e_1, \ldots, e_m$ be the identity elements of the algebras $\calC_i(S)$. Then the $e_i$ are non-zero idempotents of $\calA$ since $\calC_i(S) \ne 0$, and $e_1 \ne 1$ because $m \ge 2$.
\end{proof}

We use the concepts of (ir)reducible and (an)isotropic group schemes as defined in \ref{rev-isog}. For an Azumaya algebra $\calA$  over a scheme $S$ the $S$--group scheme $\uSL_1(\calA)$ is for example defined in \cite[3.5.0.91]{CF}; it is
semisimple by \cite[3.5.0.92]{CF}. The $S$--group scheme $\uPGL(\calA)$ is a semisimple $S$--group scheme by \cite[3.0.5.82]{CF}.

\begin{prop}\label{aniso} Let $\calA$ be an Azumaya algebra over a {\em connected\/} scheme $S$. The following are equivalent:
\begin{enumerate}[label={\rm (\roman*)}]
  \item \label{aniso-i} $\uSL_1(\calA)$ is isotropic,

  \item \label{aniso-ii} $\uSL_1(\calA)$ is reducible,

  \item \label{aniso-iii} $\uGL_1(\calA)$ is reducible,

  \item \label{aniso-v} $\uPGL(\calA)$ is reducible,

  \item  \label{aniso-vi} $\uPGL(\calA)$ is isotropic,

  \item \label{aniso-iv} $\idemp\big(\calA(S)\big) \supsetneq \{0,1\}$.

\end{enumerate}
\end{prop}

\begin{proof}
By \ref{rev-isog}\ref{rev-isog-a}, the equivalences \ref{aniso-i} $\iff$ \ref{aniso-ii} and \ref{aniso-v} $\iff$ \ref{aniso-vi} hold since $\uSL_1(\calA)$ and $\uPGL(\calA)$ are semisimple $S$--group schemes. The equivalences \ref{aniso-ii} $\iff$ \ref{aniso-iii} $\iff$ \ref{aniso-v} are special cases of \cite[3.5.3(b)\footnote{ The reference refers  to the web version of the article:
https://hal.science/hal-01063601v2/document. The published version does not distinguish between $H$ and $H^{\rm ss} = H/\rad(H)$.} and 3.2.1(2)]{G2}.

If $e$ is an idempotent as in \ref{aniso-iv}, then $\la_e$ is not central by \eqref{ansiole1}. Hence $G=\uGL_1(\la_e) \ne \rmC_G(\la_e)$ and a fortiori $G \ne \rmP_G(\la_e)$, which proves \ref{aniso-iii}.

Conversely, if \ref{aniso-iii} holds, then $G$ contains a proper parabolic subgroup $P$ with a Levi subgroup $L$. By \cite[7.3.1(1)]{G2} we have $(P,L) = (\rmP_G(\la), \rmC_G(\la))$ for a non-central cocharacter $\la$. Thus \ref{aniso-iv} follows from Lemma~\ref{anid}.  \end{proof}

\begin{cor}[Anisotropic kernel of $\uGL_1(A)$] \label{ako} Let $R$ be a connected LG-ring and let $A$ be an Azumaya $R$--algebra. By\/ {\rm \ref{dbi}} and {\rm \ref{LGdeM}} there exist unique $n\in \NN$ and an Azumaya algebra $B$ with $\idemp(B) = \{0,1\}$ such that $A= \Mat_n(B)$. Then
a minimal parabolic subgroup $P_{\min}$ of $G= \uGL_1(A)$ and a minimal Levi subgroup $L_{\min}$ are
      \[ P_{\min} = \begin{pmatrix} \uGL_1(B) & * & * \\
                                    0 &  \ddots & *  \\  0 & 0 &  \uGL_1(B)
      \end{pmatrix}, \quad
      L_{\min} = \begin{pmatrix} \uGL_1(B) & 0 & 0 \\
                                    0 &  \ddots & 0  \\  0 & 0 &  \uGL_1(B)
      \end{pmatrix}. \]

\noindent   \item \label{ako-ii} A maximal split torus $T_0$ of $G$ and an anisotropic kernel $G\an$ of $G$ are \[ T_0 = \begin{pmatrix} \GG_m & 0 & 0 \\
                                    0 &  \ddots & 0  \\  0 & 0 &  \GG_m
      \end{pmatrix}, \qquad
   G\an = \begin{pmatrix} \uSL_1(B) & 0 & 0 \\
                                    0 &  \ddots & 0  \\  0 & 0 &  \uSL_1(B)
      \end{pmatrix}. \]
\end{cor}

\begin{proof}
  It is clear that $T_0$ as displayed above is a split torus whose centralizer in $G$ has the form of the matrix group $L_{\min}$. The quotient $L_{\min}/T_0$ can be identified with
  \[ \begin{pmatrix} \uPGL(B) & 0 & 0 \\
                                    0 &  \ddots & 0  \\  0 & 0 &  \uPGL(B)
      \end{pmatrix}
  \]
which is anisotropic by \ref{aniso}. Example~\ref{aniex} can therefore be applied and yields that the matrix groups displayed above are a minimal parabolic subgroup, a minimal Levi subgroup,  a maximal split torus and an anisotropic kernel respectively. \end{proof}

\appendix
\section{Maximal split subtori of groups of multiplicative type}\label{sec:app-maxiss}

In this appendix, we study maximal split and maximal locally split $S$-subtori of groups of finite multiplicative type. We prove their existence in \ref{lem_torus}, and show in the Example~\ref{ex_tori}(b) that the two notions are different in general.

\subsection{Basic notions, review} \label{bnr}
We mainly use the terminology and notation of \cite[VIII, IX]{SGA3}, as before abbreviated by [VIII] and [IX], see also \cite[App.~B]{Co1} and \cite{O}. Throughout this appendix, $S$ is an arbitrary scheme.

Given an abelian group $M$, one denotes by $\rmD_S(M) = \Spec(\calO_{\calS}[M]))$ the abealin $S$--group scheme representing the functor of characters of the constant $S$--group $M_S$. An $S$--group scheme $G$ is called {\em split}, equivalently {\em diagonalizable\/}, if $G$ is isomorphic to an $S$--group $\rmD_S(M)$ for some abelian group $M$. A possible confusion with the notion of a split reductive $S$--group cannot occur, if the reader takes the context into account. A {\em split $S$--torus\/} is an $S$--group $T$ isomorphic to $\rmD_S(\ZZ^r)$ for some $r\in \NN_+$. We call an $S$--group scheme $G$ {\em locally split} if there exists a family $(S_i)_{i\in I}$ of open subschemes $S_i$ of $S$ such that $S= \bigcup_{i\in I} S_i$ and every $G|_{S_i} = G\times_S S_i$ is diagonalizable. If in this case all $G|_{S_i}$ are split tori, we call $T$ a {\em locally split torus\/}.

An $S$--group $G$ has {\em multiplicative type\/} if $G$ is locally diagonalizable for the fpqc topology, i.e., for every $s\in S$ there exists an open affine neighbourhood $U$ of $s$ in $G$ and a finite family $(X_i)_{i\in I}$ of affine schemes together with flat morphisms $f_i \co X_i \to U$ such that $U = \bigcup_i f_i(X_i)$ and every $G|_{X_i}$ is diagonalizable. In this case, we call $G$ a {\em torus\/}, if all $G_{X_i}$ are split tori. A group of {\em finite multiplicative type\/} is a group of multiplicative type which is also of finite type.

Let $G$ be a group of multiplicative type. For $s\in S$ let $G_s$ be the fibre of $s$. There exists an extension $k$ of the residue field of $s$ such that $(G_s)_k$ is diagonalizable, say $\cong \rmD_S(M_s)$ for some abelian group $M_s$. The isomorphism class of $M_s$ is independent of the choice of $k$, and called the {\em type\/} of $G$ at $s$, [IX; 1.4]. The function on the underlying topological space $S_{\rm top}$ of $S$, associating with $s\in S$ the type of $G$ in $s$, is locally constant. Assume now that $G$ has finite multiplicative type. For every $s\in S$ the type of $G$ is a finitely generated abelian group $M_s$. Hence we get a well-defined, locally constant function $\rank (G)\co S_{\rm top} \to \NN$,
associating with $s\in S$ the rank of the finitely generated abelian group $M_s$. It follows that we obtain a partition $S = \bigsqcup_{r \geq 0} S_r$ such that every $S_r$ is an open and closed subscheme of $S$ and $G|_{S_r}$ is of constant rank $r$. If $G=T$ is a  torus, this is the so-called {\em partition by type}, \cite[5.4]{O}.

A {\em subtorus\/} of a group $G$ of multiplicative type is a monomorphism $T\to G$ where $T$ is torus. By [IX; 2.5] or \cite[B.1.3]{Co1} in the finite type case, $T\to G$ is a closed immersion. Obviously, $T$ is called  a {\em split subtorus\/} or a {\em locally split subtorus}, if $T$ is a split torus or a locally split torus. We say $T$ is a {\em maximal split subtorus\/} or {\em maximal locally split subtorus}, if the image of $T$ in $G$ is maximal with respect to inclusion among all split subtori or locally split subtori respectively.

\begin{lem} \label{lem_quo} Let $T$ be a torus and let $Q$ be the quotient of $T$ by a subgroup of multiplicative type. Then $Q$ is a torus. Moreover,

\begin{enumerate}[label={\rm (\roman*)}]
  \item \label{lem_quo-i} if $T$ is locally split, then $Q$ is a locally split torus.

  \item \label{lem_quo-ii} If $S$ is connected and $T$ is a split torus, then $Q$ is a split torus with $\rank Q \le \rank T$.
\end{enumerate}
\end{lem}

\begin{proof} We first prove \ref{lem_quo-i}. By definition of ``locally split'', it suffices to show:
\begin{equation}\label{lem_quo1}
  \text{\em If $T$ is a split torus, then $Q$ is a locally split torus. }
\end{equation}
For the proof of \eqref{lem_quo1} we write $T$ in the form $T=\rmD_S(M)$ with $M \cong \ZZ^r$ and observe that any torus is a group of finite multiplicative type, so that we can apply [IX; 2.11(i)]: the group $Q$ is a locally split group of multiplicative type. Hence, Zariski-locally $Q = \rmD_S(N)$ for some abelian group $N$. Cartier duality [VIII; \S3] provides a monomorphism $N_S \to M_S$ of constant group schemes, so that $N$ is torsion free and finitely generated. Therefore, $N \cong \ZZ^{r'}$ with $r'\le r$, in particular $Q$ is a split $S$--torus, proving \eqref{lem_quo1} and thus also \ref{lem_quo-i}.
\sm

We next show \ref{lem_quo-ii} by modifying the proof of \eqref{lem_quo1}. Indeed, the reference [IX; 2.11(i)] also says that $Q$ is a split group of multiplicative type if $S$ is connected. Hence, as the proof of \ref{lem_quo-i} shows, $Q$ is a split torus. \sm

Finally, we can prove the general case: {\em If $T$ is a torus, then $Q$ is a torus.} Being a torus is local for the fpqc topology. We can therefore assume that $T$ is a split torus. Then $Q$ is locally split by \ref{lem_quo-i}, which is all we needed to show.
\end{proof}

\begin{example}\label{lem_quo-ex}  By \cite[B.3.3]{Co1}, any fppf-closed subgroup  of a group $G$ of multiplicative type is a group of multiplicative type. Hence, Lemma~\ref{lem_quo} holds for fppf-closed subgroups of tori.

For example, by \ref{lem_quo}\ref{lem_quo-ii}, {\em if $S$ is a connected scheme and if $\la \co \GG_{m,S} \to H$ is a non-trivial group homomorphism, then $\GG_{m,S}/ \Ker(\la) \cong \GG_{m,S}$.}
\end{example}

The following lemma is the first step towards the existence of a maximal locally split subtorus.


\begin{lem} \label{mul} Let $G$ be an $S$--group of finite multiplicative type.
\sm

\begin{inparaenum}[\rm (a)]
\item \label{mul-b} Then the family of locally split $S$--subtori of $G$
  is a directed poset with respect to inclusion. \sm

\item\label{mul-c} If $S$ is connected, the family of split $S$--subtori of $G$ is a directed poset with respect to inclusion.
\end{inparaenum}\end{lem}

\begin{proof} \eqref{mul-b} It suffices to show:
\begin{enumerate}[label={\rm (\roman*)}]
\item \label{mul-bi} {\em If $E_1$ and $E_2$ are locally split subtori of $G$, then there exists a locally split subtorus of $G$ containing $E_1$ and $E_2$.}
\end{enumerate}
To prove this, we consider the $S$--group scheme $E_1 \times_S E_2$ and claim that
$E_1\times_S E_2$ is a locally split torus. Indeed, Zariski-locally $E_1$ and $E_2$ have the form $\rmD_S(M_1)$ and $\rmD_S(M_2)$ for free abelian groups of finite type. Since for arbitrary abelian groups $N_1$ and $N_2$ we have
\begin{equation}
  \label{mul-b1} \rmD_S(N_1) \times_S \rmD_S(N_2) \cong \rmD_S(N_1 \times N_2),
\end{equation}
see e.g.~\cite[5.1]{O}, it follows that $E_1 \times_S E_2$ is a locally split torus.

Let $h \co E_1 \times_S E_2 \to G$ be the group homomorphism given by multiplication. By [IX; 2.7], its kernel $\Ker(h)$ is a subgroup of multiplicative type, so that $E_3 := (E_1 \times E_2)/\Ker(h)$ is a locally split torus by \ref{lem_quo}\ref{lem_quo-i}. Moreover, again by [IX; 2.7], the canonical map $E_3 \to G$ is a monomorphism. Clearly, $E_1 \cap \Ker(h) = \{0\}$. Therefore, $E_1 \to E_3$ is a closed immersion, and the same holds for $E_2 \to E_3$. Thus, \ref{mul-bi} holds. \sm

The proof of \eqref{mul-c} is a straightforward modification of the proof of \eqref{mul-b}, see [XXVI, 6.5]: If $E_1$ and $E_2$ are split tori, then so is $E_1 \times_S E_2$ by \eqref{mul-b1}. The quotient $E_3$ is then a split torus by \ref{lem_quo}\ref{lem_quo-ii}.
\end{proof}

\begin{lem}\label{lem_quot}
\begin{inparaenum}[\rm (a)]
 \item \label{lem_quot2}  A monomorphism $T   \hookrightarrow T'$ between
 $S$--tori with the same rank functions is an isomorphism.

\item \label{lem_quot3}  Assume that $S$ is connected and that $G$ is an $S$--group of finite multiplicative type. Then every family of $S$--subtori of $G$, which is a directed poset with respect to inclusion, admits a unique maximal element.
\end{inparaenum}
\end{lem}

\begin{proof} We can assume that $S$ is non-empty. All statements are local for the fpqc topology, allowing us to deal with split $S$-tori.
\sm

\eqref{lem_quot2} We are given a monomorphism  $f\co T=\rmD_S(M) \to \rmD_S(M')=T'$, where $M$ and $M'$ are free $\ZZ$--modules of rank $r$. Cartier duality provides an epimorphism $\wdh f\co  M'_S \to M_S$ of constant $S$--schemes.
For each  point $s$ of $S$, we obtain a surjective map $\wdh f_{\ol s}\co M'_{\ol s} \to M_{\ol s}$ between free $\ZZ$-modules
of same finite rank. The map $\widehat f_{\ol s}$ is then an isomorphism, and hence so is $\wdh f$ by [IX, 2.9], which in turn implies that $f$ is an isomorphism. 
\sm

\eqref{lem_quot3} Let $(T_i)_{i\in T}$ be a family as in the statement of \eqref{lem_quot3}. Since $S$ is connected, $G$ has of constant rank, say rank $r$. By the same reason, each $T_i$ has constant rank, say $\rank T_i = r_i$. Note $r_i \leq r$. Let $r_j$ be maximal among all $r_i$, $i\in I$. We claim that $T_j$ is a maximal element of the given family. Indeed, fix $i\in I$ and let $T_k$ be a member of the family such that $T_i \subset T_k$ and $T_j \subset T_k$. We have $r_j = r_k$ by maximality of $r_j$. Applying \eqref{lem_quot2} yields $T_j= T_k$, hence $T_i \subset T_j$. Thus $T_i$ is a maximal element; unicity is obvious.
\end{proof}

We can now prove the existence of maximal (locally) split subtori.

\begin{prop}\label{lem_torus} Let $G$ be an $S$--group of finite multiplicative type.

\begin{enumerate}[label={\rm (\alph*)}]
  \item \label{lem_torus1} If $E$ is a maximal locally split $S$--subtorus of $G$, then $E$ is unique.

  \item \label{lem_torus2} If $S$ is connected, $G$ admits a unique
   maximal locally split $S$--subtorus.

 \item \label{lem_torus3} $G$ admits a maximal split $S$--subtorus.

 \item \label{lem_torus4} If $S$ is connected,
  $G$ admits a unique maximal split $S$--subtorus.
\end{enumerate}
\end{prop}

\begin{proof}
Lemma~\ref{mul}\eqref{mul-b} 
tells us that the family of  locally split $S$--subtori of $G$ is a directed poset, implying \ref{lem_torus1}, while Lemma~ \ref{lem_quot}\eqref{lem_quot3} shows the existence of a maximal element in case $S$ is connected. Hence \ref{lem_torus2} holds.\sm

\ref{lem_torus3} We use the partition $S = \bigsqcup_{r \geq 0} S_r$ by rank, see \ref{bnr}.  Up to localizing at some $S_r$, we can then assume that $G$ has constant  rank $r$ and that $S \not = \emptyset$. Since an ascending chain of split tori of bounded rank is stationary, it follows from Lemma \ref{lem_quot}\eqref{lem_quot2} that $T$ admits a maximal split $S$--subtorus. \sm

\ref{lem_torus4} The reasoning is the same as for \ref{lem_torus2}, using this time Lemma \ref{mul}\eqref{mul-c}. \end{proof}

Proposition~\ref{lem_torus} implies that a torus $T$ over a connected $S$ admits both a maximal split subtorus $T_0$ and a maximal locally split subtorus $T_{l0}$. Clearly, $T_0 \subset T_{l0}$.
The following Proposition~\ref{prop_normal} gives a sufficient criterion for the two subtori to coincide, while the Examples~\ref{ex_tori} show that in general the two tori do not coincide. \sm

In Proposition~\ref{prop_normal} we denote by $\rmX_*(T)=\Hom_{S-{\rm gr}} (\GG_m, T)$ the commutative $S$--group of cocharacters of a torus $T$; it is a locally constant $S$--group scheme. 
We also recall \cite[Tag 033N]{Stacks}: a connected, normal  locally noetherian scheme is integral and hence has a field of fractions.

\begin{prop} \label{prop_normal} Assume that $S$ is connected, normal and locally noetherian with field of fractions $K$. By Lemma~{\rm \ref{lem_torus}\ref{lem_torus4}},  the $S$--torus $T$ has a unique maximal split $S$--subtorus, denoted $T_0$.

\begin{enumerate}[label={\rm (\alph*)}]
\item \label{prop_normal1} We have
$\rmX_*(T_0) (K) = \rmX_*(T_0)(S)=\rmX_*(T)(S)=\rmX_*(T)(K)$. \sm

\item \label{prop_normal2} $T_0$ is the unique maximal locally split  $S$--subtorus of $T$. \sm

\item  \label{prop_normal3} The formation of $T_0$ commutes with Zariski localization.
\end{enumerate}
\end{prop}


\begin{proof}
We shall use that $T$ is isotrivial [X, 5.16], that is, there exists a finite \'etale cover $S'$ of $S$ which splits $T$. Grothendieck's Galois theory \cite[V, \S8]{SGA1} (or see footnote~\ref{footnote1}) allows us to assume that the cover $S'$ is a connected Galois cover over $S$. We denote the Galois group of $S'/S$ by $\Gamma$; this is also the Galois group of $K'/K$, where $K'$ is the fraction field of $S'$.
\sm

\ref{prop_normal1} We have $\rmX_*(T)(S')= \rmX_*(T)(K')$; by taking $\Gamma$-invariants, we obtain $\rmX_*(T)(S)=\rmX_*(T)(K)$ and, similarly, $\rmX_*(T_0)(S)=\rmX_*(T_0)(K)$. We have the obvious inclusion $\rmX_*(T_0)(S) \subset \rmX_*(T)(S)$; so it remains  to establish the other inclusion.
To do so, we are given $\lambda: \GG_{m,S} \to T$. Its image $E$ is a split $S$-torus, so that Lemma \ref{mul}\eqref{mul-c} implies that $E \subset T_0$. It follows that $\lambda$ factorizes through $T_0$, so that $\lambda$ comes from $\rmX_*(T_0)(S)$. \sm

\ref{prop_normal2} Let $T_{l0}$ be the maximal
locally $S$--split subtorus of $T$. Clearly $T_0$ is the maximal $S$-split subtorus of $T_{l0}$. Since $T_{l0,K}$ is split, applying \ref{prop_normal1} to $T_{l0}$, shows that $\rmX_*(T_0)(K) = \rmX_*(T_{l0})(K)$. It follows that $T_0$ and $T_{l0}$ have  the same rank, so that $T_0=T_{l0}$ in view of Lemma~\ref{lem_quot}\eqref{lem_quot2}. \sm

\ref{prop_normal3} This follows immediately from \ref{prop_normal1}. \end{proof}

\begin{remarks} \begin{inparaenum}[(a)] \item That $\rmX_*(T)(S)=\rmX_*(T)(K)$ can be also proven by applying [X, 8.4]. \sm

\item  The argument of Proposition~\ref{prop_normal}\ref{prop_normal2} can be generalized as follows. Suppose that $S$ is an integral scheme with field of fractions $K$ and that $T$ is an $S$--torus.  Let $T_0$ and $T_{l0}$ be the maximal split and locally maximal split subtorus of $T$. If $T_{0,K} = T_{l0,K}$ (this is the crucial assumption), then $T_0 = T_{l0}$.
\lv{
Suppose $S$ is an integral scheme and let $T$ be an $S$--torus. Since $S$ is connected, $T$ has a unique maximal split subtorus $T_0$ and a unique maximal locally split subtorus $T_{l0}$. Clearly $T_0$ is a subtorus of $T_{l0}$. We have $T_0 = T_{l0} \iff \rank T_0 = \rank T_{0l}$, by Lemma~\ref{lem_quot}\eqref{lem_quot2}. Again since $S$ is connected, the rank functions are constant (partition by rank). So it suffices to have equality at one point.

Now let $K$ be the function field of $S$. Then $K$ is the field of fractions of any local ring $\calO_{S,s}$.  Hence the rank of $T_0$ and of $T_{l0}$ is the rank of $T_{0,K}$ and of $T_{l0, K}$. But $T_{l0,K}$ is a locally-split torus over a field, and is therefore split. IF $T_{0,K} = T_{l0, K}$, then in particular they have the same rank, which implies $\rank T_0 = \rank T_{l0}$.

The point is that both $T_{0,K}$ and $T_{l0,K}$ are split tori, but it is in general not clear that $T_{0,K} = T_{l0, K}$. }
\end{inparaenum}
\end{remarks}

\begin{examples} \label{ex_tori} 
Example~\eqref{ex_tori-a} below shows that \ref{prop_normal}\ref{prop_normal3} is wrong without the normality assumption, while Example~\eqref{ex_tori-b} shows that \ref{prop_normal}\ref{prop_normal2} is wrong without the normality assumption. These examples are taken from \cite[Exc.~2.4.12]{Co1}. \sm

Let $k$ be an algebraically closed field of characteristic zero. For each $n \geq 1$, let $X_n = \Spec(A_n)$  be the $k$-scheme obtained
by gluing $2^n$ affine lines in a loop, with $0$ on the $i$-th line glued to $1$ on the $(i + 1)$-th line ($i \in \ZZ/2^n \ZZ$). This means that
\[
A_n= \bigl\{ (P_i) \in \prod\limits_{i \in \ZZ/2^n \ZZ} k[x_i] \, \mid \,
   P_i(0)=P_{i+1}(1)\enskip \forall \enskip i \in  \ZZ/2^n \ZZ  \bigr\}.
\]
We have then a canonical $k$--map $f_n: \widetilde X_n=\bigsqcup_{i \in  \ZZ/2^n \ZZ} \mathbf{A}^1_k \to  X_n$, obtained by inclusion $A_n \subset \prod_{i \in \ZZ/2^n \ZZ} k[x_i]$. \sm

\begin{inparaenum}[(a)]\item \label{ex_tori-a}
The group $\ZZ/2 \ZZ$ acts on $X_1$ by permuting $x_0$ and $x_1$.
Since $\ZZ/2 \ZZ$ acts freely on $X_1(k) = k \sqcup k/ \sim$, where the equivalence relation $\sim$ identifies $0$ of each  summand  with $1$  of the following summand, the action of $\ZZ/2\ZZ$ on $X_1$ is free \cite[Cor.~ III, 2.2.5]{DG}. The algebraic group  $\ZZ/2 \ZZ \cong \mu_2$ is  diagonalizable, so that the fppf quotient $X= X_1/ (\ZZ/2\ZZ)$ is representable by the affine scheme $X=\Spec(A)$ where $A= (A_1)^{\ZZ/2\ZZ}= \{ P \in k[x] \, \mid \, P(0)=P(1) \}$, [VIII.5.1].
Furthermore, the quotient map $q: X_1 \to X$ is a $\ZZ/2\ZZ$--torsor.
We observe that $X$ is a nodal affine curve. We denote by $a$
its nodal point and note that $X \setminus \{a\} \cong \GG_m$.
Summarizing, we have a cartesian diagram of  $\ZZ/2\ZZ$-torsors
 \begin{equation*} \label{diag_pinch} \vcenter{\xymatrix@C=45pt{
     \mathbf{A}^1_k \sqcup   \mathbf{A}^1_k  \ar[d]^{f_1} \ar[r]^{\widetilde q} &   \mathbf{A}^1_k \ar[d]   \\
     X_1 \ar[r]^q & X \\
}}\qquad .\end{equation*}
Clearly the $\ZZ/2\ZZ$-torsor   $q: X_1 \to  X$ is trivial once restricted to $X\setminus \{a\}$. We claim:
\begin{equation}\label{diag_pinch1}
  \text{\em The $\ZZ/2\ZZ$-torsor   $q\co X_1 \to  X$
is not trivial at the local ring $\calO_{X,a}$.}
\end{equation}
Assume to the contrary that there exists a splitting $s_a\co \Spec(\calO_{X,a}) \to X_1$ of $q$. It extends to a splitting $s\co  U\to X_1$ of $q$ defined on an affine open neighborhood $U$ of $a$ in $X$. Let $V$ be the inverse image of $U$ in $\mathbf{A}^1_k$; this is an affine open neighborhood of $0$ and $1$.
Then $s$ induces a section $\widetilde s$ of $V \sqcup  V \to V$.
Since $\mathbf{A}^1_k$ is irreducible, $V$ is connected, 
hence $\widetilde s$ is the first (or the second) component map.
It follows that $\widetilde s(V)$ contains the two nodal points $a_0,a_1$ of $X_1$, so that $s(a)=a_0=a_1$. This is a contradiction, proving the claim \eqref{diag_pinch1}.

We consider the isotrivial $X$--torus $T=\frR_{X_1/X}(\GG_m)/\GG_m$, that is,
the quotient of the Weil restriction $\frR_{X_1/X}(\GG_m)$ by the diagonal $\GG_m$.
Its isomorphism class is determined by $[X_{1}] \in H^1(X, \ZZ/2\ZZ)$.
Then $T|_{X \setminus \{a\}}$ is split,  but $T \times_S \Spec(\calO_{X,a})$ is not split according to \eqref{diag_pinch1}.
\ms

\item \label{ex_tori-b} We now assume $n \geq 2$. The map $\ZZ/2^{n+1}\ZZ \to \ZZ/2^n\ZZ$, $[i] \mapsto [i]$
induces a morphism
\[ \widetilde g_{n}: \bigsqcup\limits_{i \in  \ZZ/2^{n+1} \ZZ} \mathbf{A}^1_k \to \bigsqcup\limits_{i \in  \ZZ/2^n \ZZ} \mathbf{A}^1_k
\]
which is a $\ZZ/2\ZZ$--torsor, where $\ZZ/2\ZZ$ is the kernel of $\ZZ/2^{n+1}\ZZ \to \ZZ/2^n\ZZ$ and acts by permuting the factors. This map induces a $\ZZ/2\ZZ$--torsor   $g_{n}: X_{n+1} \to X_n$ fitting in the cartesian diagram
\begin{equation*} \label{diag_pinch2} \vcenter{\xymatrix@C=45pt{
     \widetilde X_{n+1}  \ar[d]^{\widetilde f_{n}} \ar[r]^{\widetilde g_{n}} &    \widetilde X_n \ar[d] ^{f_n}  \\
     X_{n+1} \ar[r]^{g_{n}} & X_n \\
}}\quad .\end{equation*}
We claim
\begin{equation} \label{diag_pinch22} \begin{split}
  &\text{\em The $\ZZ/2\ZZ$--torsor $g_n \co X_{n+1} \to X_n$ is locally trivial}
  \\ &\text{\em  for the Zariski topology, but is not trivial.}
\end{split}
\end{equation}
Since a  section of $g_n$ must come from a section of $\widetilde g_{n}$,
 the  argument of \eqref{ex_tori-a} shows that such a section does not exist.
For each $i \in [0, 2^{n}-2]$, we denote by $\widetilde U_i$ (resp.\ $\widetilde V_i$) the open subset of $\widetilde X_n$ (resp.\ $\widetilde X_{n+1}$)
which is $\mathbf{A}^1_k   \sqcup \mathbf{A}^1_k  $ located at  $i$ and $i+1$.
Also, we denote by $U_i$ (resp.\, $V_i$) the fiber product
 \begin{equation*} \label{diag_pinch5} \vcenter{\xymatrix@C=45pt{
   A_n \ar@{^{(}->}[r]  &  \prod\limits_{j \in \ZZ/2^n \ZZ} k[x_j] \\
    k[U_i] \ar@{^{(}->}[u]  \ar@{^{(}->}[r]  & \ar@{^{(}->}[u] k[x_i ] \times k[x_{i+1}]
}}\quad .\end{equation*}
(resp.\ with $k[V_n] \subset A_{n+1}$). We then get a cartesian square of isomorphisms
 \begin{equation*} \label{diag_pinch55} \vcenter{\xymatrix@C=60pt{
     \widetilde V_{i}  \ar[d]^{\widetilde f_{n}}_{\wr} \ar[r]^{\widetilde g_{n}}_{\sim} &
     \widetilde U_i  \ar[d] ^{f_n}_{\wr}  \\
    V_i \ar[r]^{g_{n}}_{\sim} & U_i \\
}}\quad .\end{equation*}
Since the $U_i$'s cover $X_n$, we obtain that $g_n: X_{n+1} \to X_n$
admits  sections for the Zariski topology, establishing the claim \eqref{diag_pinch22}.

We consider  the $X_n$--torus $T_n=\frR_{X_{2n}/X_n}(\GG_m)/\GG_m$ of rank one,
whose isomorphism class is determined by $[X_{2n}] \in H^1(X_n, \ZZ/2\ZZ)$.
Then $T_n$ is isotrivial and \eqref{diag_pinch22} shows that $T_n$ is  locally split of rank 1, but is not split. \end{inparaenum}\end{examples}

\section{Parabolic subgroups of reductive groups via the dynamic method}\label{app:psr}


In this appendix we consider parabolic and Levi subgroups of a reductive  group $H$ over a {\em connected\/} scheme $S$ that can be described by the dynamic method, reviewed in \ref{dps}. The appendix complements \S\ref{sec:minipara} on minimal parabolic subgroups and their Levi subgroups.
\sm

We start with an improvement of \cite[Thm.~7.3.1]{G2}. According to part (2) of that theorem, a reductive $S$--group scheme $H$ over a connected $S$
is isotropic in the sense of \ref{rev-isog} if and only if its radical torus $\rad(H)$ is isotropic or $H$ is reducible, also defined in \ref{rev-isog}. This characterization is a consequence of part (1) of \cite[Thm.~7.3.1]{G2} which proves Theorem~\ref{thm_richardson} without the important invariance of $\la$ under the $S$--group scheme $\Aut(H,P,L)$, the automorphisms of $H$ normalizing $P$ and $L$.

\begin{thm} \label{thm_richardson}
Let $S$ be a connected scheme and let $P$ be a parabolic subgroup of the reductive $S$--group $H$, equipped with a Levi subgroup $L \subset P$. Then there exists a cocharacter $\lambda: \GG_m \to H$ which is fixed by $\Aut(H,P,L)$ and which allows a dynamic description of $P$ and $L$ as $P=\rmP_H(\lambda)$ and  $L=\rmC_H(\lambda)$.
\end{thm}

\begin{proof} Our proof  is a refinement of the proof of \cite[Thm.~7.3.1]{G2}.%

Because $S$ is connected, the reductive group $H$  has constant type [XXII, 2.8], so it is a form of a Chevalley $S$-group $G$ equipped with a Killing couple $(B,T)$ and defined over $\ZZ$. This gives rise to a root datum and a Dynkin diagram $\Delta$. For $I \subset \De$ we let $P_I$ and $L_I$ be the standard parabolic and Levi subgroup associated with $I$. Since triples of type $(H,P,L)$ allow descent, there exists $I \subset \De$ such that $(H,P,L)$ is the twist of $(G, P_I, L_I)$ under the sheaf torsor $E=\underline{\Isom}\bigl((G,P_I,L_I), (H,P,L)\bigr)$.

We now move to the adjoint quotients $H\ad = H/\rmZ_H$ and $G\ad = G/\rmZ_G$ where $\rmZ_H$ and $\rmZ_G$ are the centres of $H$ and $G$ respectively. By \cite[Lem.~3.2.1]{G2}, the pairs $(P\ad, L\ad) = (P/\rmZ_H, L/\rmZ_H)$ and  the analogously defined $(P\ad_I, L\ad_I)$ are  pairs of parabolic subgroups and Levi subgroups of $H\ad$ and $G\ad$ respectively. Using the last line of \cite[Ex.~7.2]{G2},  the original proof of \ref{thm_richardson} shows that $(P\ad_I, L\ad_I)$ has a dynamic description with respect to a cocharacter $\la_I\ad \co \GG_m \to G\ad$, which is fixed by $\Aut(G\ad, P\ad_I, L\ad_I)$. Since $(H\ad, P\ad, L\ad)$ is a twisted form  of $(G\ad, P_I\ad, L_I\ad)$, the pair $(P\ad, L\ad)$ has a dynamic description with respect to the twist $\la\ad \co \GG_m \to H\ad$ of $\la_I\ad$. By construction, $\la\ad$ is fixed by
\[ \Aut(H\ad, P\ad, L\ad)= \Aut(G\ad, P_I\ad, L_I\ad)^E. \]

Next, we consider the derived groups $\scD(H)$ and $\scD(G)$ and recall that $H\ad$ and $G\ad$ are quotients of $\scD(H)$ and $\scD(G)$ respectively. Let $n$ be an integer annihilating the finite diagonalizable $S$--group scheme $\rmZ_{\scD(G)}$. Then $n$ also annihilates the centre $\rmZ_{\scD(H)}$ of $\scD(H)$. Hence
the cocharacter 
$(\la\ad)^n\co \GG_m \to H\ad$ uniquely lifts to a cocharacter
$\la\co \GG_m\to \scD(H)$. The original proof of \ref{thm_richardson} shows that $(P,L)$ has a dynamic description with respect to $\la$. It remains to be shown that $\la$ is fixed by $\Aut(H,P,L)$. This follows from the unicity of the lifting of
$(\la\ad)^n$, since the action of $\Aut(H,P,L)$ on $\Hom_{S-{\rm gp}}(\GG_m,H\ad)$ factorizes through the map $\Aut(H\ad,P\ad,L\ad)$. \end{proof}


\subsection{Irreducible and anisotropic actions}\label{iaa} We extend the notion of a reducible or isotropic reductive group scheme (\ref{rev-isog}) to the setting where an $S$-group scheme $M$ acts on a reductive $S$--group scheme $H$ by automorphisms of $H$. 
\new
In this situation, $M$ acts canonically on the set of pairs $(P,L)$ consisting of  an everywhere  proper parabolic subgroup $P$ of $H$ in the sense of \ref{rev-isog} and a Levi subgroup $L$  of $P$. We say that the action is {\em reducible} if it normalizes such a pair $(P,L)$ in the sense of \cite{SGA3}, see for example [I; 2.3.3] or [VI$_B$, 6.4.4]. 
\enew 
Otherwise, the action is called {\em irreducible}. 

If $S$ is connected, as in \ref{cor_richardson} and \ref{prop_morozov}, then $P$ is everywhere proper if and only if $P$ is proper because the type of a parabolic subgroup is locally constant [XXVI, 3.2].

Similarly, we say that the  action of $M$ on $H$ is  {\em  isotropic}, if it centralizes  an $S$--subgroup of $H$ isomorphic to $\GG_{m,S}$. Otherwise, the action is called {\em anisotropic}. \sm

For $S$ the spectrum of a field and $M=\GG_{m,S}$, Corollary~\ref{cor_richardson} is \cite[4.23]{BT}.


\begin{cor} \label{cor_richardson}
Assume that $S$ is connected and that an $S$--group scheme $M$ acts on the reductive $S$--group $H$. Then the following are equivalent:
\begin{enumerate}[label={\rm (\roman*)}]
\item \label{cor_richardson-i} The action of $M$ on $H$ is isotropic.
\sm

\item \label{cor_richardson-ii} The action of $M$ on $\rad(H)$ is isotropic or
the action of $M$ on $H$ is reducible.
\end{enumerate} \end{cor}

\begin{proof}
\ref{cor_richardson-i} $\Longrightarrow$ \ref{cor_richardson-ii}:
Suppose the image of the non-trivial map $\la \co \GG_{m,S} \to H$ is
centralized by the action of $M$, and consider the parabolic subgroup
$P=\rmP_H(\la)$ together with its Levi subgroup $L=\rmC_H(\la)$, see \ref{dps}.
If $P=H$, then $\la$ is central and takes values
in $\rad(H)$, so that  the action of $M$ on $\rad(H)$ is isotropic. Otherwise,  $P$ is a proper parabolic subgroup. Since $M$ normalizes $(P,L)$, its action on $H$ is reducible. \sm

\ref{cor_richardson-ii} $\Longrightarrow$ \ref{cor_richardson-i}:
If the action of $M$ on $\rad(H)$ is isotropic, so is its action on $H$. Assume now that the action of $M$ on $H$ is reducible, that is, there exists a pair $(P,L)$ consisting of a proper parabolic subgroup $P$ of $H$ and a Levi subgroup $\subset P$, such that $(P,L)$ is
normalized by $M$. Theorem~\ref{thm_richardson} provides a group homomorphism $\la \co \GG_{m,S} \to H$ such that $(P,L)=(\rmP_H(\la), \rmC_H(\la))$
and such that $\la$ is fixed by $\Aut(H,P,L)$. In particular, $\la$ is non-central. After quotioning by $\Ker(\la)$, we can assume that $\la$ is a monomorphism, \ref{lem_quo-ex}. Since $M$ acts on $H$ through $\Aut(H,P,L)$, the image of $\la$ is fixed by $M$. Thus, the action of $M$ on $H$ is isotropic.
\end{proof}


The following Proposition~\ref{prop_morozov} is a complement to Theorem~\ref{thm_richardson}.

\begin{prop} \label{prop_morozov} Assume that $S=\Spec(R)$ is affine and connected, that $M$ is a flat affine  $R$--group scheme whose geometric fibers are linearly reductive, e.g., $M$ is of multiplicative type, and that $M$ acts on the reductive group $H$. If $M$ normalizes a  parabolic subgroup $P$ of $H$,  there exists an $M$--invariant group homomorphism $\la\co \GG_{m,S} \to H$
such that $P=\rmP_H(\la)$ and such that $\rmC_H(\lambda)$ is a Levi subgroup of $P$ which is normalized by $M$.
\end{prop}

\begin{proof} Let $U$ be the unipotent radical of $P$ which, 
according to [XXVI, 2.1], is $\Aut(P)$--linearizable, that is, $U$ admits a composition series $1=U_n \subset U_{n-1} \subset \dots \subset U_1= U$ which is $\Aut(P)$-stable and  such that  $U_i/U_{i+1} = W(E_i)$
where each $E_i$ is a locally free $R$--module of finite type on which the action of $\Aut(H,P)$ is linear. A fortiori,  the action of $M$ on $U$ is linearizable.

By [XXVI, 2.3 and 1.9], the set of sections of $P \to X := P/U$ is
non-empty and is a principal homogeneous space under $U(R)$.
For each $R$--ring $A$, the group $M(A)$ acts on the $A$-sections of  $P \to X$.
Let $s: X \to P$ be such a section. For each $m \in M(A)$,
there exists a unique $u_A(m)\in U(A)$ such that
${^m s_A}= \, ^{u_A(m)^{-1}}\!s_A$. For $m_1,m_2 \in M(A)$, we apply $m_1$ to
 ${^{m_2} s_A}= ^{u_A(m_2)^{-1}}\!s_A$ and get
\begin{align*}
  {^{m_1 m_2} s_A} &= \, ^{^{m_1}\!(u_A(m_2))^{-1}}(^{m_1}\!s_A)
  =  ^{^{m_1}(u_A(m_2))^{-1}}(^{u_A(m_1)^{-1}}\!s_A)
\\ &= ^{^{m_1}(u_A(m_2))^{-1} \,  u_A(m_1)^{-1}}\!(s_A),
\end{align*}
so that $u_A(m_1 m_2)=u_A(m_1) \cdot {^{m_1}\!(u_A(m_2))}$.
The map $u_A: M(A) \to U(A)$, $m \mapsto u_A(m)$, is therefore
a $1$-cocycle for the action on $M(A)$ on $U(A)$, so that the data of the $u_A$'s for $A$ running over all $R$--rings define an Hochschild  $1$-cocycle (or crossed homomorphism) $u \in \rmZ^1_{coc}(M,U)$, as defined by Demarche \cite[\S 3.2]{De}.
Thus, we obtain  a cohomology class $[u] \in \rmH^1_{coc}(M,U)$. We now claim \begin{equation} \label{prop_morozov-1} \rmH^1_{coc}(M,U)=1.
\end{equation}
We postpone the proof of \eqref{prop_morozov-1} for the moment.
Assuming \eqref{prop_morozov-1}, we get an element $v \in U(R)$ such that
$u_A(m)= v_A^{-1}   \, ^m \! v_A$ for all $m \in U(A)$.
It follows that ${^m s_A}= ^{^m v_A^{-1} v_A}\!s_A$,
so that  $s'= \, ^v\!s$ is an  $M$-invariant
section of $P \to X$. In particular, $s(X)$ is a Levi subgroup of $P$
which is $M$--invariant.
Theorem~\ref{thm_richardson} then shows that there exists
an $M$--invariant homomorphism $\la \co \GG_m \to H$
such that $(P, s(X))= (P_H(\lambda), C_H(\lambda))$ as desired. \sm

We now come to the proof of \eqref{prop_morozov-1}. According to \cite[Prop.~3.2.8]{De}, we have an exact sequence of pointed sets
\[
1 \to \rmH^1_{coc}(M,U) \to \rmH^1_0(M,U) \to \rmH^1(R, U)
\]
where $\rmH^1_0(M,U)$ is the set of isomorphism classes of $M$-$U$--torsors over $R$ \cite[3.2.6]{De}. We know from [XXVI, 2.2] that  $\rmH^1(R, U)=1$,  
so that it is enough to show that $\rmH^1_0(M,U)=1$.

According to Grothendieck for $M$ of multiplicative type
[IX, 3.1] and Margaux in general \cite[Th.~1.2]{Mg},
we have $\rmH^1_{coc}(M,U_i/U_{i+1})=0$ for $i=1,\dots, n-1$
so that $\rmH^1_0(M,U_i/U_{i+1})=0$ by the exact sequence above applied
to $U_i/U_{i+1}$. On the other hand, by \cite[Prop.~3.5.1]{De}  for $i=1,\dots, n-1$
we have an exact sequence of pointed sets
\[ \rmH^1_0(M,U_{i+1})  \to \rmH^1_0(M,U_i) \to  \rmH^1_0(M,U_i/U_{i+1}),
\]
so that we conclude by ``d\'evissage'' that $\rmH^1_0(M,U)=1$, finishing the proof of \eqref{prop_morozov-1}. \end{proof}

\begin{remarks}\label{morem} \begin{inparaenum}[(a)] \item The proof of Proposition~\ref{prop_morozov} is clearly  inspired by Demarche's paper \cite{De} and by  McNinch's approach \cite{Mc}.
\sm

\item \label{morem-b}Let $k$ be a field, let $H$ be a reductive $k$--group and let $M \subset H$ be a subgroup of $H$. Following Serre \cite{Serre1, Serre2}, see also \cite{BMR}, the subgroup $M$ is called {\em $H$--completely reducible\/}, if whenever $M\subset P$, where $P$ is a parabolic subgroup of $H$, then $M\subset L$ for some Levi subgroup of $P$.

Provisionally, we will use the completely analogous terminology for reductive groups over a scheme $S$. Specializing Proposition~\ref{prop_morozov} to the case of $M$ being a subgroup of $H$ acting on $H$ by inner automorphisms, the proposition implies that $M$ is $H$--completely reducible. This generalizes \cite[11.24]{Jantzen}, proven there for algebraically closed fields and restated in \cite[Lem~2.6]{BMR}. In characteristic $0$, the result is due to Mostow. \end{inparaenum} \end{remarks}



\section{Hermitian and quadratic forms over rings}\label{sec:hqr}

In this appendix we review the concepts needed for Propositions~\ref{hercan} and \ref{canqf}, but not more. This appendix is not an introduction to the topics of the section heading, like for example the book~\cite{K}, nor did the authors strive for the most general setting, i.e., considering forms over schemes instead of base rings. Despite this limitation, the appendix will hopefully still serve as a useful quick introduction to regular hermitian and nonsingular quadratic forms over rings, their associated isometry groups and the link to their flat cohomology sets.
\sm

Throughout this appendix, $R$ is an arbitrary base ring, $S$ is an algebra in $\Ralg$, equipped with an $R$--linear involution $\si$. This will later be specialized to the {\em hermitian case}, where $S$ is a quadratic \'etale extension with standard involution $\si$, and to the {\em quadratic case}, where $S=R$ and $\si = \Id_R$. The general setting avoids a duplication of definitions.


\subsection{Sesquilinear forms} \label{sef} \begin{inparaenum}[(a)] \item\label{sef-a} ({\em Modules})
Unless specified otherwise, all modules considered will be finite projective $S$--modules, \new usually considered as right $S$--modules. \enew Given such a module $M$,  we denote by $_\si M$ its $\si$--conjugate, i.e., the $S$--module with the same additive group as $M$, but with the $S$--action given by $({_\si m})\star s = m \, \si(s) $ for $s\in S$ and $m\in M$, using the given $S$--action of $M$ on the right-hand side of the equation.

Instead of the usual dual space $\Hom_S(M,S)$ we will use its twisted version
\[ M^* := {_\si \Hom_S(M,S)}. \]
One has an isomorphism of $S$--modules
$M^* \simlgr \Hom_S({_\si M}, S)$, given by $\vphi \mapsto \big({_\si m} \mapsto (\si \circ \vphi)(m)\big)$ (\cite[I; Lem.~(2.1.1)]{K}). \sm

\item\label{sef-b} ({\em Sesquilinear forms}) A {\em sesquilinear form\/}, also called an  {\em $(S,\si)$--sesqui\-linear form} is a bi-additive map $h \co M \times M \to S$, which is defined on a finite projective $S$--module $M$ and which satisfies
    \[ h(m_1 s_1, \, m_2s_2) = \si(s_1) \, h(m_1, m_2) \, s_2 \]
    for all $m_i \in M$ and $s_i \in S$. We will denote such a sesquilinear form by $(M,h)$ or simply by $h$.

    Given two sesquilinear forms $(M_1, h_1)$ and $(M_2, h_2)$, an $S$--linear map $f\co M_1 \to M_2$ is an {\em isometry\/} if $f$ is bijective and if $h_2\big(f(m_1), \, f(m_1')\big) = h_1(m_1, m'_1)$ holds for all $m_1, m_1' \in M_1$.  In this case, we say that $(M_1, h_1)$ and $(M_2, h_2)$ are {\em isometric}, abbreviated by $(M_1, h_1) \cong (M_2, h_2)$ or simply by $h_1 \cong h_2$.
    \sm

\item ({\em Hermitian and symmetric forms}) A sesquilinear form $(M,h)$ is called {\em $(S,\si)$-hermitian\/} or simply {\em hermitian}, if $h(m,m') = \si \big( h(m', m) \big)$ holds for all $m,m'\in M$. In the quadratic case, i.e., $(S,\si) = (R, \Id)$, we say $h$ is {\em symmetric\/} instead of hermitian. \sm

\item \label{sef-reg} ({\em Adjoints, regularity}) A sesquilinear form $(M,h)$ gives rise to the $S$--linear map $h^* \co M \to M^*$, $m \mapsto h(m, -)$, called the {\em adjoint of $(M,h)$}. Mapping $h$ to $h^*$ gives rise to a bijection between the $S$--module of sesquilinear forms on $M$ and $\Hom_S(M,M^*)$.
 We call $h$ {\em regular\/} if $h^*$ is  a bijection, hence an isomorphism of $S$--modules.
\sm

\item \label{sef-orth} ({\em Orthogonality})
    Given two sesquilinear forms $(M_1, h_1)$ and $(M_2, h_2)$, their {\em orthogonal sum\/} is the sesquilinear form $h_1 \perp h_2$ defined on $M = M_1 \oplus M_2 $ by
    \[ (h_1 \perp h_2)\, (m_1 + m_2, \, m_1' + m_2') = h_1(m_1, m'_1) + h(m_2, m'_2)\]
    for $m_1, m'_1 \in M_1$ and $m_2, m'_2\in M_2$. The sesquilinear form $h_1 \perp h_2$ is regular if and only if $h_1$ and $h_2$ are regular (\cite[I, (3.6.2.2)]{K}). The analogous statement holds for the property ``hermitian''.
\sm

\item\label{sef-bc} ({\em Base change}) Let $T\in \Ralg$. Then $S_T = S\ot_R T$ is an object in $\Ralg$ with a $T$--linear involution $\si_T = \si\ot 1_T$, and $S_T/T$ satisfy the assumptions of this appendix.     Given a finite projective $S$--module $M$ we put $M_T = M \ot_R T = M\ot_S S_T $, which is canonically a finite projective $S_T$--module.
    The twisted dual respects this base change: the canonical map
    $\om \co   M^* \ot_R T \simlgr (M_T)^* $
    is an isomorphism because $M$ is finite projective.

    Let now $(M,h)$ be an $(S,\si)$--sesquilinear form. Then $(M, h)$ uniquely extends to an $(S_T, \si_T)$-sesquilinear form $h_T$ on $M_T$ by requiring
    \[   h_T(m\ot t, \, m'\ot t') = h(m,m')\ot tt'\]
    for $m,m'\in M$ and $t,t'\in T$. The adjoint maps $(h_T)^*$ and the base change $h^* \ot 1_R$ are related by the commutative diagram
    \begin{equation}\label{sesquip-a} \vcenter{
     \xymatrix{ M \ot_R T  \ar[rr]^{ h^* \ot \Id_T }
            \ar[dr]_{( h_T )^*}
                 && M^* \ot_R T  \ar[dl]_\om^\cong \\
                   & (M_T)^* }
     } \quad .
     \end{equation}
\lv{
Commutativity of the diagram: $ (h^* \ot \Id_{\scB'})(m \ot b') = h^*(m) \ot
b'$ and $ \om(h^*(m) \ot b')\, (\m_1 \ot c') = \tau'(b') \m h^*(m)(m_1) \, c' =
h_{\scB'}(m \ot b', \, m_1 \ot c') = h_{\scB'} (m\ot b')(m_1 \ot c')$
}
Hence, if $(M,h)$ is regular, then so is $(M_T, h_T)$.
\sm

\item\label{sef-hyp} ({\em Hyperbolic forms} \cite[I, (3.5)]{K}) Let $U$ be a finitely generated projective $S$-module. One defines a hermitian module $\HH(U) = (U \oplus U^*, h_{\HH(U)})$ by $h_{\HH(U)}(u_1 + \vphi_1, \, u_2 + \vphi_2) = \vphi_1(u_2) + \si (\vphi_2(u_1))$ for $u_i \in U$ and $\vphi_i \in U^*$. One calls $(M, h)$ {\em hyperbolic\/} if it is isometric to $\HH(U)$ for some $U$.

A hyperbolic space $(M,h)=\HH(U)$ is equipped with a natural action of the $R$-group scheme $\frR_{S/R}\big( \uGL_S(U)\big)$, $\frR(\cdot)$ = Weil restriction, as follows. For each $R$-ring $T$,  each $g \in \frR_{S/R}\big(\uGL_S(U)\big)(T) = \GL_{S \ot_R T}(U \ot_R T)$ and each pair $( u, f) \in (U \ot_R T) \oplus (U^*\ot_R T)$, we have
 $g.( u, f)= (g.u, f \circ g^{-1})$. It follows that we have a closed embedding
 \begin{equation}\label{sef-hyp1}
   \frR_{S/R}\bigl(\uGL_S(U)\big) \hookrightarrow \uU(M,h)
 \end{equation}
 of $R$--group schemes, \new where the $R$--group scheme $\uU(M,h)$ represents the $R$--functor of isometries of $(M,h)$, cf.~\ref{ung}\eqref{ung-c}.\enew 
 \sm

\item\label{sef-ti}  ({\em Totally isotropic submodules}) Let $(M,h)$ be a hermitian module. A submodule $U\subset M$ is called {\em totally isotropic\/} if $U$ is complemented and $U \subset U^\perp$, where, in general,
    \[ U^\perp = \{ m\in M : h(m, U) = 0 \}. \]
    A vector $m\in M$ is {\em isotropic\/} if $m$ is unimodular and $h(m,m) = 0$. In case $(S,\si) = (R,\Id)$ we also require $q(m) = 0$. 
\sm


\item\label{sef-diag} A sesquilinear form $(M,h)$ is {\em diagonalizable\/} if there exists a basis $(e_1, \ldots, e_n)$ of the $S$--module $M$ such that $h(e_i, e_j) = 0$ whenever $i\ne j$. \sm

\item \label{sef-split} The free $S$--module $M=S^n$, $n\in \NN_+$, carries the {\em split hermitian form $h_{0,n}$\/}, defined by
\begin{equation}  \label{sef-fre1}
\textstyle h_{0, n}\big( \sum_{i=1}^n e_i s_i, \, \sum_{j=1}^n e_j s'_j\big) = \sum_{i=1}^n \si(s_i) s_i'.
\end{equation}
The term ``split'' will be justified in Lemma~\ref{dia} where we consider regular hermitian forms. 
\end{inparaenum}

We characterize regularity of sesquilinear forms in the following lemma. We call a family $(f_1, \ldots, f_n)$ of elements of $R$ a {\em Zariski cover} if $R=Rf_1 + \cdots + Rf_n$. Equivalently, $\{ \Spec(R_{f_i}) \to \Spec(R)\}_{i=1, \ldots, n}$ is a standard Zariski covering of $\Spec(R)$ in the sense of \cite[Tag 020R]{Stacks}.

\begin{lem}\label{sesquip-b} Assume $S\in \Ralg$ is finite projective as $R$--module.
Then the following are equivalent for a sesquilinear form $h$.
\begin{enumerate}[label={\rm (\roman*)}]

\item\label{quadfob-i} $h$ is regular;

\item\label{quadfob-vi} $h_T$ is regular for all $T\in \Ralg$;

\item\label{quadfob-vii} $h_K$ is regular for all algebraically closed fields
    $K\in \Ralg$;


\item\label{quadfob-ii} $h_{R_\m}$ is regular for all maximal\/ $\m \in
    \Spec(R)$;

\item\label{quadfob-iv} there exists a Zariski cover $(f_1, \ldots, f_n)$ of
    $R$ such that $h_{R_{f_i}}$  is regular for all $i$;

\item\label{quadfob-iii} $h_{R/\m}$ is regular for all maximal\/  $\m \in
    \Spec(R)$;

\item\label{quadfob-v}$h_T$ is regular for some faithfully flat $T\in \Ralg$.
\end{enumerate}
\end{lem}

\begin{proof} Let $h=(M,h)$. We will use that the adjoint map $h^* \co M \to M^*$ is bijective as $S$--linear map if and only if the underlying $R$--linear map is bijective. Moreover, by transitivity of ``finite projective'' (\cite[1.1.8]{Ford}), $M$ and $M^*$ are finite projective as $R$--modules. We will use the following scheme to  prove \ref{sesquip-b}.
\[\xymatrix{
   \ref{quadfob-v} \ar@{=>}[r] & \ref{quadfob-i} \ar@{=>}[r] 
         &\ref{quadfob-vi}
       \ar@{=>}[r] \ar@{=>}[d] &\ref{quadfob-vii} \ar@{=>}[dl]     \\
   \ref{quadfob-iv} \ar@{=>}[u] & \ref{quadfob-ii} \ar@{=>}[l] \ar@{<=>}[r] &
           \ref{quadfob-iii}
}
\]

As already noted in \ref{sef}\eqref{sef-bc}, the implication \ref{quadfob-i} $\implies$  \ref{quadfob-vi}  
follows from the commutative diagram \eqref{sesquip-a}, while  \ref{quadfob-vi} $\implies$ \ref{quadfob-vii} and \ref{quadfob-vi} $\implies$ \ref{quadfob-iii} are trivial. \new To see  \ref{quadfob-vii} $\implies$ \ref{quadfob-iii}, let $\m$ be a maximal ideal of $R$ and let $K$ be an algebraic closure of the field $F=R/\m$. Then \eqref{sesquip-a} for $(R,S,T)$ replaced by $(F, S\ot_R F, K)$ shows that $h_F \ot 1_K$ is invertible, hence so is $h_F$. \enew
For \ref{quadfob-ii} $\iff$ \ref{quadfob-iii} we need $h^*_{R_\m}$ bijective $\iff h^*_{R/\m}$ bijective, which is \cite[II,
\S3.2, Cor. de la Prop.~6]{BAC}.

\ref{quadfob-ii} $\implies$ \ref{quadfob-iv}:
By \cite[II, \S5.1, Prop.~2(ii)]{BAC} there exists for every $\m \in \Spec(R)$
a $g_\m\in R$ such that $h^*_{g_\m}$ is bijective. The existence of a Zariski
cover with the stated property then follows as usual: $\{g_\m : \m \in \Spec(R)
\hbox{ maximal} \}$ generates $R$ as an ideal; one takes a generating set
$\{f_1, \ldots, f_n\}$ of this ideal.  \ref{quadfob-iv} $\implies$
\ref{quadfob-v}:  By \cite[II, \S5.1, Prop.~3]{BAC} the algebra $S=R_{f_1}
\times \cdots   \times R_{f_n}$ is faithfully flat and $h^*_S = \prod_i
h^*_{f_i}$ is bijective. Finally, the implication \ref{quadfob-v} $\implies$ \ref{quadfob-i} follows from \cite[I, \S3.1, Prop.~2]{BAC}.
\end{proof}
\sm

\subsection{Hermitian spaces} \label{hsp}
In \ref{hsp} -- \ref{dec} we consider the ``hermitian case'', i.e.,
\begin{equation}\label{hsp0}
\text{$S$ is a quadratic \'etale $R$--algebra with standard involution $\si$.}
\end{equation}
A hermitian sesquilinear form will be hermitian with respect to $(S,\si)$.
A {\em hermitian space $(M,h)$\/} is a hermitian module $(M,h)$ with $h$ a regular hermitian form; a {\em hermitian space of rank $r$\/} is a hermitian space $(M,h)$ for which $M$ is a
locally free $S$--module of constant rank $r$.
\comments{Lemma~\ref{dec} is not true for regular symmetric bilinear forms, but it holds for regular quadratic forms by \cite[I, (3.6), (4.6)]{Ba}. Lemma~\ref{LGdia} is not true for $(S,\si) = (R, \Id)$ and characteristic $2$. }

\begin{lem}[LG diagonalizability, hermitian case] \label{LGdia} Let $R$ be an LG-ring, let $(S,\si)$ be a quadratic \'etale $R$--algebra with standard involution $\si$, and let
$(M,h)$ be a hermitian space with $M$ having constant positive rank. Then $(M,h)$ is diagonalizable, in particular
\begin{equation}\label{LGdia1}
  \{ h(m,m): m\in M\}\cap R\ti \ne \emptyset.
\end{equation}
\end{lem}

\begin{proof} This is \cite[Exc.~21.21(a)]{PRbook}; its proof is given in \cite{PRbook-sol}: one can assume that $M$ has constant rank; one then shows \eqref{LGdia1} and finally proves diagonalizability by induction on the rank of $M$ (recall \cite[I, (3.6.2.1)]{K} and \ref{prop_baireco}\eqref{prop_baireco-b}).
\end{proof}

\begin{lem} \label{dia} Let $R$ be arbitrary and let $(S,\si)$ be a quadratic \'etale $R$--algebra with standard involution $\si$. \sm

\begin{inparaenum}[\rm (a)] \item \label{dia-a} {\em (Zariski-diagonalizability)} Let $(M,h)$ be a hermitian space  with $M$ being faithfully projective. Then $(M,h)$ is Zariski-locally diagonalizable, i.e., there exists a Zariski cover $(f_1, \ldots, f_m)$ of $R$ such that for every $g\in \{f_1, \ldots, f_m\}$ the base change $(M,h)_{R_g} := (M_{R_g}, h_{R_g})$ is 
diagonalizable of rank $n\in \NN_+$. 
\ms

\item \label{dia-b} 
The following are equivalent for a hermitian form $(M,h)$.
\end{inparaenum} \begin{enumerate}[label={\rm (\roman*)}]

\item\label{diagc-a} $M$ is locally free of rank $n$ and $h$ is
    regular.

\item\label{diagc-b} There exists a flat cover $T\in \Ralg$ such that
    $(M, h)_T$ is  isometric to the split form \eqref{sef-fre1}.
 \end{enumerate}
\end{lem}

\begin{proof} \eqref{dia-a} Let $\m\ideal R$ be a maximal ideal of $R$. By Lemma~\ref{LGdia}, $(M,h)$ is diagonalizable over the local ring $R_\gm$. Therefore, since $R_\m \cong \limind_{f\notin \m} R_f$, there exists $f\notin \m$ such that $h_{R_f}$ is diagonalizable. Hence the ideal generated by $\{f\in R: h|_{R_f} \hbox{ is diagonalizable}\}$ is all of $R$, from which \eqref{dia-a} easily follows.%
\sm

\eqref{dia-b}   Assume \ref{diagc-a}. Let $(f_1, \ldots, f_m)$ be a Zariski cover as in \eqref{dia-a}, and put $T' = R_{f_1} \times  \cdots \times  R_{f_m}$. Then $T'$ is a flat cover of $R$ and $(M, h)_T \cong (S^n, h')$ with a diagonalizable $h'$.
Since a flat cover $T''$ of $T'$ is a flat cover of $R$, it
suffices to prove \ref{diagc-b} under the assumption that $M \cong S^n$
and that $h$ is diagonalizable with respect to the standard basis $(e_1, \dots,
e_n)$ of $S^n$.
Note $h(e_i, e_i) = r_i \in R\ti$ since $h$ is supposed to be hermitian  and regular. Let $R[X_1, \ldots, X_n]$ be the polynomial ring in $n$ variables. Then \[ T=R[X_1, \ldots, X_r]\big/ (X_1^2 - r_1, \ldots, X_n^2 - r_n)\] is a flat cover
of $R$ such that $h_T \cong h_{0, n}$, the split rank $n$ form over $T$.

Conversely, if \ref{diagc-b} holds, then $M$ is locally free of
rank $n$ (\cite[II, \S5.3, Prop.~4]{BAC}) and $h$ is regular by Lemma~\ref{sesquip-b}. \end{proof}

\begin{prop}[Unitary groups]\label{ung} Let $(S,\si)$ be quadratic \'etale with standard involution $\si$, and let $(M,h)$ be a hermitian space. Then the following hold. \sm

\begin{inparaenum}[\rm (a)]
  \item \label{ung-a} $A := \End_S(M)$ is an Azumaya $S$-algebra. \sm

  \item \label{ung-b} There exists a unique involution $\tau_h$ on $A$,
  the\/ {\em adjoint involution},  satisfying
  \begin{equation} \label{ung-b1}
    h( m_1, f(m_2)) = h( \tau_h(f)(m_1), m_2)
  \end{equation}  for all $m_i \in M$ and $f\in A$.  The triple $(S, A, \tau_h)$ is an Azumaya algebra with involution of second
  kind in the sense of\/ \cite[2.7.0.33]{CF}. \sm

  \item \label{ung-c} We define
$\uU(M, h) := \uU(A, \tau_h)$ where $\uU(A, \tau_h)$ is the unitary $R$--group scheme of\/ \cite[3.5.0.84]{CF}. Then, for every $R$--ring $T$,
\[ \uU(M,h)(T) = \{ f\in \GL_{S\ot_R T}(M\ot_R T) : f \hbox{ is an isometry of } h_T \}\]
where $h_T$ is the base change of $h$ in the sense of {\rm \ref{sef}\eqref{sef-bc}}. \sm

\item \label{ung-d} The group $\uU(M,h)$ is a reductive $R$--group scheme. If
$M$ has constant rank $n$, then $\uU(M,h)$ has type $\rmA_{n-1}$. \sm

\item \label{ung-e}  Fix $n\in \NN_+$. The map
$  [(M,h)] \mapsto [\uIsom\big( (S^n, h_{0,n}), \, (M,h)\big)] $
is a bijection between the set of isometry classes of hermitian spaces over $(S,\si)$ of rank $n$ and $\rmH^1\big(R, \uU(S^n, h_{0,n})\big)$. Equivalently,
for every hermitian space $(M,h)$ of rank $n$, the set of twisted forms of $(M,h)$ is in bijection with $\rmH^1\big(R, \uU(M,h)\big)$.
\end{inparaenum}
\end{prop}

\begin{proof}
For \eqref{ung-a} see for example \cite[7.1.10]{Ford}. In case $S$ is a field, \eqref{ung-b} is proven in \cite[4.1]{KMRT}; the same proof works in our setting.
\eqref{ung-c} Since $h\big( (\tau_h(f) \circ f)(m_1), m_2\big) = h\big(
f(m_1), f(m_2)\big)$ by \eqref{ung-b1}, the equation $\tau_h(f) \circ f = \Id_M $, defining $\uU(M,h)(T)$, is equivalent to $h\circ (f \times f) = h$, i.e., $f$ is an isometry. \eqref{ung-d} follows from \cite[3.5.0.87]{CF}.

\eqref{ung-e} Using \ref{dia}\eqref{dia-b}, it is straightforward (but somewhat technical) to define the gerbe of hermitian spaces of rank $n$ over the big affine fppf site of $\Spec(R)$. The claim then becomes a special case of \cite[2.2.4.5]{CF}. Alternatively (and more down-to-earth), one can easily adjust the notion of a tensor system of \cite[54.7]{PRbook} to apply to hermitian spaces, cf.~\ref{sef}\eqref{sef-reg}. The claim then becomes a consequence of the Descent Theorem \cite[54.15]{PRbook}. \end{proof}
\sm

\textbf{Remark.} Since $\uU(M,h)$ is smooth, it is known (\cite[III,
Remark~4.8]{milne} or \cite[2.2.5.15]{CF}) that the canonical map $\rmH^1_{\et}(R, \uU(M, h)) \simlgr \rmH^1(R, \uU(M, h))$ is a bijection. Thus, one can replace twisted forms with respect to the flat topology in \ref{ung}\eqref{ung-e} by twisted forms in the \'etale topology. The analogous remark applies to \ref{dia}\eqref{dia-b}.

\begin{prop}[Parabolic and Levi subgroups of $\uU(M,h)$]\label{dec} Assume \eqref{hsp0}, and let $(M,h)$ be a hermitian space over $(S,\si)$ containing a totally isotropic submodule $U\subset M$. \sm

\begin{inparaenum}[\rm(a)]
  \item \label{dec-a} Then there exists a totally isotropic submodule $V$ and a submodule $M'$ such that
  \begin{equation}\label{dec-a1}
   M= (U \oplus V) \perp M', \quad (U\oplus V, h|_{U\oplus V}) \cong \HH(U),
   \end{equation}
 in particular $(M', h|_{M'})$ is a hermitian space. \sm

 \item \label{dec-b} The subgroup scheme $P$ of $G=\uU(M,h)$ which stabilizes $U$ is a parabolic subgroup of the reductive $R$--group scheme $G$. \sm

\item \label{dec-c} Consider a  decomposition  $(M,h)= (U \oplus V) \perp M'$ as in \eqref{dec-a1}, and the embedding $\frR_{S/R}(\uGL(U)) \hookrightarrow G$ defined in \eqref{sef-hyp1}. Then the $R$--group scheme $L=\frR_{S/R}(\uGL(U)) \times_R \uU(M',h|_{M'})$ is a Levi subgroup of $P$.
\end{inparaenum}
\end{prop}

\begin{proof} \eqref{dec-a} Since the trace map $S \to R$, $s \mapsto s + \si(s)$ is surjective \cite[III, (4.2.1)]{K}, it follows from \cite[I, (3.1.1)]{K} that $h$ is an even hermitian module in the sense of \cite[I, (3.1)]{K}. The lemma is therefore a consequence of \cite[I, (3.7.1)]{K}. That $V$ can be chosen totally isotropic follows from the proof of \cite[I, (3.7.1)]{K}. \sm

\eqref{dec-b} We will use the dynamic method of \ref{dps}. It is easily verified that
\[ \la \co \GG_m \to \uU(M,h), \quad t \mapsto t\Id_U + t\me \Id_V + 1_{M'} \]
defines a cocharacter. Writing a $g\in \uU(M,h)(T)$, $T\in \Ralg$, as a matrix $g=(g_{ij})_{1\le i,j\le 3}$ with respect to the submodules $(U,V, M')$, one finds
\[
  \la(t) \, (g_{ij}) \, \la(t)\me =
  \begin{pmatrix} g_{11} & t^2 g_{12} & t g_{13} \\
             t^{-2}g_{21} & g_{22} & t^{-1} g_{23} \\
             t\me g_{31} & t g_{32} & g_{33} \end{pmatrix} \; .
\]
Hence, $g\in \rmP_G(\la)(T) \iff 0 = g_{21} = g_{31} = g_{23}$, i.e., $g$
stabilizes $U$ and $U^\perp = U \oplus M'$. Since $g$ is an isometry, the latter conditions are equivalent to $g$ stabilizing $U$. \sm

\eqref{dec-c} The centralizer of the cocharacter $\la$ above consists of those $g$ that are diagonal with respect to $(U,V, M')$. These are easily seen to be the isometries in $L(T)$. The claim then follows from the dynamic method. \end{proof}
\sm

\textbf{Remarks.} \begin{inparaenum} \item That $P$ in \ref{dec}\eqref{dec-b} is parabolic with Levi subgroup $L$ as in \eqref{dec-c}, is not surprising and likely folklore. It is in the spirit of Appendix $T$ of Conrad's course  \cite{Co3}, which identifies the parabolic subgroups for symplectic and special orthogonal groups over fields. \sm

\item \label{dec-2} Suppose that $M$ is free of rank $n\ge 2$ and that $U$ has rank $i$ for some $1\le i < n/2$. Then $P$ is of type $\rmA_{n-1} \setminus \{i\}$, while the derived group $\scD(L)$ has absolute type $\rmA_{i-1} \times \rmA_{n-1-i}$ for $i\ge 2$ and type $\rmA_{n-2}$ for $i=1$.
\end{inparaenum}

%
%

\subsection{Quadratic forms}\label{qf} From now on until the end of this appendix we consider quadratic forms. Thus, we specialize \ref{sef} to $(S, \si) = (R, \Id)$. Unless stated otherwise, $R$ is arbitrary and, we recall, $M$ is a finite projective $R$--module.  Below is a quick reminder of the basics of quadratic forms over rings; proofs and more can \new for example be found in \cite{Ba}. 

\comments{(2025-02-22) I deleted the reference \cite[I,II]{K}, which before was here. Was this your suggestion? I am not sure I understood your suggestion in your email of 2025-02-21 correctly}
\enew 

\sm

\begin{inparaenum}[(a)]
 \item A {\em quadratic form (over $R$)\/} is a pair $(M,q)$ consisting of a finite projective $R$--module $M$ and a map $q \co M \to R$ satisfying $q(rm) = r^2 q(m)$ for all $r\in R$ and $m\in M$ and for which the {\em polar form $b_q$\/}, defined by $b_q(m, m') = q(m+m') - q(m) - q(m')$ for $m,m'\in M$, is a (symmetric) bilinear form on $M$. We often abbreviate $(M,q)= q$ and refer to $(M,q)$ as a quadratic module.
     We call $(M,q)$ a {\em faithful\/} quadratic $R$--module if $M$ is a faithfully  projective $R$--module, \ref{fapmod}. We say that $(M,q)$ is a quadratic module of rank $n\in \NN$, if $M$ has constant rank $n$.

An {\em isometry\/} $f\co (M_1, q_1) \to (M_2, q_2)$ is an $R$--linear isomorphism $f\co M_1 \to M_2$ satisfying $q_2 \circ f = q_1$.
If $(M_1, q_1) = (M_2, q_2) = (M,q)$, the isometries of $(M,q)$ form a group $\orth(M,q) = \orth(q)$, the {\em orthogonal group of $q$}.      \sm

\item ({\em Regularity}) By definition, a quadratic form $q$ is {\em regular\/}, if $b_q$ is regular in the sense of \ref{sef}\eqref{sef-reg}.
\sm

\item ({\em Base change}) \label{qf-bc} Let $T\in \Ralg$ and let $(M,q)$ be a quadratic form. Analogous to \ref{sef}\eqref{sef-bc} we consider the $T$--module $M_T = M \ot_R T$. There exists a quadratic form $(M_T, q_T)$ uniquely determined by the condition $q_T(m\ot t) = q(m) t^2$ for all $m\in M$ and $t\in T$ (\cite[11.5]{PRbook}, \cite[Thm.~1]{Sah}). The polar of $q_T$ is the base change of the polar $b_q$ of $q$, i.e., $b_{q_T} = (b_q)_T$. In particular, if $q$ is regular, then so is $q_T$. \sm

\item\label{qf-rad} ({Radical}) The {\em radical\/} of a quadratic module $(M,q)$ is the submodule $\rad(q)= \{m\in M : q(m) = 0 = b_q(m,M)\}$ of $M$. Given $T\in \Ralg$, clearly $\rad(q) \ot_R T \subset \rad(q_T)$. This inclusion is in general not an equality, cf.~\eqref{qf-ns}. \sm

\item \label{qf-ns} ({\em Nonsingularity}) A quadratic form $(M,q)$ is called {\em nonsingular\/} if $\rad(q_F) = 0$ for all fields $F\in \Ralg$. In this case, we call $(M,q)$ a {\em quadratic space}. We point out that a ``quadratic space'' in the sense of \cite{Ba} or \cite{PRbook} is a quadratic module $(M,q)$ with a regular $q$.

    If $q$ is regular, then $\rad(q) = 0$ and hence \ref{sesquip-b}\ref{quadfob-vi} implies
\begin{equation}\label{qf-ns1}
   \text{$q$ regular} \implies \text{$q$ nonsingular.}
\end{equation}

\inparcom{The referee suggested to change the reference 
 \ref{sesquip-b}\ref{quadfob-vi} to \ref{sesquip-b}\ref{quadfob-vii}. I did not do this, since I think that \ref{sesquip-b}\ref{quadfob-vi} is more natural. But maybe I misunderstood the referee.}
 
The converse of \eqref{qf-ns1} is not true. For example, any $u\in R\ti$ gives rise to a nonsingular quadratic form $\lan u \ran \co R \to R$, given by $\lan u \ran (r) = u r^2$ for $r\in R$. Its polar form is the symmetric bilinear form $b_{\lan u \ran}(r_1, r_2) = 2 u r_1 r_2$. Hence, $\lan u \ran$ is regular if and only if $2\in R\ti$.

However, if $2\in R\ti$, then $\rad(q) = \{ m\in M : b_q(m,M) = 0 \}$ for any quadratic form $q$ and so $\rad(q)$ is stable under base change. Hence,%
\begin{equation}\label{qf-ns2}
\text{ if  $2 \in R\ti$, then $q$ is nonsingular $\iff q$ is regular.}
\end{equation}%

\item ({\em Orthogonality}) \label{qf-perp} Given two quadratic forms $(M_1, q_1)$ and $(M_2, q_2)$, their {\em orthogonal sum\/} is the quadratic form $q_1 \perp q_2$, defined on $M= M_1 \oplus M_2$ by $(q_1 \perp q_2)(m_1, m_2) = q_1(m_1) + q_2(m_2)$ for $m_1 \in M_1$ and $m_2 \in M_2$. The polar form of $q_1 \perp q_2$ is the orthogonal sum $b_{q_1 \perp q_2} = b_{q_1} \perp b_{q_2}$ defined in \ref{sef}\eqref{sef-orth}.

     By \ref{sef}\eqref{sef-orth}, the quadratic form $q= q_1 \perp q_2$ is regular if and only if $q_1$ and $q_2$ are regular. Regarding nonsingularity,
     one easily sees:
     \begin{equation} \label{qf-perp1} \begin{split}
      q_1 \perp q_2 \text{ nonsingular } &\implies \text{$q_1$ and $q_2$ nonsingular, } \\
     \text{$q_1$ {\em regular\/} and $q_2$ nonsingular} &\implies
        \text{$q_1 \perp q_2$ nonsingular.}%
    \end{split} \end{equation}
     \sm

\item \label{qf-redc} ({\em Direct products})
Let $R= R_0 \times \cdots \times R_n$ be a direct product. Every $R$--module $M$ uniquely decomposes $M= M_0 \times \cdots \times M_n$ as a direct product of $R_i$-modules $M_i = R_i M$. The $R$--module $M$ is finite projective if and only if every $R_i$--module $M_i$ is finite projective.

Let $(M,q)$ be a quadratic $R$--module. Then $b_q(M_i,M_j) = 0$ for $i\ne j$ and
$(M_i, q_i)$ with $q_i = q|_{M_i}$ is a quadratic $R_i$--module. Thus
\[ (M,q) = (M_0, q_0) \perp \ldots \perp (M_n,q_n),\quad q_i = q|_{M_i}. \]
The quadratic $R$--module $(M,q)$ is regular (nonsingular respectively) if and only if every $(M_i, q_i)$ is a regular (nonsingular respectively) quadratic $R_i$--module.

A standard way to obtain the situation considered here occurs by letting $M=M_0\times \cdots \times M_n$ be the {\em rank decomposition\/} of a finite projective $R$--module $M$ for which $M_i$, $0\le i \le n$, is a finite projective $R_i$--module of constant rank $i$.
\sm

\item\label{qf-hyp} ({\em Hyperbolic spaces}) Let $U$ be a finite  projective $R$--module. The associated {\em hyperbolic space $\HH(U)$} is the quadratic module $(U^* \oplus U, \hyp)$ with quadratic form $\hyp_U(\vphi \oplus u) = \vphi(u)$, where $\vphi \in U^*$ and $u\in U$. The quadratic form $\hyp_U$ is regular, hence nonsingular by \eqref{qf-ns1}. The polar form of $\hyp$ is the hyperbolic symmetric bilinear form of \ref{sef}\eqref{sef-hyp}. In general, a {\em hyperbolic space\/}  is a quadratic module $(M,q)$ isometric to some $\HH(U)$.
\sm

\item\label{quadfoc} ({\em Split quadratic forms}) Let $m\in \NN$. The quadratic  form $q_{0, 2m}$ is the hyperbolic form associated with the free $R$--module $R^m$. After identifying $R^m{}^* = R^m$, it is given on $R^{2m}$ by
\begin{equation}  q_{0, 2m} (r_{-m}, \ldots, r_{-1}, r_1, \ldots, r_m) = \textstyle \sum_{i=1}^m r_i r_{-i}, \label{quadfoc1}
\end{equation}
It is regular, hence also nonsingular by \eqref{qf-ns1}. The quadratic form
$q_{0, 2m+1} = \lan 1 \ran \perp q_{0, 2m}$ on $R^{2m+1}$, defined by
\begin{equation}
q_{0, 2m+1} (r_{-m}, \ldots, r_{-1}, r_0 ,r_1 \ldots, r_{m}) = r_0^2
+ \textstyle \sum_{i=1}^m r_i r_{-i},
\label{quadfoc2}
\end{equation}
is nonsingular, e.g.,  by the even rank case, by \eqref{qf-perp1} and by nonsingularity of $\lan 1 \ran$, see \eqref{qf-ns}. We will refer to $q_{0, n}$ for
$n$ even or odd as the {\em split quadratic forms\/}, see
Proposition~\ref{qfnsp} for a justification for this
terminology.
\end{inparaenum}

\begin{prop} \label{qfnsp}
For a faithful quadratic module $(M,q)$ the following are equivalent:
\begin{enumerate}[label={\rm (\roman*)}]
  \item \label{qfnsp-i} $q$ is nonsingular;

  \item \label{qfnsp-ii} $q_T$ is nonsingular for all $T\in \Ralg$;

  \item \label{qfnsp-iii} $q_{R/\gm}$ is nonsingular for all maximal ideals $\gm \ideal R$;

  \item\label{qfnsp-iv}  there exists a flat cover $(R_1, \ldots, R_n)$ such
    that each $M \ot_R R_i$ is a free $R_i$-module of finite rank $r_i$ and $q_{R_i}$  is the split quadratic form $q_{0,r_i}$ over $R_i$ defined in {\rm
    \ref{qf}\eqref{quadfoc}};

  \item \label{qfnsp-v} $q_S$ is nonsingular for some faithfully flat $S\in \Ralg$;
\end{enumerate}
If $M$ has constant rank $n\in \NN_+$, then \ref{qfnsp-i}--\ref{qfnsp-v} are equivalent to
\begin{enumerate}[label={\rm (\roman*)}]\setcounter{enumi}{5}
 \item\label{qfnsp-vi} $q$ is regular if $n$ is even and $q$ is semiregular in the sense of {\rm \cite[IV, (3.1)]{K}} if $n$ is odd;
\end{enumerate}
If $R$ is a field, then $q$ is nonsingular if and only if
\begin{enumerate}[label={\rm (\roman*)}]\setcounter{enumi}{6}
  \item \label{qfnsp-vii} $q$ is {\em nondegenerate\/} in the sense of {\rm \cite[(7.17)]{EKM}}, i.e., one of the following two conditions hold:

   \begin{enumerate}[label={\rm (\alph*)}]
     \item $q$ is regular, or

     \item $\Char(R) = 2$, $\rad(q) = 0$, $\dim_R\{ m\in M : b_q(m,M) = 0 \} = 1$, and $\dim_R M$ is odd.
\end{enumerate}
\end{enumerate}
\end{prop}

\begin{proof} \ref{qfnsp-i} $\iff$ \ref{qfnsp-ii} is easy. It is shown in \cite[Prop.~1.1, Cor.~1.3]{Sw} that \ref{qfnsp-i} $\iff$ \ref{qfnsp-iii} $\iff$ \ref{qfnsp-ivet}, defined as
\begin{enumerate}[label={\rm (\roman*)$\et$}]\setcounter{enumi}{3}
\item \label{qfnsp-ivet} there exists an \'etale cover $(R_1, \ldots, R_n)$ of $R$ such that each $M \ot_R R_i$ is a free $R_i$-module of finite rank $r_i$ and $q_{R_i}$ is hyperbolic as in {\rm \eqref{quadfoc1}} if $r_i$ is even,   while $q_{R_i}$ is  the orthogonal sum of a hyperbolic form and a  $1$-dimensional form $\lan a_i \ran$ with $a_i \in  R_i^\times$ in case $r_i$ is odd.
\end{enumerate}
One shows \ref{qfnsp-ivet} $\implies$ \ref{qfnsp-iv} as in the proof of \ref{dia}\eqref{dia-b}\ref{diagc-a} $\implies$ \ref{diagc-b}.  If \ref{qfnsp-iv} holds, then \ref{qfnsp-v} is satisfied with $S= R_1 \times \cdots \times R_n$. For \ref{qfnsp-v} $\iff$ \ref{qfnsp-i} see \cite[Lem.~6.2]{Lo2}. Thus, we proved that \ref{qfnsp-i} -- \ref{qfnsp-v} are all equivalent using the diagram
\[ \xymatrix@C=45pt{
  \ref{qfnsp-ii} \ar@{<=>}[r] & \ref{qfnsp-i} \ar@{<=>}[r]
      \ar@{<=>}[d]
      &\ref{qfnsp-iv}\et \ar@{<=>}[r]  \ar@{=>}[d] 
& \ref{qfnsp-iii}
 \\ & \ref{qfnsp-v} & \ar@{=>}[l] \ref{qfnsp-iv}
}\]
If $R$ is a field, then \ref{qfnsp-i} $\iff$ \ref{qfnsp-vii} by \cite[7.16]{EKM}. Again over a field, \ref{qfnsp-vii} implies that $q$ is semiregular in the odd rank case by \cite[IV, (3.1.7)]{K}. Finally, if $M$ has constant rank over an arbitrary $R$, then $q$ is regular,  respectively semiregular, if and only if it is so over all residue fields $R/\gm$ (\cite[IV; (3.1.5)]{K} for odd rank, \ref{sesquip-b} for even rank). Thus, to prove that \ref{qfnsp-vi} is equivalent to $q$ being nonsingular, we can assume that $R$ is a field, in which case the equivalence \ref{qfnsp-vi} $\iff$ \ref{qfnsp-vii} is easily established using \cite[IV, (3.1.7)]{K} for a semiregular $q$. \end{proof}

\subsection{Quadratic algebras} \label{qua}
We recall (\cite[I, (1.3.6), and III, \S4]{K}) that a {\em quadratic $R$--algebra\/} is an algebra $S$ in $\Ralg$ whose underlying $R$--module is projective of rank $2$. Here are some properties of quadratic algebras: \sm

\begin{inparaenum}[(a)] \item \label{qua-inv} A quadratic $R$--algebra $S$ has a {\em standard involution $\si_S$\/}. If $S$ is free with basis $\{1, z\}$ and hence $z^2 = az + b$ with $a,b\in R$, it is given by $\si_S(z) = a - z$. In general, $\si_S$ is defined by Zariski-descent. \sm

\item \label{qua-aut} ({\em Automorphisms}) We denote by $\ZZ/2\ZZ_R$ the constant group scheme of locally constant functions with values in $\ZZ/2\ZZ = \{0,1\}$. By \cite[III, (4.1.2)]{K}, there exists a natural $R$--group homomorphism
\begin{equation}\label{qua-aut-1}
    \psi \co \ZZ/2\ZZ_R \longto \uAut_R(S),
    \end{equation}
    which is an isomorphism if $S$ is \'etale \cite[III, (4.1.2)]{K}.
\sm

\item\label{qua-ex} ({\em Example}) Let $S=S_0 \oplus S_1\in \Ralg$ be a $\ZZ/2\ZZ$--graded $R$--algebra. We denote by $\Aut(S, S_1)$ the group of $R$--linear automorphisms of the $R$--algebra $S$, i.e., the automorphisms $\al$ of $S$ satisfying $\al(S_1) = S_1$,  and by
\begin{equation}\label{qua-ex0} \uAut(S, S_1)
\end{equation}
the $R$--group scheme representing the $R$--functor $T \mapsto \Aut(S_T, (S_1)_T)$.

Suppose now that $S=S_0 \oplus S_1$ satisfies \end{inparaenum}
\begin{enumerate}[label={\rm (\roman*)}]
 \item\label{qua-exi} $R \simlgr S_0$, $r \mapsto r1_S$, and
 \item\label{qua-exii} $\theta \co S_1 \ot_R S_1 \simlgr S_0$, $s_1 \ot s_1' \mapsto s_1 s_1'$ (isomorphism of $R$--modules).
\end{enumerate}
Thus $S$ is a quadratic $R$--algebra, $(S_1, \theta)$ is a discriminant module in the sense of \cite[III, \S3]{K}, and its standard involution $\si_S$ is the grading automorphism, \begin{equation}\label{qua-ex2}
 \si_S(s_i ) = (-1)^i s_i, \quad s_i \in S_i, i=0,1,
\end{equation}
which follows from the free case by localization.
We have isomorphisms of $R$--group schemes
\begin{equation}\label{qua-ex3}
 \uAut(S, S_1) \simlgr \uAut(S_1, \theta)  \simla \bmu_{2,R},
 \end{equation}
where the first (obvious) isomorphism is obtained by restriction and where the second isomorphism is $x \mapsto x \Id_{S_1}$ \cite[III, (3.2.1)]{K}.

\subsection{Discriminant algebras} \label{qfdi}
Let $(M,q)$ be a faithful quadratic $R$--space and let $\Cli(M,q)= \Cli(q)$ be its Clifford algebra, \cite[IV, \S1]{K}. It is a $(\ZZ/2\ZZ)$--graded $R$--algebra: $\Cli(q) = \Cli_0(q) \oplus \Cli_1(q)$. The {\em discriminant algebra $\Dis(q)$ of $(M,q)$\/} is the subalgebra of $\Cli(q)$ centralizing $\Cli_0(q)$:
\[ \Dis(q) = \Cli(q)^{\Cli_0(q)}. \]
Some facts that we will use (see \cite[IV, \S4]{K}). \sm

\begin{inparaenum}[(a)]
\item \label{qfdi-a} Discriminant algebras respect base change and direct products of base rings. \sm

\item $\calD := \Dis(q)$ inherits the $\ZZ/2\ZZ$--grading of $\Cli(q)$,
\[ \calD = \calD_0 \oplus \calD_1, \quad
       \calD_j =  \calD \, \cap \, \Cli_j(q), \; j=0,1.
\]
It is a quadratic $R$--algebra in the sense of \ref{qua}.
\sm

\item ({\em The group homomorphism $\Dis$}) By the universal property of the Clifford algebra $\Cli(q)$, every $g\in \orth(q)$ induces an automorphism $\Cli(g)$ of the algebra $\Cli(q)$ stabilizing $\Cli_0(q)$ and $\Cli_1(q)$, and hence an automorphism of the $\ZZ/2\ZZ$--graded algebra $\calD = \calD_0 \oplus \calD_1$. Thus, we get a homomorphism of groups,
\begin{equation}\label{sog-dishom}
 \Dis \co \orth(q) \longto \Aut (\calD, \calD_1), \quad g \mapsto
    \Cli(g)|_{\calD} =:\Dis(g).
\end{equation}

\item \label{discralg-even} ({\em Even rank}) Assume $M$ has constant even rank. Thus $q$ is regular by
\ref{qfnsp}\ref{qfnsp-vi}. In this case $\calD$ is the centre of $\Cli_0(q)$, in particular $\calD_1 = 0$, and $\calD$ a quadratic \'etale  $R$--algebra. Hence, by \eqref{qua-aut-1}, 
\begin{equation}\label{evdi1}
 \uAut(\calD, \calD_1) = \uAut(\calD) \simla (\ZZ/2\ZZ)_R.
\end{equation}
The standard involution $\si_{\calD}$ of the quadratic $R$--algebra $\calD$ is (of course) the standard involution of the quadratic \'etale $R$--algebra $\calD$.
\sm

\item \label{discralg-d} ({\em Odd rank}) Assume $M$ has constant odd rank. Then the   discriminant algebra $\calD $ is the centre of $\Cli(q)$: $\calD = \rmZ(\Cli(q))$. Moreover, $\calD = \calD_0 \oplus \calD_1$ is a $\ZZ/2\ZZ$--graded $R$--algebra satisfying the conditions \ref{qua-exi} and \ref{qua-exii} of \ref{qua}\eqref{qua-ex}. Thus
\begin{equation}\label{discralg-d1}
 \uAut(\calD, \calD_1) \cong \bmu_{2,R}.
\end{equation} \end{inparaenum}

\begin{lem}[Realizing the standard involution of $\Dis(q)$]\label{refso}
Let $(M,q)$ be a faithful quadratic space, let $x\in M $ with $q(x) \in R\ti$, and let $\rho_x$ be the associated reflection, given by
$ \rho_x(m) = m - b_q(m,x) q(x)\me x$ for $m\in M$.
Then the automorphism $\Dis(\rho_x) \in \Aut(\calD, \calD_1)$, $\calD=\Dis(q)$, is the standard involution of\/ $\calD$:
\begin{equation}\label{refso1}
 \Dis(\rho_x) = \si_{\calD}.
 \end{equation}
\end{lem}

\begin{proof}
 The element $x\in M \subset \Cli_1(q)$ is invertible in $\Cli(q)$ with inverse $x\me = q(x)\me x$. We will use the well-known formula relating $\rho_x(m)$, $m\in M$, with the inner automorphism of $\Cli(q)$ induced by $x$:
\begin{equation}\label{reso0}
 \rho_x(m) = - xmx\me
\end{equation}
which follows from $ xmx\me = (xm)(xq(x)\me) = (-mx + b_q(m,x))(xq(x)\me)
= - m + b_q(m,x) q(x)\me x = - \si_x(m)$.
It implies
\begin{equation*}\label{reso00}
 \Dis(\rho_x)(c_j) = (-1)^j xc_jx\me, \qquad (c_i \in \Cli_j(q), \, j = 0,1).
\end{equation*}
For the proof of \eqref{refso1}, we can without loss of generality assume that $M$ has constant rank.

Suppose $M$ has constant even rank. By \cite[IV, (4.3.1.4)]{K}, $\si_{\calD}(d) x = x d$ holds for $d\in \calD$. Since $\calD \subset \Cli_0(q)$ we get $\Dis(\rho_x)(d) = xdx\me = \si_{\calD}(d)$.

Suppose $M$ has constant odd rank. Since then $\calD = \rmZ(\Cli(q))$, we obtain for $d_j \in \calD_j$, $j=0,1$,  that $\Dis(\rho_x)(d_j) = (-1)^j xd_j x\me = (-1)^j d_j x x\me = (-1)^j d_j = \si_{\calD}(d_j)$ by \ref{qua}\eqref{qua-ex}.
\end{proof}

\begin{lem}\label{LGqdi}  Let $(M,q)$ be a quadratic space over an LG-ring $R$. \sm

\begin{inparaenum}[\rm (a)] \item\label{LGqdi-a}
The following are equivalent: \begin{inparaenum}[\rm (i)] 

\quad \item\label{LGqdi-i} $q(M) \cap R\ti \ne \emptyset$;


\quad \item \label{LGqdi-iii} $M$ is faithfully projective.
\end{inparaenum} \sm

\item \label{LGqdi-b} If $M$ is faithfully projective, the group homomorphism
\[ \Dis\co \orth(q) \to \Aut(\calD, \calD_1)\]
of \eqref{sog-dishom}    is surjective.
\end{inparaenum} \end{lem}

\begin{proof} \eqref{LGqdi-a}  The equivalence \eqref{LGqdi-i} $\iff$ \eqref{LGqdi-iii}  is proven in \cite[Lem.~11.26]{PRbook} for a regular $q$. The proof in the nonsingular case is essentially the same. We reproduce it here for the convenience of the reader.

Assuming \eqref{LGqdi-i}, it follows that
$q(M) \not\subset \m$ for every maximal ideal $\m \ideal R$. In particular, $M_{R_\gm}\ne 0$, which, by \eqref{fapmod}, implies \eqref{LGqdi-iii}.

Conversely, suppose that $M$ is faithfully projective. We can view $q$ as a quadratic polynomial on the affine scheme $\uW(M)$ and define $U=\uW(M)_q$, the principal open subscheme determined by $q$. Let $\gm$ be a maximal ideal of $R$, put $k=R/\gm$ and observe that $M_k \ne 0$ by \ref{fapmod}, while $\rad(q_k)=0$ by nonsingularity of $q$, in particular $q_k \ne 0$. Hence $U(k) = \{ m\in M_k: q_k(m) \ne 0\} \ne \emptyset$, and so $U(R) \ne \emptyset$ by \ref{prop_baire}\ref{prop_baire-a}, i.e., \eqref{LGqdi-i} holds. \sm

\eqref{LGqdi-b} Applying the rank decomposition of quadratic modules  \ref{qf}\eqref{qf-redc} and \ref{qfdi}\eqref{qfdi-a}, we see that we can assume that $M$ has constant rank $r$.

Assume that $r$ is odd. By \ref{qfdi}\eqref{discralg-d} and \eqref{qua-ex3}, an automorphism $g\in \Aut(\calD, \calD_1)$ has the form $g|_{\calD_0} = \Id$ and $g|_{\calD_1} = x \Id_{\calD_1}$ for some $\det(g) = x\in \mu_2(R)$. Observe $x\Id_M\in \orth(q)$. By construction of $\Cli(q)$, the orthogonal map $x\Id_M$ acts
on $\Cli_0(q)$ as identity, and on $\Cli_1(q)$ as $x\Id$. Thus $\Dis(x\Id_M) = g$.%

Assume that $r$ is even. By \ref{qfdi}\eqref{discralg-even} and the proof of \cite[III, (4.1.2)]{K}, to every $g\in \Aut(\calD, \calD_1) = \Aut(\calD)$ one can associate a unique complete system $(\veps_0, \veps_1= 1-\veps_0)$ of orthogonal idempotents $\veps_i$, such that $g$ stabilizes $\veps_i \calD$ and satisfies $g|_{\veps_0 \calD} = \Id$, $g|_{\veps_1\calD}$ is the standard involution of the quadratic \'etale $\veps_1R$--algebra $\veps_1\calD$. By \ref{LG-def}\eqref{LG-defa}, $\veps_1R$ is an LG--ring. Hence, by \ref{qf}\eqref{qf-redc}, we can assume that $R=\veps_1R$, and have to show that there exists $f\in \orth(q)$ such that $\Dis(f) = \si_{\calD}$. Since $q(M) \cap R\ti \ne \emptyset$ by  \eqref{LGqdi-a}, this follows from \eqref{refso1}.  \end{proof}

\subsection{Orthogonal group schemes}\label{ogs} 
Let $(M,q)$ be a quadratic form over $R$. The $R$--group functor $\underline{\orth(q)}$, assigning to $T\in \Ralg$ the group $  \underline{\orth(q)}(T) = \orth(q_T)$ is represented by
an affine finitely presented $R$--group scheme $\uO(q)$, \cite[Def.~4.1.0.2]{CF} or \cite[page 364]{Co1}. Properties that we will use: \sm

\begin{inparaenum}[(a)]\item \label{ogs-dec} ({\em Direct products}) Let $R= R_0 \times \cdots \times R_n$ be a direct product. Orthogonal groups respect the decomposition of \ref{qf}\eqref{qf-redc} as follows:
\begin{align*}
 &\orth(q) = \orth(q_0) \times \cdots \times \orth(q_n), \\
&  \uO(M,q) \cong p_{0*}\big(\uO(M_0, q_0)\big) \times \cdots \times
             p_{n*}\big( \uO(M_n, q_n)\big) \\
 & \qquad (\text{isomorphism of $R$--group schemes})
\end{align*}
where $p_i \co \Spec(R_i) \to \Spec(R)$ is the morphism associated with the canonical projection $R\to R_i$. \sm

\item \label{ogs-ns} $\uO(q)$ is in general not reductive. For example (\cite[Thm.~C.1.5]{Co1}),  if $(M, q)$ has constant positive rank $n$, then $\uO(q)$
    is smooth if and only if either $n$ is even or $n$ is odd and $2\in R\ti$.\sm

\item \label{ogs-dis} The homomorphism $\Dis$ of \eqref{sog-dishom} respects base change and thus defines a homomorphism of $R$--group schemes
\begin{equation}\label{qfdi1}
 \Dis \co \uO(q) \to \uAut (\calD, \calD_1), \quad \calD = \Dis(q),
\end{equation}
assigning to $T\in \Ralg$ the map $\Dis(T) \co \orth(q_T) \to \Aut\big(\calD_T, \calD_{1,T})$.%
\end{inparaenum}

\begin{lem}[Cohomology]\label{ogs-co} Fix $n\in \NN_+$. The map
\[ [(M,q)] \mapsto [\uIsom\big( (R^n, q_{0,n}), \, (M,q)\big)] \]
is a bijection between the set of isometry classes of quadratic spaces over $R$ of rank $n$ and $\rmH^1\big(R, \uO(q_{0,n})\big)$.
Equivalently, for every quadratic space $(M,q)$ of rank $n$, the set of twisted forms of $(M,q)$ is in bijection with $\rmH^1\big(R, \uO(M,q)\big)$.
\end{lem}

\begin{proof} This can be proven in the same way as \ref{ung}\eqref{ung-e}, replacing the reference \ref{dia}\eqref{dia-b} used there by \ref{qfnsp}. \end{proof}
\sm

It may be appropriate to remind the reader that we are using fppf-cohomology which does in general not coincide with \'etale cohomology, see \ref{ogs}\eqref{ogs-ns}.
In this respect, the description of the Galois cohomology set $\rmH^1_{\rm Gal}(R, \uO(q))$ in characteristic $2$ (\cite[p.~408]{KMRT}) is instructive.

\subsection{Special orthogonal group scheme $\uSO(q)$} \label{sogsc}
Let $(M,q)$ be a faithful quadratic $R$--space, and let $\calD = \Dis(q)$ be its discriminant algebra. We define the $R$--group scheme as the kernel of the homomorphism $\Dis$ of \eqref{qfdi1}:
\begin{equation} \label{sogsc0} \uSO(q) = \Ker( \Dis). \end{equation}
Thus, for $T\in \Ralg$ we have
$\uSO(q)(T) = \{ g\in \orth(q_T) : \Dis(g) = \Id_{\Dis(q_T)}\}$.
Since the inclusion $\inc\co \uAut(\calD, \calD_1)\to \uAut(\calD)$ is a monomorphism, this definition coincides with the one of \cite[IV, (5.1)]{K}, where $\uSO(q)$ is defined as the kernel of $\inc \circ \Dis$. The discussion below shows that it also agrees with the definition used in \cite[App.~C]{Co1}. However, it coincides with the definition of $\uSO(q)$ in \cite[\S4.3]{CF} only in the even rank case.
Following is a list of properties of $\uSO(q)$ that we will use. \sm

\begin{inparaenum}[(a)]
  \item \label{sogsc-dir} ({\em Direct products}) Let $R= R_0 \times \cdots \times R_n$ be a direct product of rings, and let $(M,q) = (M_0, q_0) \times \cdots \times (M_n, q_n)$
be the corresponding decomposition into a direct product of quadratic $R_i$--modules, \ref{qf}\eqref{qf-redc}.
Analogously to \ref{ogs}\eqref{ogs-dec}, we have
\begin{equation}\label{sogsc11}
 \uSO(M,q) \cong p_{1*}\big(\uSO(M_1, q_1)\big) \times \cdots \times
             p_{n*}\big( \uSO(M_n, q_n)\big)
\end{equation}
into a direct product of $R$--group schemes. It will follow from this and the discussion below that, in general, $\uSO(q)$ is a reductive $R$--group scheme which is semisimple if $\rank M \ge 3$. \sm

\item ({\em $\uSO$ and determinants}) \label{sogsc-det}
By \cite[IV, (5.1.1)]{K} we have a group homomorphism
\begin{equation} \label{dh1}
   \det \co \uO(q) \longto \bmu_{2, R}, \quad g \mapsto \det(g).
\end{equation}
We always have $\uSO(q) \subset \Ker(\det)$; equality holds if $M$ has constant odd rank or if $2\in R\ti$.  This is proven in \cite[IV, (5.1.1)]{K}, but note the misprint in (3) of loc.\ cit., where ``$\subset$'' should be ``$=$", as one can see from the proof. \sm

\item ({\em Odd rank})\label{sogsc-odd}  Let $(M,q)$ be a quadratic space of odd rank. Then the morphism $ z_M \co \bmu_{2,R} \longto \uO(q)$, $x \mapsto x\Id_M$, is a section of $\det$. Hence, by \eqref{sogsc-det},
\begin{equation}
  \label{orthsc-2} \uO(q) \cong \bmu_2 \times_R \uSO(q).
\end{equation}
If $M$ has constant rank $1$, then $\uSO(q) = \{\star \}$, and if $M$ has constant odd rank $2n+1\ge 3$, then $\uSO(q)$ is an adjoint semisimple $R$--group scheme of type ${\rm B}_n$ (${\rm B}_1 = \rmA_1$ for $n=1$) \cite[Prop.~C.3.10]{Co1}. \sm

\item({\em Even rank}) \label{sogsc-even}  Let $(M,q)$ be a quadratic space of even rank $2n\ge 2$. Then $q$ is regular by \ref{qfnsp}, and the following hold.
\end{inparaenum}
\begin{enumerate}[label={\rm (\roman*)}]

\item Since the map $\psi$ of \eqref{qua-aut-1} is an isomorphism, the {\em Dickson map}
\[ \Di : = \psi\me \circ \Dis \co \uO(q) \to \ZZ/2\ZZ_R \]
is a group homomorphism satisfying $\uSO(q) = \Ker(\Di)$. Moreover, by \cite[C.2.8]{Co1} or \cite[IV, (5.2.2)]{K}, the sequence
\begin{equation}  \label{sogsc-even1}
 1 \longto \uSO(q) \longto \uO(q) \xrightarrow{\Di} \ZZ/2\ZZ_R \longto 1
\end{equation}
is exact. \sm

 \item \label{sogsc-even-ii} If $M$ has constant rank $2$,  then $\uSO(q)$ is a rank one torus.  Indeed, by descent, we are reduced to the hyperbolic case $q=xy$  where $\uSO(q)=\GG_m$ \cite[IV, 5.1.3 and V, (2.6.3)]{K}.
\sm

\item \label{sogsc-even-iii}  If $M$ has constant rank $2n\ge 4$, then $\uSO(q)$ is a semisimple $R$--group scheme of type $\rmD_n$ \cite[C.3.10]{Co1}.
\end{enumerate}
\sm

\begin{inparaenum}[(a)] \setcounter{enumi}{4}
\item\label{sogqs-hyp} Suppose $M = U \oplus V$ such that $q(U) = 0 = q(V)$ and $(M,q) \simlgr \HH(U)$ under an isometry identifying $V$ with $U^*$. Then, associating with $g\in \GL(U)$ the map $M \to M$, $(u,v) \mapsto (g u, g\me v)$, gives rise to a closed embedding of $R$--group schemes
\begin{equation} \label{sogqs-hyp1}
  \uGL(U) \longto \uSO(q).
\end{equation}
\end{inparaenum}

\begin{prop}[Parabolic and Levi subgroups of $\uSO(q)$] \label{sopl}
Let $(M, q)$ be a quadratic space, let $U$ and $V$ be faithfully projective submodules and let $M'\subset M$ be a submodule such that
\[ M = (U \oplus V) \perp M', \quad q(U)=0 = q(V), \quad
    (U\oplus V, q|_{U\oplus V})\cong \HH(U).\]
\begin{inparaenum}[\rm (a)]
  \item The subgroup scheme $P$ of $\uSO(q)$, which stabilizes $U$, is a parabolic subgroup of\/ $\uSO(q)$. \sm

 \item Moreover, $\uGL(U) \times \uSO(q|_{M'})$ is a Levi subgroup of $P$, where the first factor embeds into $P$ using the embedding  $\uGL(U) \to \uSO(q|_{U\oplus V}) \to \uSO(q)$ of \eqref{sogqs-hyp1}.
\end{inparaenum}
\end{prop}

\begin{proof}
  This can be shown in the same way as parts \eqref{dec-b} and \eqref{dec-c} of \ref{dec}, using the same cocharacter.
\end{proof}

\end{document}